\documentclass{amsart}

\usepackage{graphicx}              
\usepackage{amsmath}               
\usepackage{amsfonts}              
\usepackage{amsthm}                
\usepackage{amssymb}
\usepackage{amscd}
\usepackage{listings}
\usepackage{times}
\usepackage{comment}
\usepackage{mathtools}
\usepackage{tikz}
\usetikzlibrary{arrows, matrix}
\usepackage{tikz-cd}

\usepackage{stmaryrd}
\usepackage{xcolor}

\usepackage[ordering=Kac,
edge/.style=thick,
indefinite-edge={draw=black,fill=white, dotted},
indefinite-edge-ratio=1.5,
mark=o,
root-radius=.1cm,
o/.style={thick}]{dynkin-diagrams}
\pgfkeys{/Dynkin diagram,
edge length=.9cm,
fold radius=.6cm,
indefinite edge/.style={
draw=black,
fill=white,
thick, dotted}}

\usepackage[colorinlistoftodos,textsize=tiny]{todonotes}
\usepackage{makecell}
\usepackage{array} 
\usepackage{relsize}

\newcommand{\fltns}{{\mathbin{\mkern-6mu\fatslash}}}

\theoremstyle{definition}
\newtheorem{thm}{Theorem}[section]
\newtheorem{lem}[thm]{Lemma}
\newtheorem{prop}[thm]{Proposition}
\newtheorem{cor}[thm]{Corollary}
\newtheorem{conj}[thm]{Conjecture}
\newtheorem{ex}[thm]{Example}

\newtheorem{defn}[thm]{Definition}

\newtheorem{definition}[thm]{Definition}

\newtheorem{rem}[thm]{Remark}

\newcommand{\RR}{\mathbf{R}}      
\newcommand{\ZZ}{\mathbf{Z}}      
\newcommand\CC{{\mathbf C}}
\newcommand{\Gm}{\mathbf{G}_m}

\newcommand{\R}{\mathbf{R}}

\newcommand{\C}{\mathbf{C}}

\newcommand{\g}{\mathfrak{g}}

\newcommand{\spec}{\mathrm{Spec}\, } 
 
\newcommand{\ad}{\mathrm{ad}}

\newcommand{\acts}{\curvearrowright}

\newcommand{\Ggr}{\mathbf{G}_{\mathrm{gr}}}

\newcommand{\Pic}{\mathrm{Pic}}

\renewcommand{\bar}[1]{\overline{#1}}
\renewcommand{\emptyset}{\varnothing}
\renewcommand{\epsilon}{\varepsilon}

\renewcommand{\tilde}[1]{\widetilde{#1}}

\DeclarePairedDelimiter{\set}{\{}{\}}

\newcommand{\calB}{\mathcal{B}}

\newcommand{\calE}{\mathcal{E}}

\newcommand{\calO}{\mathcal{O}}

\newcommand{\frakc}{\mathfrak{c}}

\newcommand{\frakg}{\mathfrak{g}}
\newcommand{\frakh}{\mathfrak{h}}

\newcommand{\frakm}{\mathfrak{m}}

\newcommand{\frakp}{\mathfrak{p}}

\newcommand{\Lie}{\mathrm{Lie}}

\newcommand{\PSL}{\mathrm{PSL}}

\newcommand{\U}{\mathrm{U}}

\DeclareMathOperator{\Bun}{Bun}

\usepackage[backref,hidelinks]{hyperref}
\hypersetup{
    colorlinks,
    citecolor=blue,
    filecolor=blue,
    linkcolor=blue,
    urlcolor=blue
}

\thanks{E.C. was partially supported by the Swiss National Science Foundation No. 196960 and the JSPS Postdoctoral Fellowship during the completion of this project. E.H. was funded by the Deutsche 
Forschungsgemeinschaft (DFG, German Research
Foundation) – 541679129. M.Y. was supported by the World Premier International Research Center Initiative (WPI), MEXT, Japan. }

\begin{document}

\author{Eric Y. Chen}
\address{E. Y. Chen, \newline\indent \'Ecole Polytechnique F\'ed\'erale de Lausanne, 
\newline\indent CH-1015 Lausanne, Switzerland.}
\email{eric.chen@epfl.ch}

\author{Enya Hsiao}
\address{E. Hsiao, \newline \indent Max-Planck-Institute Mathematics in the Sciences \newline\indent 04103 Leipzig, Germany}
\email{enya.hsiao@mis.mpg.de}

\author{Mengxue Yang}
\address{M. Yang, \newline\indent Kavli IPMU (WPI), UTIAS, The University of Tokyo, 
\newline\indent Kashiwa, Chiba 277-8583, Japan.}
\email{mengxue.yang@ipmu.jp}

\title{(BAA)-branes from higher Teichmüller theory}
\begin{abstract}
Interpreting certain holomorphic Lagrangians that arise from the relative Langlands program, we construct moduli stacks underlying the generalized Slodowy categories of Collier--Sanders and $G^\RR$-Higgs bundles over a Riemann surface. Furthermore, we extend the Cayley correspondence of Bradlow--Collier--García-Prada--Gothen--Oliveira to a morphism of Lagrangians over the Hitchin moduli stack, and initiate the study of its hyperholomorphic mirror partner under $S$-duality. 
\end{abstract}
\maketitle

\setcounter{tocdepth}{1}
\tableofcontents


\section{Introduction}

\subsection{Boundary conditions and relative Langlands} \label{subsect intro to BC}

At a physical level of rigor, one can regard various mathematical structures organized under the umbrella term ``Langlands correspondence" as the study of a certain 4d $\mathcal{N} = 4$ quantum field theory (QFT) under $S$-duality \cite{Kapustin-Witten} \cite{Gaiotto-Witten}. This 4d QFT has various supersymmetric twists under which the theory becomes more topological and hence more amenable to study. 

Most relevant to the present article is the $A$-twist at the Dolbeault point, at which $S$-duality can be interpreted as a mirror symmetry between the Hitchin moduli stack of a reductive group $G$ and its Langlands dual group \cite{Donagi-Pantev} \cite{Hausel-Thaddeus}. For the purposes of this introductory section, we use $A_{\mathrm{Dol}}$ to denote this 4d TQFT.\footnote{For more details on this perspective, we refer the reader to the embedded Remarks \ref{rem GL nonsense 1}, \ref{rem AKSZ} and \ref{rem GL nonsense 2}; a parallel discussion for the $B$-twist can be found in \cite{CF}.}

One way to formalize the above discussion is through the language of functorial field theory: the data of $A_{\mathrm{Dol}}$ can theoretically be encapsulated in a 4-category, the so-called \textit{category of boundary conditions} $\mathfrak{B}_{\mathrm{GL}}$, and $A_{\mathrm{Dol}}$ can be understood as a functor on $\mathfrak{B}_{\mathrm{GL}}$ with values in various $\CC$-linear categories. With this data, one naively expects to be able to label any closed 4-dimensional manifold with objects of $\mathfrak{B}_{\mathrm{GL}}$ and evaluate the theory to obtain a numerical invariant of the 4-manifold. 

More generally, suppose we have a (not necessarily closed) $(d \leq 4)$-manifold with defect, i.e. we have $d$-dimensional bulk regions separated by walls and corners of various codimensions $0 \leq k \leq 4$. In order to obtain a $(3-d)$-categorical invariant\footnote{We regard a $(-1)$-categorical invariant as a numerical invariant, and a $0$-categorical invariant as a chain complex.} from $A_{\mathrm{Dol}}$ on this picture, one should label the codimension $k$ pieces by $k$-morphisms in $\mathfrak{B}_{\mathrm{GL}}$. Such choices of labels for the defects of our $d$-manifold are termed choices of boundary conditions, and those that survive the supersymmetric twists most relevant for us were first observed in \cite{Gaiotto-Witten} as arising from Hamiltonian actions. 

Mathematically, the emergent relative Langlands program \cite{BZSV} initiated by Ben-Zvi--Sakellaridis--Venkatesh (BZSV in the following) studies a distinguished class of boundary conditions labeled by \textit{hyperspherical} Hamiltonian actions, which are conjectured to be well-behaved under $S$-duality. On the $A$-side, they study specializations of these boundary conditions at the de Rham and Betti points; the resulting objects are termed \textit{period sheaves}, in analogy with automorphic periods which play a central role in the arithmetic Langlands program \cite{Sakellaridis-Venkatesh} \cite{Jacquet-Friedberg} \cite{Jacquet-Shalika}. On the other hand, for the Dolbeault point on the $A$-side it has been understood from the work of \cite{Gaiotto} and \cite{Ginzburg-Rozenblyum} that certain holomorphic Lagrangians (\textit{BAA-branes}) over the Hitchin moduli space are the evaluations of $A_{\mathrm{Dol}}$ on codimension 1 labels. The relationship between BZSV's period sheaves and (BAA)-branes can be formalized by the following slogan:
\begin{center}
    the microlocal support of period sheaves are (BAA)-branes,
\end{center}
and it is the latter that play a central role in our current discussion.

Our first results (Theorem \ref{thm CS moduli stack} and Theorem \ref{thm GR Higgs bundles}) give moduli interpretions to certain (BAA)-branes arising from microlocal supports of BZSV's period sheaves: we observe that Collier--Sanders' global analogue of the Slodowy slice \cite{Collier-Sanders}, and the $G^\RR$-Higgs bundles studied extensively by \cite{Schaposnik2013}\cite{Schaposnik2014}\cite{GPPN} are examples of boundary conditions in $A_{\mathrm{Dol}}$ encoded by specific Hamiltonian actions. In particular, these moduli problems have natural presentations as (derived) holomorphic Lagrangians over the Hitchin moduli stack, and we recover the categories of the previously mentioned works by passing to $\CC$-points.

\subsection{Higher Teichmüller theory and 2-morphisms}

In recent work of Bradlow--Collier--García-Prada--Gothen--Oliveira \cite{Cayley}, a general Cayley correspondence is proposed: roughly speaking, one constructs Higgs bundles for a reductive group $G$ from the data of Higgs bundles for a Cayley partner subgroup along with various global holomorphic poly-differentials. The construction depends on the data of a \textit{magical $\mathfrak{sl}_2$-triple} and, under certain conditions on the starting data, the resulting flat $G$-connection obtained by composing with the nonabelian Hodge correspondence has monodromy lying in a real reductive subgroup $G^\RR$ of $G$. 

Interestingly, the image of the Cayley correspondence labels special components in the Higgs moduli space that are of topological significance in the real character variety. The first discovery in this direction was that of the Hitchin component, which was shown independently by Fock--Goncharov \cite{FockGoncharov2006} and Labourie \cite{Labourie2006} to consist entirely of discrete and faithful surface group representations. Such connected components of the $G^\R$-character variety were later termed \textit{higher Teichm\"uller components}, generalizing the usual notion of Teichmüller space when $G^\RR = \mathrm{PSL}_2(\RR)$. 

A simple real Lie group $G^\R$ admits a \textit{magical $\mathfrak{sl}_2$-triple} if a certain vector space involution on its complexified Lie algebra $\mathfrak{g}$ is a Lie algebra involution, a condition generalizing Hitchin's use of principal triples in the construction of the Hitchin components. In order to describe the shape of the Cayley correspondence of \cite{Cayley} we introduce temporarily the following notation: for $H^\RR$ a real form of a complex reductive group $H$, we write $\mathcal{M}_{L}(H^\RR)$ to denote the moduli space of stable $L$-twisted $H^\RR$-Higgs bundles for a line bundle $L$, where we omit the subscript when $L=K$. Associated to a magical $\mathfrak{sl}_2$-triple $\rho: \mathfrak{sl}_2 \to \mathfrak{g}$ is an injective, open and closed map 
\begin{align*}
        \Psi_\rho:\, \mathcal{M}_{K^{m_c+1}}(\tilde{G}^\R)\times\prod_{j=1}^{r(\rho)}\mathcal{M}_{K^{l_j+1}}(\R^+) \,\,\longrightarrow \,\,\mathcal{M}(G^\R),
    \end{align*}
    where $\tilde{G}^\R$ is a semisimple subgroup of $G^\R$, $r(\rho)$ is the rank of a certain subalgebra of $\mathfrak{g}$, and the numbers $m_c$ and $l_j$ are weights of the  $\rho(h)$-weight decomposition of $\mathfrak{g}$, all of which are Lie theoretic data associated to the magical triple (see \cite[Lemma~5.7]{Cayley}). We refer to the connected components in the image $\Psi_\rho$ as \textit{Cayley components}, and it has been shown that under the nonabelian Hodge correspondence, Cayley components are higher Teichm\"uller components in the $G^\R$-character variety \cite{GuichardEtAl2021}. Furthermore, it is conjectured that all higher Teichm\"uller components are Cayley components, apart from those parametrizing maximal representations of nontube type Hermitian groups. Evidence of the relation between $\Psi_\rho$ and Teichmüller components was suggested by the classification result \cite[Theorem C]{Cayley}, which showed that the list of real simple Lie groups admitting magical $\mathfrak{sl}_2$-triples coincides with those admitting $\Theta$-positive structures: 

    \begin{thm}[\cite{Cayley}, Theorems 3.1 and 8.14]
\label{thm:classification-canonical-real-forms}
      A simple real Lie group $G^\R$ admits a magical $\mathfrak{sl}_2$-triple if and only if the pair $(\mathfrak{g}^\R, \mathfrak{g})$ of its Lie algebra and its complexification belongs to one of the following four families: 
      \begin{enumerate}
      \item [(1)] $\mathfrak{g}$ is split real and $\mathfrak{g}^\R$ is the split real form,
      \item [(2)] $\mathfrak{g}$ is of type $A_{2n-1}, B_n, C_n, D_n, D_{2n}$ or $E_7$ and $\mathfrak{g}^\R$ is Hermitian of tube type,
      \item [(3)] $\mathfrak{g}=\mathfrak{so}_N\C$ and $\mathfrak{g}^\R=\mathfrak{so}(p,N-p)$, for $p\ge 3$,
      \item [(4)] $\mathfrak{g}$ is exceptional and $\mathfrak{g}^\R$ is one of the following:
        $$\begin{array}{c| c| c| c| c}
          \mathfrak{g} & E_6 & E_7& E_8 &F_4  \\
          \hline
          \mathfrak{g}^\R & \mathfrak{e}_{6(2)} & \mathfrak{e}_{7(-5)}& \mathfrak{e}_{8(-24)}& \mathfrak{f}_{4(4)}
        \end{array}$$
      \end{enumerate}
    \end{thm}

In the present article, we provide a new conceptual understanding of the morphism $\Psi_\rho$ of \textit{op. cit} in the Higgs setting: we observe that the Cayley correspondence can be regarded as an example of 2-morphisms in $A_{\mathrm{Dol}}$, i.e., a morphism of (BAA)-branes over the moduli stack of $G$-Higgs bundles (see Theorem \ref{thm main 1}). With this interpretation, we deduce geometric properties of the Cayley correspondence by analyzing the underlying Hamiltonian actions (Theorems \ref{thm main 2} and \ref{thm main 3}), and we extend the geometric properties established in Section 7 of \textit{loc. cit} to their natural generalizations at the level of moduli stacks.

\subsection{$S$-duality and (BAA)/(BBB) mirror symmetry}

One of the principal motivations of this work was to further our understanding of $S$-duality beyond relative Langlands duality (which can be viewed as a study of dual objects, or $S$-duality of boundary conditions), to the study of \textit{dual morphisms} (informally, $S$-duality in codimension 2). We give a brief introductory account of these ideas here, deferring to Remarks \ref{rem GL nonsense 1}, \ref{rem AKSZ}, \ref{rem GL nonsense 2} and the introductory part of Section \ref{sect S-duality} for a more developed discussion. 

Our entry point involves understanding the effect of $S$-duality on the Cayley correspondence, which we regard as a morphism of (BAA)-branes on the moduli stack of $G$-Higgs bundles. More precisely, one expects that there ought to be a pair of (BBB)-branes, i.e., the Fourier--Mukai duals of the source and targets of the Cayley correspondence, and a hyperholomorphic morphism between these (BBB)-branes which is Fourier--Mukai dual to the Cayley morphism. Schematically, we may summarize the situation via the following diagram (for precise formulations, we refer to Theorems \ref{conj 1} and \ref{conj 2} and Conjecture \ref{conj 3} in the body of this text)\footnote{Note that the Cayley correspondence arrow has its direction reversed since we consider its induced morphism on the ring of functions.}: 
\begin{equation} \label{eqn intro S-dual of Cayley diagram}
\begin{tikzcd}
	{\mathcal{O}(\text{Cayley space})} & {B_1} \\
	{\mathcal{O}(G^\RR\text{-Higgs bundles})} & {B_2}
	\arrow["{S\text{-dual}}", from=1-1, to=1-2]
	\arrow["{\text{Cayley corresp.}}", from=2-1, to=1-1]
	\arrow[from=1-2, to=1-1]
	\arrow["{??}", from=2-2, to=1-2]
	\arrow["{S\text{-dual}}"', from=2-1, to=2-2]
	\arrow[from=2-2, to=2-1]
\end{tikzcd}
\end{equation} 
where $B_1, B_2$ are (BBB)-branes Fourier--Mukai dual to the Caylay partner space and $G^\RR$-Higgs bundles, respectively, and the sought after morphism is marked by $??$. For this question to be well-posed, one would like to at least have candidates for the (BBB)-branes $B_1$ and $B_2$. 

To this end, we focus on the specific case of Cayley correspondences of Hermitian tube type in type $A$, where the real form involved is $G^\RR = \mathrm{PU}(n,n)$, although we expect that a similar analysis can be carried out in other related types which we call \textit{tempered}.\footnote{There are notable difficulties, however. For instance, in the split case of type $A$, one is driven to understand the Fourier--Mukai dual of $\mathrm{GL}_n(\RR)$-Higgs bundles. The complete answer should involve metaplectic covers of the Langlands dual group, interpreted appropriately in the Dolbeault setting.} Following the philosophy of functorial field theories explained in Section \ref{subsect intro to BC}, we may regard diagram \eqref{eqn intro S-dual of Cayley diagram} as the functorial image of the following diagram of Hamiltonian actions: consider the pair of Langlands dual reductive groups $G = \mathrm{PGL}_{2n}$ and $\check{G} = \mathrm{SL}_{2n}$, and the diagram\footnote{By convention, our $S$-dualities are \textit{contravariant} on 2-morphisms. This is why the Cayley morphism and the arrow marked by $??$ go in opposite directions.}
\begin{equation} \label{eqn intro S-dual of Cayley diagram 2}
\begin{tikzcd}
	{G \acts M_1 := T^*_\psi(U\mathrm{PGL}_n^\Delta \backslash G)} & {\check{G} \acts \check{M}_1 := T^*(\mathrm{Sp}_{2n} \backslash \check{G})} \\
	{G \acts M_2 := T^*(\mathrm{P}(\mathrm{GL}_n \times \mathrm{GL}_n) \backslash G)} & {\check{G} \acts \check{M}_2 := T^*(\mathbf{A}^{2n} \times^{\mathrm{Sp}_{2n}} \check{G})}
	\arrow["{S\text{-dual}}", from=1-1, to=1-2]
	\arrow["{\text{Cayley morphism}}"', from=1-1, to=2-1]
	\arrow[from=1-2, to=1-1]
	\arrow["{??}"', from=2-2, to=1-2]
	\arrow["{S\text{-dual}}"', from=2-1, to=2-2]
	\arrow[from=2-2, to=2-1]
\end{tikzcd}
\end{equation}
The precise definitions of these actions, and importantly their structures as graded Hamiltonian spaces, will be spelled out in Section \ref{subsect two Hamiltonian actions}. The functor we apply to \eqref{eqn intro S-dual of Cayley diagram 2} in order to obtain \eqref{eqn intro S-dual of Cayley diagram} can be understood as the functor of evaluation, in the sense of functorial TQFTs, of the ``Dolbeault A/B-twist Langlands functor", and \eqref{eqn intro S-dual of Cayley diagram 2} can be viewed as living in the (4-)category of boundary conditions $\mathfrak{B}_{\mathrm{GL}}$ underlying the Langlands TQFT. 

Let us explain in some detail the characters appearing in \eqref{eqn intro S-dual of Cayley diagram 2} and their relationships. On the left hand side column we have a morphism of Hamiltonian $G$-spaces $M_1 \to M_2$, for which we demonstrate that, under the Dolbeault A-twist interpretation of these boundary conditions, realizes the Cayley correspondence for $G^\RR = \mathrm{PU}(n,n)$. The passage from $G \acts M_i$ to $\check{G} \acts \check{M}_i$ for $i = 1,2$, marked by the horizontal arrows labeled ``$S$-duality", is a hypothetical involution on our category $\mathfrak{B}_{\mathrm{GL}}$ extending Langlands' duality of reductive groups $G \leftrightarrow \check{G}$. While $S$-duality of Hamiltonian actions, i.e., 1-morphisms, does not have a definition in general, there are two overlapping proposals which apply to our case of interest:
\begin{itemize}
    \item (BZSV's hyperspherical duality). For hyperspherical Hamiltonian actions \cite[Section 3.5]{BZSV}, of which $M_1, M_2$ are examples, there is a combinatorial procedure informed by global/local harmonic analysis of automorphic forms to construct the $S$-dual. From this perspective, the duality of $G \acts M_1 \leftrightarrow \check{G} \acts \check{M}_1$ first appeared as Jacquet--Shalika's integral representation of the residue of the exterior square $L$-function \cite{Jacquet-Shalika}, while the duality $G \acts M_2 \leftrightarrow \check{G} \acts \check{M}_2$ first appeared as Jacquet--Friedberg's extension of Hecke's integral representation of the standard $L$-function \cite{Jacquet-Friedberg}. 
    \item (Nakajima's $S$-duality). In \cite{Nakajima}, a definition of $S$-duality for polarized Hamiltonian actions is given (which applies readily to $M_2$); combined with Hanany--Witten transition moves, one can significantly extend the coverage of of this method to cover $M_1$ as well.
\end{itemize}
It is a consequence of the Local Conjecture of \cite{BZSV} that these two constructions should agree, when they overlap. Both methods do indeed lead to the same $\check{G} \acts \check{M}_i$ displayed in diagram \eqref{eqn intro S-dual of Cayley diagram 2}, and it is natural to wonder if $S$-duality extends to 2-morphisms, i.e., whether the ``Cayley morphism" $M_1 \to M_2$ has an $S$-dual in terms of a morphism between $\check{M}_1$ and $\check{M}_2$. 

We propose that the answer is yes, with an expected modification: 2-morphisms between Hamiltonian actions should include objects more flexible than just equivariant morphisms, but rather equivariant Lagrangian correspondences. Searching among these generalized morphisms, one notices an evident candidate: the morphism
\begin{equation}
    \mathbf{A}^{2n} \times^{\mathrm{Sp}_{2n}} \check{G} \longrightarrow \mathrm{Sp}_{2n} \backslash \check{G} 
\end{equation}
induced by projection to the zero section gives rise to a Lagrangian correspondence between their cotangent bundles
\begin{equation}
 \check{M}_2 \longleftarrow \mathcal{L} \longrightarrow \check{M}_1
\end{equation}
and our proposal can then be summarized succinctly as follows:
\begin{center}
    The $S$-dual to the Cayley morphism is $\mathcal{L}$. 
\end{center}
We obtain mathematical evidence for our proposal in the Dolbeault twist of the hypothetical Langlands TQFT, which is the content of Section \ref{sect S-duality}. More precisely, while we used the ``Dolbeault A-twist" functor $A_{\mathrm{Dol}}$ on the left hand side column of \eqref{eqn intro S-dual of Cayley diagram 2} to obtain that of \eqref{eqn intro S-dual of Cayley diagram}, under $S$-duality we will use a certain ``Dolbeault B-twist" functor $B_{\mathrm{Dol}}$ on the right hand side column of \eqref{eqn intro S-dual of Cayley diagram 2} to obtain the objects $B_1, B_2$ of \eqref{eqn intro S-dual of Cayley diagram}. We defer to \cite{CF} for a proper treatment of $B_{\mathrm{Dol}}$ and focus here only on the immediately relevant aspects to the Cayley correspondence:  
\begin{itemize}
    \item (Theorem \ref{conj 1} and Hitchin's Theorem \ref{conj 2}).\footnote{The dual pair for $i = 1$ will be established as Theorem \ref{conj 1}. The dual pair for $i = 2$ was first observed by Hitchin in Section 7 of \cite{Hitchin2013}, and examined again recently by \cite{HameisterMorrissey2024} \cite{HameisterEtAl2024} with slightly different language. Following \cite{Hitchin2013} closely, we state Hitchin's result as Theorem \ref{conj 2} and give another argument based on Hecke operations which is well-adapted to our goals.}  Generically over the Hitchin base, the objects $B_i := B_{\mathrm{Dol}}(\check{M}_i)$ are Fourier--Mukai duals of the Cayley space and $\mathrm{PU}(n,n)$-Higgs bundles, respectively.
    \item The sheaves $B_1, B_2$ are hyperholomorphic in the sense proposed by \cite{CF}, following Deligne--Simpson's twistor space construction. The morphism $B_{\mathrm{Dol}}(\mathcal{L})$ is a hyperholomorphic morphism, as expected of the $S$-dual of a morphism of (BAA)-branes.
\end{itemize}

\subsection{Organization}

In Section \ref{sect BAA branes} we recall the definition of graded Hamiltonian spaces and the notion of Gaiotto's Lagrangians attached to them. We introduce some combinatorial data inspired by \cite{BZSV} to label our Hamiltonian actions of interest, and in Section \ref{sect Slodowy and homogeneous} we show that Lagrangians associated to certain Hamiltonian actions recover Collier--Sanders' Slodowy category and $G^\RR$-Higgs bundles; this gives a uniform proof of their representability by (derived) algebraic stacks and clarifies their nature as a Lagrangian over the moduli stack of Higgs bundles. In Section \ref{sect Cayley morphism} we show that the Cayley correspondence (corresponding to the data of a magical $\mathfrak{sl}_2$-triple) can be extended to a morphism of Lagrangians over the Hitchin moduli stack, induced by morphisms of Hamiltonian actions; we can retrieve the statement at the level of moduli spaces from these derived Lagrangians by analyzing their derived and stacky structures, which we explain in Appendix \ref{appendix}.  Finally, in Section \ref{sect S-duality} we study $S$-duality phenomena and produce Conjecture \ref{conj 3} as a Dolbeault-form of certain examples of $S$-duality phenomena in codimension 2. We prove our conjecture for the magical real form $\mathrm{PU}(n,n)$.

\subsection{Notations and conventions}

\subsubsection{Curves} We use $\Sigma$ to denote a smooth projective algebraic curve over $\CC$ with genus $g(\Sigma) \geq 0$. We write $K$ for its canonical line bundle, and we choose once and for all a square root $K^{1/2}$ upon which many constructions are based. We will often drop this dependency from our notation. 

The choice of $K^{1/2}$ determines, in particular, a \textit{uniformizing Higgs bundle} with structure group $\mathrm{SL}_2$: 
\begin{equation} \label{eqn uniformizing bundle}
    \Theta = \left(E_0 = K^{1/2} \oplus K^{-1/2}, \phi_0 = \begin{bmatrix} 0&0\\1&0\end{bmatrix} \in H^0(\Sigma, \mathrm{ad}(E_0)\otimes K)\right).
\end{equation}
It is well-known that this is a stable Higgs bundle whose solution to Hitchin's equations is equivalent to a Hermitian metric on the tangent bundle $K^{-1}$ compatible with the conformal structure on $\Sigma$ viewed as a Riemann surface. 

Note that given an aribtrary line bundle $L$ and a choice of square root $L^{1/2}$, the same formula as \eqref{eqn uniformizing bundle} determines an $L$-twisted $\mathrm{SL}_2$-Higgs bundle which we continue to denote as $\Theta = (E_0, \phi_0)$ in the main text. 

\subsubsection{Groups, Lie algebras, and actions}
Whenever we use capital Roman letters, e.g., $G,H,P,U$ to denote algebraic groups over $\CC$, we will use the corresponding lower case fraktur letters, e.g., $\mathfrak{g,h,p,u}$, to denote their corresponding Lie algebras. 

For a reductive Lie algebra $\g = \mathrm{Lie}(G)$, we choose a $G$-invariant Killing form to identify $\g \simeq \g^*$ throughout, which also gives an isomorphism between the adjoint with the coadjoint representation of $G$. 

If a group $G$ (resp. Lie algebra $\g$) is equipped with a distinguished real form, we indicate the corresponding real Lie group by $G^\R$ (resp. $\g^\R$). 

All group actions will right actions unless otherwise stated, \textit{including linear representations}. This will not cause too much pain, since the representations of importance that we consider are usually self-dual (e.g., the adjoint representation for a semisimple group). 

We use subscripts on the Cartesian product symbol $\times$ to mean fiber products. This is mostly used in the context of moment maps: if $\mu_i: M_i \to \g^*$ for $i = 1,2$ are two Hamiltonian $G$-actions, then $M_1 \times_{\g^*} M_2$ denotes their fiber product over their respective moment maps. On the other hand, we use superscripts on the Cartesian product symbol to mean quotienting by a diagonal action: if $X, Y$ are two $G$-spaces, we write $X \times^G Y := (X \times Y)/G$.

\subsubsection{Maps, bundles and sections} \label{subsect bundles and sections}

For $\mathfrak{X}$ a derived stack, we write $\mathbf{T}_\mathfrak{X}$ and $\mathbf{L}_{\mathfrak{X}}$ for its tangent and cotangent complexes, respectively. If a morphism $f: \mathfrak{X} \to \mathfrak{Y}$ of derived stacks has a relative cotangent complex, we denote it by $\mathbf{L}_f$.

Let $X$ be a projective variety and $Y$ an Artin stack. We denote by
$$\mathrm{Map}(X,Y)$$
the (derived) mapping stack, whose tangent complex can be computed using the diagram 
$$\begin{tikzcd}
	{X \times \mathrm{Map}(X,Y)} & Y \\
	{\mathrm{Map}(X,Y)}
	\arrow["{\mathrm{ev}}", from=1-1, to=1-2]
	\arrow["p"', from=1-1, to=2-1]
\end{tikzcd}$$
and the formula
$$\mathbf{T}_{\mathrm{Map}(X,Y)} = p_* \mathrm{ev}^*\mathbf{T}_Y.$$

Now let $E$ be a (left) principal $G$-bundle on $X$, and let $Y$ be a (right) $G$-space. Consider the Cartesian diagram
$$
    \begin{tikzcd}
	{\mathrm{Sect}(X, Y \times^G E)} & {\mathrm{Map}(X, [Y/G])} \\
	{\{E\}} & {\mathrm{Bun}_G(X)}
	\arrow[from=1-1, to=1-2]
	\arrow[from=1-1, to=2-1]
	\arrow[from=1-2, to=2-2]
	\arrow[from=2-1, to=2-2]
\end{tikzcd}$$
which defines the moduli stack of $Y$-valued sections of $E$ in the upper left corner. 

We write $Y_E = Y \times^G E$, and if $Y$ is the adjoint representation of $G$ we write $\mathrm{ad}(E) := \g \times^G E$. If $H \subset G$ is a subgroup and $F$ is a principal $H$-bundle, we write $\mathrm{Ind}_H^G(F) := F \times^H G$.  

If $G = \Gm$ we will abuse notation slightly and denote a line bundle on $X$ and the associated $\Gm$-torsor by the same letter.

For a reductive group $G$ and $L$ a line bundle on $\Sigma$, we write
$$
    \mathrm{Higgs}_G^L := \mathrm{Sect}(\Sigma, [\g_L/G])
$$
for the moduli stack of $L$-twisted $G$-Higgs bundles on $\Sigma$. When $L = K$ we will uniformly drop the dependence on $L$ from our notation and simply write $\mathrm{Higgs}_G$. Similarly, we write
$$
    \mathrm{Bun}_G := \mathrm{Map}(\Sigma, BG)
$$
for the moduli stack of $G$-bundles on $\Sigma$. When $G = \Gm$, we write
$$
    \mathrm{Pic}(\Sigma) := \mathrm{Bun}_{\Gm}(\Sigma)
$$
for the moduli stack of line bundles on $\Sigma$.

\subsubsection{Symplectic geometry}

We will employ minimally the basic concepts and language of derived symplectic geometry \cite{PTVV}. (In fact, the spaces of most importance to us will turn out to be underived Deligne--Mumford stacks in the end, but it is convenient to formulate them in the realm of derived algebraic geometry \textit{a priori}).

Suppose $(\mathfrak{M}, \omega)$ is a (0-shifted) symplectic stack; the main example we are interested in is the Hitchin moduli stack $\mathfrak{M} = \mathrm{Higgs}_G$ which is naturally the cotangent stack to $\mathrm{Bun}_G$.\footnote{When the genus of $\Sigma$ is at least 2 and $G$ is semisimple, the derived enhancement of $\mathrm{Higgs}_G$ is trivial.} A Lagrangian morphism $\mu: \mathfrak{L} \to \mathfrak{M}$ is a null-homotopy for $\mu^*\omega$ such that the induced composition from the tangent complex $\mathbf{T}_\mathfrak{L}$ of $\mathfrak{L}$ to its cotangent complex $\mathbf{L}_{\mathfrak{L}}$ 
\begin{equation}
    \mathbf{T}_{\mathfrak{L}} \overset{\omega}{\longrightarrow} \mu^*\mathbf{L}_{\mathfrak{M}} \longrightarrow \mathbf{L}_{\mathfrak{L}}
\end{equation}
is a fiber sequence, i.e., that the induced morphism $\mathbf{T}_{\mathfrak{L}} \to \mathbf{L}_\mu[-1]$ is an equivalence.

\subsubsection{Jacobson--Morozov data} \label{subsubsect JM data}

Let $\g$ be a complex reductive Lie algebra. Given the data of a Lie algebra homomorphism $\rho: \mathfrak{sl}_2 = \langle e,f,h\rangle \to \g$ one can associate the following list of useful data which we generically term \textit{Jacobson--Morozov data}. First of all, according to $\rho(h)$-weights we can decompose
$$
    \g = \bigoplus_{j \in \ZZ} \, \g_j
$$
so that $\rho(h)|_{\g_j} = j \, \mathrm{Id}$. The positive (resp. nonnegative) weight spaces assemble to form a unipotent (resp. parabolic) Lie subalgebra which we denote by $\mathfrak{u} = \mathfrak{u}_\rho$ (resp. $\mathfrak{p} = \mathfrak{p}_\rho$) where we drop the dependence on $\rho$ if the context is clear. It is convenient also to write $\mathfrak{u}_+ \subset \mathfrak{u}$ for those weight spaces of weight $\geq 2$. 

The homomorphism $\rho$ gives rise to a representation of $\mathfrak{sl}_2$ on $\g$ by post composition with the adjoint representation. We write the decomposition into $\mathfrak{sl}_2$-irreducible representations as
$$
    \g = W_0^{\oplus M} \oplus \bigoplus_{j = 1}^N \, W_{n_j}
$$
where $n_j > 0$ are positive integers, and for an integer $k \geq 0$ we write $W_k$ for the irreducible representation of dimension $k+1$ (equivalently, of highest weight $k$). We distinguish 
$$
    V_{n_j} \subset W_{n_j}
$$
the complex line of the highest weight space. Note that the direct sum of these $V_{n_j}$'s, along with the $W_0$'s, is exactly the kernel of the adjoint action of $\rho(e)$: 
\begin{equation}
    \g_e := W_0^{\oplus M} \oplus \bigoplus_{j = 1}^N \, V_{n_j} = \mathrm{ker}(\mathrm{ad}_e). 
\end{equation} 
The $0$th pieces $W_0$ can be identified with the Lie algebra of the centralizer $C_\rho$ of $\rho$:
$$
    \mathfrak{c} := \mathrm{Lie}(C_\rho) = W_0^{\oplus M} \subseteq \g
$$
and we shall use the notation $\mathfrak{c}$ instead of $W_0^{\oplus M}$ when we want to emphasize its structure as a Lie algebra. 

If all the $n_j$'s that appear are even, then $\rho$ is called an \textit{even} $\mathfrak{sl}_2$-triple, and we write $n_j = 2m_j$ in this case. Since $\g_1 = 0$, we have $\mathfrak{u} = \mathfrak{u}_+$ and several constructions built from these Jacobson--Morozov data simplify considerably. In any case, since $W_{n_j}$ can be generated by applying $\mathrm{ad}_f$ to the highest weight line $V_{n_j}$, we can also write
\begin{equation}
  \label{eq:g-highest-weight-line-decomposition}
    \g = \mathfrak{c} \oplus \bigoplus_{j=1}^N \, \bigoplus_{k=0}^{n_j} \, \mathrm{ad}_f^{k}\cdot V_{n_j}.
\end{equation}

Finally we define the \textit{Slodowy slice} attached to $\rho$ as the following affine subspace
$$
    \mathcal{S}_\rho := f + \g_e \subset \g
$$
considered as an affine subspace of $\g \simeq \g^*$. It is a transverse slice to the $U$-action on $f + \mathfrak{u}^\perp_+$.

\subsection{Acknowledgments}
This collaboration mainly took place at Kavli IPMU in Tokyo. We are grateful for their hospitality and for the calm and comfortable environment that facilitated our work. E.C. would like to thank the second and third named authors for introducing him to higher Teichmüller theory, Emilio Franco for inspiring conversations, and Hiraku Nakajima for his feedback and support along with the JSPS postdoctoral fellowship which made this collaboration possible. E.H. would like to express her thanks to Anna Wienhard for her guidance and encouragement, and to the Max-Planck Insitute in Leipzig for supporting several trips to Japan. M.Y. is grateful to Laura Schaposnik for helpful discussions.

\section{Holomorphic Lagrangians in the Hitchin moduli space} \label{sect BAA branes}

In this section we review the construction due to Gaiotto \cite{Gaiotto} (and mathematically rigorously by Ginzburg--Rozenblyum \cite{Ginzburg-Rozenblyum}) of Lagrangians over the Hitchin moduli space. The input data is an arbitrary Hamiltonian space, but we will be most interested in those Hamiltonian actions arising from the relative Langlands program \cite{BZSV}. These are in particular indexed by group-theoretic data which we term \textit{BZSV triples} (Definition \ref{defn: boundaryCondition}). 

\subsection{Hamiltonian actions}

\begin{defn}\label{defn Hamiltonian action}
    Let $G$ be a reductive algebraic group. A \textit{Hamiltonian $G$-space} is a smooth symplectic variety $M$ with $G$-action, equipped with the the choice of a $G$-equivariant moment map $\mu: M \to \g^*$.
\end{defn}
\begin{defn} \label{defn graded Hamiltonian action}
    A \textit{graded Hamiltonian $G$-space} is a Hamiltonian $G$-space equipped with the commuting action of the multiplicative group $\Gm$ which scales the symplectic form with some weight $w \in \ZZ$. The integer $w$ is referred to as the \textit{weight} of the Hamiltonian action.
\end{defn}
Since the $\Gm$ scaling our Hamiltonian actions plays a distinguished role, we will follow notation introduced by \cite{BZSV} and denote it by $\Ggr$.

\begin{defn}
    Let $(M_1,\mu_1), (M_2,\mu_2)$ be two graded Hamiltonian $G$-spaces. A morphism $\phi: M_1 \to M_2$ is a $(G \times \Ggr)$-equivariant symplectic morphism making the following moment map diagram commute
    \begin{equation} \label{eqn morphism of graded Hamiltonian action}
        \begin{tikzcd}
	{M_1} && {M_2} \\
	& {\g^*}
	\arrow["\phi", from=1-1, to=1-3]
	\arrow["{\mu_1}"', from=1-1, to=2-2]
	\arrow["{\mu_2}", from=1-3, to=2-2]
\end{tikzcd}
    \end{equation}
\end{defn}

Although the definition is given at a natural level of generality, those graded Hamiltonian $G$-spaces that we consider will always be \textit{smooth and affine with $\Ggr$-weight 2}. 

\subsection{BZSV triples} \label{subsect BZSV triple}

We are especially interested in those graded Hamiltonian spaces that play a central role in the relative Langlands program \cite{BZSV} as proposed by Ben-Zvi--Sakellaridis--Venkatesh (BZSV in the following). In this section, we introduce group-theoretic data with which we construct these Hamiltonian actions via an induction procedure. 

\begin{defn} \label{defn: boundaryCondition}
    Let $G$ be a reductive group. A \textit{BZSV triple} for $G$ is a triple $(H, S, \rho)$, where
    \begin{itemize}
        \item $H$ is a reductive group with an inclusion $H \to G$, 
        \item $S$ is a finite dimensional symplectic representation of $H$ equipped with moment map $S \to \mathfrak{h}^*$, and
        \item $\rho: \mathfrak{sl}_2 \to \mathfrak{g}$ is a Lie algebra homomorphism whose image commutes with the image of the Lie algebra $\mathrm{Im}(\mathfrak{h}) \subset \g$. 
    \end{itemize}
\end{defn} 
Note that given a BZSV triple $(H,S,\rho)$, the symplectic vector space $S$ is in particular a Hamiltonian $H$-space. Since the moment map of linear representations are quadratic functions, if we give $S$ the linear scaling action of $\Ggr$ then $H \acts S$ is equipped naturally with the structure of a graded Hamiltonian $H$-space of weight 2.

\begin{rem}
    In fact, in the first paragraph of Gaiotto--Witten's paper on $S$-duality of boundary conditions \cite{Gaiotto-Witten}, they have already identified such triples as the labeling data of boundary conditions for 4-dimensional $\mathcal{N} = 4$ Yang--Mills theory. Families of topological twists of this gauge theory are shown to be responsible for geometric Langlands duality in \textit{op. cit}.
\end{rem}
We write 
$$\mathfrak{sl}_2 = \mathrm{span}\bigg( e = \begin{bmatrix} 0 & 1\\0&0\end{bmatrix},h = \begin{bmatrix} 1 & 0\\0&-1\end{bmatrix},f = \begin{bmatrix} 0&0\\1&0\end{bmatrix}\bigg)$$
for the standard basis of $\mathfrak{sl}_2$, and we often identify these elements with their images under $\rho$ in $\g \simeq \g^*$ when the context is clear. In particular, we often view $f$ as living in $\g^*$.

Given a BZSV triple $(H, S, \rho)$ for $G$, one can construct a graded Hamiltonian $G$-space via \textit{Whittaker induction} which we now review for the reader's convenience (although we often reduce to working directly with the BZSV triple, so one may view further constructions as taking BZSV triples as input). Recall (see Section \ref{subsubsect JM data}) the Jacobson--Morozov unipotent subgroup $U \subset G$ defined by $\rho$, whose Lie algebra consists of the strictly positive $\rho(h)$-weight spaces. We write $\mathfrak{u} = \mathrm{Lie}(U)$ for its Lie algebra, and $\mathfrak{u}_+$ for the summands with weight $\geq 2$. We regard $\mathfrak{u}/\mathfrak{u}_+$ as a Hamiltonian $HU$-space where
\begin{itemize}
    \item the $H$-action is via the adjoint action, 
    \item $U$ acts by translation via $U/U_+ \simeq \mathfrak{u}/\mathfrak{u}_+$, and
    \item the symplectic form is given by
    $$(x,y) \longmapsto \langle f, [x,y]\rangle.$$
\end{itemize}

The resulting Hamiltonian $HU$-space is denoted by $(\mathfrak{u}/\mathfrak{u}_+)_f$, and we can think of it as a pointed affine space with base point $f$. Finally, the Hamiltonian $G$-space attached to the triple $(H, S, \rho)$ is obtained by \textit{symplectic induction} from $HU$ to $G$ by the formula
\begin{equation} \label{eqn: Whittaker induction is symplectic induction}
    (H,S, \rho) \longmapsto M = \mathrm{WInd}^G_{H,\rho}(S):= \bigg(S \times (\mathfrak{u}/\mathfrak{u}_+)_f \times^{HU}_{(\mathfrak{h}+\mathfrak{u})^*} \, T^*G\bigg).
\end{equation}
By the theory of Slodowy slices, one can also rewrite $M$ as a vector bundle over the homogeneous space $H \backslash G$ (although its symplectic nature becomes less obvious)
\begin{equation} \label{eqn: Whittaker induction formula}
    M = \big[S \oplus (\mathfrak{m}\cap \g_e)\big] \times^H G
\end{equation}
where $\mathfrak{m} \subset \mathfrak{g}^*$ is the annihilator of $\mathfrak{h}$. We may equip $M$, as presented in \eqref{eqn: Whittaker induction formula}, with a $\Ggr$-action which gives it the structure of a graded Hamiltonain $G$-space: 
\begin{itemize}
    \item $\Ggr$ acts on the $G$ factor through left multiplication by $\exp(\rho)$ by identifying $\mathrm{Lie}(\Ggr) = \mathrm{span}_\CC(h) \subset \mathfrak{sl}_2$,
    \item $\Ggr$ acts on $S$ by linear scaling, and
    \item $\Ggr$ acts by weight $2+t$ on the weight $t$ component of $\g_e$. 
\end{itemize}

\begin{ex}
    Consider the BZSV triple $(H,S = 0, \rho = \mathrm{triv})$ for $G$, where $H \subset G$ is a reductive subgroup. Then we obtain the graded Hamiltonian $G$-space
    \begin{equation}
        M = \mathrm{WInd}^G_{H, \mathrm{triv}}(0) = \mathfrak{m} \times^H G \simeq T^*(H \backslash G)
    \end{equation}
    with $\Ggr$ acting by weight $2$ on $\mathfrak{m}$, i.e., scaling the cotangent fibers with weight 2. The moment map is given by the formula
    \begin{equation}
        \mathfrak{m} \times^H G \longrightarrow \g^*
    \end{equation}
    $$(\xi, g) \longmapsto \mathrm{Ad}_{g^{-1}}(\xi)$$
    where $\mathrm{Ad}_{g^{-1}}$ is the right adjoint action (equivalently, left adjoint action by $g^{-1}$).
\end{ex}

\begin{ex}
    Consider the BZSV triple $(1,0,\rho = \rho_{\mathrm{prin}})$ for $G$, where $\rho_{\mathrm{prin}}$ is a principal $\mathfrak{sl}_2$ in $\g$. Then $M$ is the equivariant Slodowy slice 
    \begin{equation}
 M = \mathrm{WInd}_{1,\rho}^G(0) = \mathcal{S}_{\mathrm{\rho}_{\mathrm{prin}}} \times G = (f+\g_e) \times G
    \end{equation}
    and can be alternatively viewed as a twisted symplectic reduction of $T^*G$ by $U$ (on the left) at the moment map fiber of $f \in \mathfrak{u}^*$. With the obvious map $M \to U\backslash G$ by projection, it acquires the structure of a $T^*$-torsor over $U \backslash G$, and can be understood as a twisted cotangent bundle. 
\end{ex}

We verify a simple weight condition for those graded Hamiltonian space arising from BZSV triples. 

\begin{lem}
    Let $(H,S,\rho)$ be a BZSV triple for $G$, and let $M = \mathrm{WInd}_{H,\rho}^G(S)$ be the graded Hamiltonian $G$-space obtained by Whittaker induction. Then the moment map $\mu: M \to \g^*$ is $\Ggr$-equivariant where $\Ggr$ scales $\g^*$ by weight $2$. 
\end{lem}
\begin{proof}
    Write $S = T^*V$ for a Lagrangian subspace $V \subset S$ (we do not assume that $V$ is a subrepresentation of $H$). Note that the moment map of $S$ with respect to $H$ is the canonical morphism
    \begin{equation}
       \mu_S: S = T^*V \longrightarrow \mathrm{End}(S)^* \longrightarrow \mathfrak{h}^* \simeq \mathfrak{h}
    \end{equation}
    sending a point $(x,\xi) \in T^*V = V \oplus V^*$ to the endomorphism $x \xi$ followed by the natural projection dual to $\mathfrak{h} \subset \mathrm{End}(S)$. This morphism is quadratic so we see that $\mu_S$ is indeed $\Ggr$-equivariant with respect to the weight 2 scaling action on $\mathfrak{h}^*$. 

    By \eqref{eqn: Whittaker induction formula} we can write $M$ explicitly as $M = S \oplus (\mathfrak{m} \cap \g_e) \times^H G$, with $\Ggr$ action specified by the bullet points under \eqref{eqn: Whittaker induction formula}. In the $S$-coordinate the moment map is the composition of the quadratic $H$-moment map in the previous paragraph with the linear map $\mathfrak{h}^* \to \mathfrak{g}^*$. In the $\mathfrak{m} \cap \g_e$-coordinate, we see that for $t \in \Ggr$ and $(v_j,g) \in V_{n_j} \times^H G$, 
    $$
        t \cdot (v_j, g) = (t^{2+n_j}v_j, t^\rho g) \overset{\mu}{\longmapsto} t^{2+n_j}\mathrm{Ad}_{(t^\rho g)^{-1}}(v_j) = t^2\mathrm{Ad}_{g}(v_j)
    $$
    since $\mathrm{Ad}_{t^{-\rho}}$ acts on $v_j$ by weight $-n_j$ by definition, so $\mu$ is indeed of weight 2. 
\end{proof}

As explained in \S 3.5.1 of \cite{BZSV}, the Hamiltonian $G$-space built out of the above Whittaker induction procedure will always satisfy 3 out of the 5 conditions of being a \textit{hyperspherical $G$-variety}; experts will observe that we are dropping the coisotropicity condition (which is a smallness condition on $M$ relative to the $G$-action), and the condition that the stabilizer of a generic point in $M$ be connected since they do not play a role for us.

\subsection{Gaiotto's Lagrangians and BAA branes}

As explained by Gaiotto \cite{Gaiotto} in the physical context and later mathematically by Ginzburg--Rozenblyum \cite{Ginzburg-Rozenblyum}, one can use graded Hamiltonian actions as labels for A-type boundary conditions for the Hitchin system. The key step is to consider a certain Lagrangian object over the Hitchin moduli space. 

\begin{defn}[Gaiotto, Ginzburg--Rozenblyum] \label{defn Gaiotto Lagrangian}
    Let $(M,\mu: M \to \mathfrak{g}^*)$ be a graded Hamiltonian $G$-space of weight 2, and let $L$ be a line bundle on $\Sigma$ for which we choose a square root $L^{1/2}$. We define \textit{Gaiotto's Lagrangian} attached to $M$ as the derived mapping space
    \begin{equation}
        \mathrm{Lag}^L(M) := \mathrm{Sect}(\Sigma, [M_{L^{1/2}}/G]) = \left\{(E,s) : \begin{matrix} E \in \mathrm{Bun}_G \\ \text{ and } s \in \mathrm{Sect}(\Sigma, M_{EL^{1/2}})\end{matrix} \right\}
    \end{equation}
    which has a natural \textit{global moment map} 
    \begin{equation}
        \mu_M: \mathrm{Lag}^L(M) = \mathrm{Sect}(\Sigma, [M_{L^{1/2}}/G]) \longrightarrow \mathrm{Sect}(\Sigma, [\g_L/G]) = \mathrm{Higgs}_G^L
    \end{equation}
    $$(E,s) \longmapsto (E,\mu(s))$$
    to the moduli stack of $L$-twisted $G$-Higgs bundles.
\end{defn}

Given a graded Hamiltonian $G$-space $\mu: M \to \g^*$, note that the natural morphism 
\begin{equation}\label{eq coarse quotient of Hamiltonian action}
    [M/G] \longrightarrow M\sslash G
\end{equation}
from the stack to the coarse quotient induces a morphism
\begin{equation}
    \chi_M^L: \mathrm{Lag}^L(M) \longrightarrow \mathrm{Sect}(\Sigma, (M\sslash G)_{L^{1/2}}) =: \mathcal{A}_M^L
\end{equation}
which, in the case when $L = K$ and $M$ is replaced with $\g^*$, gives the usual Hitchin morphism to the Hitchin base. Writing the $L$-twisted Hitchin base as $\mathcal{A}_G^L = \mathrm{Sect}(\Sigma, \g^*_L\sslash G)$, we have a commutative diagram 
\begin{equation} \label{eqn relative Hitchin fibration}
    \begin{tikzcd}
	{\mathrm{Lag}^L(M)} & {\mathrm{Higgs}_G^L} \\
	{\mathcal{A}_M^L} & {\mathcal{A}^L_G}
	\arrow["{\mu_M^L}", from=1-1, to=1-2]
	\arrow["{\chi_M^L}"', from=1-1, to=2-1]
	\arrow["{\chi_G^L}", from=1-2, to=2-2]
	\arrow[from=2-1, to=2-2]
\end{tikzcd}
\end{equation}
where the bottom horizontal arrow is induced by the $\Ggr$-equivariant moment map on coarse quotients $\mu: M\sslash G \to \g^*\sslash G$.

The main theorem of \cite{Ginzburg-Rozenblyum} states that, in the case when $L = K$ so that $\mathrm{Higgs}_G^K = \mathrm{Higgs}_G$ is the usual moduli stack of $G$-Higgs bundles, $\mu_M$ has the structure of a Lagrangian morphism for arbitrary graded Hamiltonian $G$-spaces $M$. Their argument uses the following two basic principles of derived symplectic geometry:
\begin{itemize}
    \item The morphism $[M/G] \to [\g^*/G] = T^*[1]BG$ can be equipped with the structure of a 1-shifted Lagrangian, and
    \item (a $K^{1/2}$-twisted version of) the fact that applying the functor $\mathrm{Map}(\Sigma, \, \cdot \, )$, where $\Sigma$ is Calabi-Yau, to a 1-shifted Lagrangian morphism yields a Lagrangian morphism. 
\end{itemize}

When we consider $L = K$, we will generally omit $L$ from our notation. Note that for general $L$, the morphsim $\mu_M$ is not Lagrangian in any sense, but we will continue to refer to $\mathrm{Lag}^L(M)$ as Gaiotto's Lagrangian. 

\begin{ex}
    It is important to note that $\mu_M: \mathrm{Lag}(M) \to \mathrm{Higgs}_G$ is Lagrangian in a derived sense, and the naïve comparison with Lagrangian submanifolds (or even finite covers of Lagrangian submanifolds) of the (classical, schematic) moduli space of stable $G$-Higgs bundles may not be direct. For instance, consider $\Gm$ with the trivial action on $T^*\mathbf{A}^1$ and $\Ggr$ scaling $T^*\mathbf{A}^1$ as a vector space. We consider the moment map $\mu: T^*\mathbf{A}^1 \to \mathrm{Lie}(\Gm)^*$ sending everything to 0. Then the underlying classical stack of $\mathrm{Lag}(T^*\mathbf{A}^1)$ is the moduli of a line bundle $L$, and two (unrelated) global sections $s_1,s_2$ of $K^{1/2}$. The global moment map $\mu_M$ sends $(L, s_1,s_2)$ to $(L,0)$, the Higgs line bundle with trivial Higgs field, and there is no classical sense in which this is a Lagrangian over $\mathrm{Higgs}_{\Gm}$.
\end{ex}

Following the the notion of BZSV's period sheaves in the de Rham and Betti settings, we make the following definition.
\begin{defn}
    Let $(M, \mu: M \to \g^*)$ be a graded Hamiltonian $G$-space of weight 2. We define the \textit{Dolbeault period sheaf} attached to $M$ as 
    \begin{equation}
        A_{\mathrm{Dol}}(M) := \mu_{M,*}\mathcal{O} \in \mathrm{QC}(\mathrm{Higgs}_G^L).
    \end{equation}
\end{defn}

Let us apply the previous constructions to the following specialized scenario. Consider two reductive groups $C$ and $G$, with a joint Hamiltonian action $C \times G \acts M$ of weight 2. Then Definition \ref{defn Gaiotto Lagrangian} defines a morphism
    \begin{equation}
        \mu_M: \mathrm{Lag}^L(M) \longrightarrow \mathrm{Higgs}_C^L \times \mathrm{Higgs}_G^L.
    \end{equation}

    We can then view the Dolbeault period sheaf $A_{\mathrm{Dol}}(M)$ as a Fourier--Mukai kernel which defines a functor
    \begin{equation}
        \Phi_M: \mathrm{QC}(\mathrm{Higgs}_C^L) \longrightarrow \mathrm{QC}(\mathrm{Higgs}_G^L)
    \end{equation}
defined by the diagram
\begin{equation}
    \begin{tikzcd}
	& {\mathrm{Higgs}_C^L \times \mathrm{Higgs}_G^L} \\
	{\mathrm{Higgs}_C^L} && {\mathrm{Higgs}_G^L}
	\arrow["p"', from=1-2, to=2-1]
	\arrow["q", from=1-2, to=2-3]
\end{tikzcd}
\end{equation}
using the usual formula
\begin{equation}
    \Phi_M := q_*\big(A_{\mathrm{Dol}}(M)\otimes p^*( \, \cdot \, )\big).
\end{equation}
 
\begin{rem} \label{rem GL nonsense 1}
    There exists, at a physical level of rigor, a 4-category $\mathfrak{B}_{\mathrm{GL}}$ of boundary conditions for geometric Langlands (see, for instance Section 2 of \cite{Oblomkov-Rozansky} for a mathematically inclined summary). The first two layers of this category $\mathfrak{B}_{\mathrm{GL}}$ are defined as follows: objects of this category are reductive groups, and the 1-morphisms between two reductive groups $C$ and $G$ are joint Hamiltonian actions of $C \times G$. 
    
    Our construction of $\Phi_M$ is an A-twist interpretation of such an interface between $C$-gauge theory and $G$-gauge theory. More precisely, we have built a functor $A_{\mathrm{Dol}}$ (A-twist of Dolbeault geometric Langlands) on $\mathfrak{B}_{\mathrm{GL}}$ which assigns
    \begin{equation}
        \mathfrak{B}_{\mathrm{GL}}\ni G \text{ reductive group } \overset{A_{\mathrm{Dol}}}{\longmapsto} \mathrm{QC}(\mathrm{Higgs}_G),
    \end{equation}
    \begin{equation}
        \bigg(C \overset{M}{\longrightarrow} G\bigg) \text{ 1-morphism } \overset{A_{\mathrm{Dol}}}{\longmapsto} \bigg(\Phi_M: A_{\mathrm{Dol}}(C) \to A_{\mathrm{Dol}}(G)\bigg).
    \end{equation}

    The de Rham and Betti versions of these functors are closely related to the period sheaves of the relative Langlands program; see for instance Section 12.3.6 of \cite{BZSV} where the relevant example appears in the automorphic literature under the name of \textit{theta correspondences}. 
    \end{rem}
    
\begin{rem} \label{rem AKSZ}
    Another, more ``nonlinear" but rigorous way to interpret the functor $A_{\mathrm{Dol}}$ is to take as its target not dg-categories, but rather the category $\mathrm{LagCorr}$ of Lagrangian correspondences, using the AKSZ construction (see \cite{Calaque-Haugseng-Scheimbauer} and Section 0.3 of \cite{Safranov}). We first consider $\mathfrak{B}_{\mathrm{GL}}$ to be (a full subcategory of) the 1-shifted Weinstein category whose first two layers are defined by the following data:
    \begin{itemize}
        \item The objects of $\mathfrak{B}_{\mathrm{GL}}$ are reductive groups $G$, which label the 1-shifted sympletic stack $T^*[1]BG \simeq [\g^*/G]$.
        \item The 1-morphisms are given by 1-shifted Lagrangian correspondences. Hamiltonian actions give examples of 1-morphisms by considering their equivariant moment maps. 
    \end{itemize}
    The category $\mathrm{LagCorr}$ is defined identically, with a shift by $(-1)$:
    \begin{itemize}
        \item The objects of $\mathrm{LagCorr}$ are (0-shifted) symplectic stacks, 
        \item the 1-morphisms are given by Lagrangian correspondences.
    \end{itemize}
    We may then apply AKSZ formalism to define, for a Riemann surface $\Sigma$ equipped with a choice of $K^{1/2}$, a functor $A_{\mathrm{Dol}}: \mathfrak{B}_{\mathrm{GL}} \to \mathrm{LagCorr}$ as follows:
    \begin{equation}
    \mathfrak{B}_{\mathrm{GL}} \ni G \text{ reductive group } \overset{A_{\mathrm{Dol}}}{\longmapsto} \mathrm{Higgs}_G,
    \end{equation}
    \begin{equation}
        \bigg(C \overset{M}{\longrightarrow} G\bigg) \text{ 1-morphism } \overset{A_{\mathrm{Dol}}}{\longmapsto} \bigg(\mathrm{Lag}(M) \to \mathrm{Higgs}_C \times \mathrm{Higgs}_G\bigg).
    \end{equation}

\end{rem}

\subsection{2-Functoriality} \label{subsect 2functoriality}

The primary advantage of describing Gaiotto's Lagrangians as mapping stacks is that functoriality properties of the construction $M \mapsto \mathrm{Lag}(M)$ become evident. Indeed, if we have a morphism of graded Hamiltonian $G$-spaces of weight 2
\begin{equation} \label{eqn morphism of Hamiltonian action} \begin{tikzcd}
	{M_1} && {M_2} \\
	& {\g^*}
	\arrow["\phi", from=1-1, to=1-3]
	\arrow["{\mu_1}"', from=1-1, to=2-2]
	\arrow["{\mu_2}", from=1-3, to=2-2]
\end{tikzcd}\end{equation}
then taking $L^{1/2}$-sections from $\Sigma$ into the above diagram yields a morphism (of Lagrangians when $L = K$) over $\mathrm{Higgs}_G^L$:
\begin{equation}\label{eqn functoriality of Lag} \begin{tikzcd}
	{\mathrm{Lag}^L(M_1)} && {\mathrm{Lag}^L(M_2)} \\
	& {\mathrm{Higgs}_G^L}
	\arrow["{\mathrm{Lag}(\phi)}", from=1-1, to=1-3]
	\arrow["{\mu_{M_1}}"', from=1-1, to=2-2]
	\arrow["{\mu_{M_2}}", from=1-3, to=2-2]
\end{tikzcd}\end{equation}
We will see in Section \ref{sect Cayley morphism} that the Cayley correspondence, a well-known construction in higher Teichmüller theory, can be seen as an example of such a morphism of Lagrangians.

\begin{rem} \label{rem GL nonsense 2}
    Returning to the expectations outlined in Remark \ref{rem GL nonsense 1}, we can try to extend the definition of $A_{\mathrm{Dol}}$ to 2-morphsims in the category $\mathfrak{B}_{\mathrm{GL}}$. If $M_1, M_2$ are two Hamiltonian actions, then a 2-morphism in $\mathfrak{B}_{\mathrm{GL}}$ is an equivariant Lagrangian correspondence $\mathcal{L} \to M_1^- \times M_2$ (where $M_1^-$ denotes the space $M_1$ with negated symplectic form). In the simplest case, a morphism of Hamiltonian actions $\phi: M_1 \to M_2$ gives a Lagrangian correspondence $\mathrm{Graph}(\phi) \subset M_1^- \times M_2$, and we can interpret diagram \eqref{eqn functoriality of Lag} as the evaluation of $A_{\mathrm{Dol}}$ on such a 2-morphism, i.e., 
    \begin{equation}
        \left(M_1 \overset{\mathrm{Graph}(\phi)}{\longrightarrow} M_2\right) \overset{A_{\mathrm{Dol}}}\longmapsto \left(A_{\mathrm{Dol}}(M_2) \to A_{\mathrm{Dol}}(M_1)\right)
    \end{equation}
    which induces a natural transformation of Fourier--Mukai functors $\Phi_{M_2} \to \Phi_{M_1}$. Note that since we consider functions on Gaiotto's Lagrangians, the functor $A_{\mathrm{Dol}}$ is \textit{contravariant} on 2-morphisms.
    
    The AKSZ version of $A_{\mathrm{Dol}}$, following Remark \ref{rem AKSZ}, interprets the 2-morphism $\phi$ in $\mathfrak{B}_{\mathrm{GL}}$ using the graph of $\mathrm{Lag}(\phi)$:
    \begin{equation}
        \bigg(M_1 \overset{\phi}{\to} M_2 \bigg) \overset{A_{\mathrm{Dol}}}{\longmapsto} \, \bigg( \mathrm{Graph}(\mathrm{Lag}(\phi)) \longrightarrow  \mathrm{Lag}(M_1) \times_{\mathrm{Higgs}_G} \mathrm{Lag}(M_2) \bigg).
    \end{equation}
    See also Theorem A of \cite{CF} for an analogous construction on the B-twist (more precisely, a hyperholomorphic family of twists). 
\end{rem}

\section{Slodowy slices and homogeneous cotangent bundles} \label{sect Slodowy and homogeneous}

We are ready to analyze those Lagrangians defined in Section \ref{sect BAA branes} arising from specific Hamiltonian actions labeled by BZSV triples. We find that this produces moduli stacks representing several categories of fundamental importance: Collier--Sanders' global Slodowy category \cite{Collier-Sanders} (in fact we slightly generalize their construction to cover the case of odd $\mathfrak{sl}_2$'s), and the category of $G^\RR$-Higgs bundles. 

\subsection{Collier--Sanders' Slodowy category} \label{subsect Collier Sanders}

In this section we review Collier--Sanders' construction of the Slodowy category and use it to describe a category of $(G,P_\rho)$-Higgs bundles arising from a $\rho: \mathfrak{sl}_2 \to \g$. Let $C_\rho := Z_G(\rho)$ be the centralizer subgroup of $\rho$, with Lie algebra $\frakc = \Lie(C_\rho)$. 

\begin{defn}[Definition 5.1 \cite{Collier-Sanders}] \label{defn Slodowy category}
  Let $\rho$ be an even $\mathfrak{sl}_2$ triple. The $\rho$-Slodowy category $\calB_{\rho}$ contains the following data:
  \begin{itemize}
  \item Objects are triples $(F,\phi,\set{\psi_{n_j}})$, where $(F,\phi)$ is a ${C_\rho}$-Higgs bundle on $\Sigma$, and $\psi_{n_j}$ are the following collection of poly-differentials. Recalling the notation from Section \ref{subsubsect JM data}, let $\set{n_j}$ be the set of highest weights in the $\rho(h)$-weight decomposition of $\g$, and let $V_{n_j} \subset W_{n_j}$ be the highest weight lines. Since $[\mathfrak{c},V_{n_j}] \subset V_{n_j}$, we may view $V_{n_j}$ as ${C_\rho}$-representations, and we let
    \begin{align*}
      \psi_{n_j} \in H^0(\Sigma, (V_{n_j})_{F} \otimes K^{n_j/2+1}). 
    \end{align*}
  \item A morphism between two objects $(F,\phi,\set{\psi_{n_j}})$ and $(F',\phi',\set{\psi_{n_j}'})$ is an isomorphism of Higgs bundles $\Phi: (F, \phi) \overset{\sim}{\to} (F', \phi')$ such that $\psi_{n_j}=\Phi^*\psi_{n_j}'$ for all $j$.
  \end{itemize}
\end{defn}
Continuing to assume that $\rho$ is even, we write $n_j = 2m_j$. Note that the auxiliary sections $\psi_{n_j} = \psi_{2m_j}$'s lie in $H^0(\Sigma, (V_{2m_j})_{F} \otimes K^{m_j+1})$ so there was no need to choose a $K^{1/2}$ for this construction a posteriori. We may define the following $\C^*$-action on $\calB_{\rho}$, extending the usual action of scaling the Higgs field: 
\begin{align}
  \label{eq:Cstar-action-on-Slodowy}
  t \cdot (F, \phi, \set{\psi_{2m_j}}) = (F, t\phi, \set{t^{m_j+1}\psi_{2m_j}}). 
\end{align}

From the data parametrized by the category $\mathcal{B}_\rho$, Collier--Sanders construct certain Higgs-theoretic analogues of parabolic opers, i.e., principal $G$-bundles with a parabolic reduction, whose Higgs field is ``as transverse as possible" to the reduction. 

\begin{defn}[Definition 4.1 \cite{Collier-Sanders}] \label{defn (G,P)-Higgs}
  Let $\rho$ be an even $\mathfrak{sl}_2$-triple. A $(G,P_\rho)$-Higgs bundle is a triple $(E,F,\phi)$, where $(E,\phi)$ is a $G$-Higgs bundle, and $E = F \times^{P_\rho} G$ is a reduction of structure group to $P_\rho$. Moreover the second fundamental form $\mathrm{SFF}_{F}(\phi)$ of $\phi$ relative to $F$ satisfies a transversality condition: 
  \begin{align*}
    \mathrm{SFF}_{F}(\phi) \subset H^0(\Sigma, O_K), 
  \end{align*}
  where $O \subset \frakg/\frakp_\rho$ is the unique open dense $P_\rho$-orbit (equivalently, $L_\rho$-orbit for the Levi subgroup $L_\rho \subset P_\rho$) in the subspace 
  $$
      \{x \in \g : \mathrm{ad}_x(\mathfrak{u}_\rho) \subseteq \mathfrak{p}_\rho\}/\mathfrak{p}_\rho \subset \g/\mathfrak{p}_\rho.
  $$
  Following \textit{loc. cit}, we denote the category of $(G,P_\rho)$-Higgs bundles by $\mathrm{Op}^0(G,P_\rho)$, which we call the category of $0$-opers.\footnote{In reference to the fact Higgs bundles can be regarded as $\lambda$-connections for $\lambda=0$}
\end{defn}

\begin{ex}
  Let $B$ be a Borel subgroup in $G=\PSL_2$. A $(\PSL_2,B)$-Higgs bundle $(E,F,\phi)$ is a $\PSL_2$-Higgs bundle $(E,\phi)$ with a reduction of structure group to $B$.

  Consider the rank 2 vector bundle associated to the standard representation (which we will just keep denoting by $E$) with $\det(E) \simeq \calO_\Sigma$. The $B$-reduction yields a short exact sequence
  \begin{align*}
    0 \to L_2 \to E \to L_1 \to 0
  \end{align*}
  where $L_1$ and $L_2$ are line bundles. The transversality condition on the second fundamental form of the Higgs field implies that $\mathrm{SFF}(\phi): L_2 \xrightarrow{\sim} L_1 \otimes K$ is an isomorphism. We deduce that $L_1, L_2$ are line bundles of degrees $\mp (g-1)$ respectively, and a prototypical example is given by the uniformizing bundle $\Theta$ of \eqref{eqn uniformizing bundle}. 
\end{ex} 

Let $S_\rho \subset G$ be the connected subgroup exponentiating $\rho$, and let $B_\rho \subset S_\rho$ be the Borel subgroup containing $\exp{(e)}$. We recall the following theorem which can be found in \cite[Section~5.4]{Collier-Sanders}. 

\begin{thm}[Collier--Sanders]
  \label{thm:Theta-Slodowy-equiv-cat}
  Let $\rho$ be an even $\mathfrak{sl}_2$ triple. Fixing the uniformizing $(S_\rho,B_\rho)$-Higgs bundle $\Theta = (E_0, F_0, \phi_0)$ as in \eqref{eqn uniformizing bundle}, there is a $\Theta$-Slodowy functor
  \begin{align*}
    &\mathrm{Slo}_{\Theta}: \calB_\rho \longrightarrow \mathrm{Op}^0(G,P_\rho)\\
    &(E,\phi,\set{\psi_{n_j}}) \longmapsto \\
    & \quad \left(\mathrm{Ind}_{C_\rho \times S_\rho}^G(E \times E_0), \mathrm{Ind}_{C_\rho \times S_\rho}^{P_\rho}(E \times E_0), \phi+ \phi_0 + \sum_{j} \psi_{n_j}\right), 
  \end{align*}
  which is an equivalence of categories when $S_\rho = \mathrm{PSL}_2$ and essentially surjective and full when $S_\rho = \mathrm{SL}_2$.  
\end{thm}

\begin{rem}
    The functor $\mathrm{Slo}_\Theta$ is furthermore $\CC^\times$-equivariant with respect to the action defined by \eqref{eq:Cstar-action-on-Slodowy} and usual scaling of the Higgs field. 
\end{rem}

\begin{rem} \label{rem reprove CS equivalence}
    Let $\rho$ be an even $\mathfrak{sl}_2$ triple, and $P_\rho = L_\rho U_\rho$ be the Levi decomposition of the parabolic subgroup defined by $\rho$. The category of $(G,P_\rho)$-Higgs bundles can be represented by the following mapping stack:
$$
    \mathrm{Higgs}_{(G,P_\rho)} := \mathrm{Sect}(\Sigma, [(O+\mathfrak{u}_\rho^\perp) \times^{L_\rho \times U_\rho} G/G]_K).
$$
Indeed, this stack of sections parametrizes a principal $G$-bundle along with a $P_\rho$-reduction, whose Higgs field relative to the reduction lands in the orbit $O$. Being an orbit, $O$ can be written as a homogeneous space for $L_\rho$. We can fix the base point $f$ and write 
$$
    [(O+\mathfrak{u}_\rho^\perp) \times^{P_\rho} G \simeq (f+\mathfrak{u}_\rho^\perp) \times^{C_\rho \times U_\rho} G.
$$
since $\mathrm{Stab}_{L_\rho}(f) = C_\rho$. On the other hand, the fundamental property of the Slodowy slice tells us that $f+\mathfrak{u}_\rho^\perp$ has a transversal slice through $U_\rho$-orbits given by $\mathcal{S}_\rho = f+ \g_{e}$. So we have an isomorphism 
$$
    \mathrm{Higgs}_{(G,P_\rho)} \simeq \mathrm{Sect}(\Sigma, [(f+\mathfrak{u}_\rho^\perp) \times^{C_\rho U_\rho} G/G]_K) \simeq \mathrm{Sect}(\Sigma, [\mathcal{S}_\rho \times^{C_\rho} G/G]_K). 
$$
It will become clear with the proof of Theorem \ref{thm CS moduli stack} that this equivalence recovers Collier--Sanders' Theorem \ref{thm:Theta-Slodowy-equiv-cat} upon passing to $\CC$-points.
\end{rem}

\subsection{Collier--Sanders moduli stack}

In this section we give an immediate application of our constructions: namely, we provide moduli stacks representing the $L$-twisted (and odd $\rho$) generalizations of Collier--Sanders' categories, we interpret the Slodowy functor as a representable morphism of moduli stacks, and observe that in the $L = K$ case the Slodowy functor can be regarded as a Lagrangian correspondence.

We start by building the relevant Hamiltonian action. Given an even $\mathfrak{sl}_2$ triple $\rho$, consider the shifted Hamiltonian reduction 
$$
    M_{\mathrm{CS}} = T^*_f(U \backslash G) = \mathrm{WInd}_{1,\rho}^G(0)
$$
where $U$ is the unipotent subgroup of the associated parabolic of $\rho$ as in Section \ref{subsubsect JM data}. We view $M_{\mathrm{CS}}$ as a graded Hamiltonian $C_\rho \times G$-space via right $G$-action and left-inverse $C_\rho$-action, with grading given by its description as a Whittaker induction. 

\begin{ex} \label{ex principal slice is Hitchin section}
    Consider the simplest case of $G = \mathrm{SL}_2$ with $\rho$ the principal $\mathfrak{sl}_2$-triple defined by
    $$
        \rho(f) = \begin{bmatrix}
            0&0\\1&0\end{bmatrix}.
    $$
    By the explicit presentation of \eqref{eqn: Whittaker induction formula}, we can write
    $$
        M_{\mathrm{CS}} = \begin{bmatrix} 0&\ast\\1&0\end{bmatrix} \times G \simeq \mathbf{A}^1 \times G
    $$
    with $\Ggr$ acting by weight 4 on the $\mathbf{A}^1$ coordinate and by $\mathrm{exp}(\rho)$ on the $G$-coordinate. Considering $L^{1/2}$-twisted sections, we obtain
    \begin{align*}
        \mathrm{Lag}^L(M_{\mathrm{CS}}) &= \mathrm{Sect}(\Sigma, [(\mathbf{A}^1 \times G)_{L^{1/2}}/G])\\
        &= \left\{(E,s_1,s_2): \begin{matrix}E \text{ is a principal }G\text{-bundle, and}\\ s_1 \in R\Gamma(\Sigma, L^2), \, s_2 \in \mathrm{Sect}(\Sigma, \mathrm{Isom}(E, E_0))\end{matrix}\right\}
    \end{align*}
where $E_0$ is the induced bundle $(L^{1/2} \oplus L^{-1/2}) \times^{\Gm} G$. In particular, the existence of a section $s_2$ implies that $E \simeq E_0$ and we may write 
$$
    \mathrm{Lag}^L(M_{\mathrm{CS}}) \simeq R\Gamma(\Sigma, L^2) \times \{E_0\} \times B\mu_2,
$$
with the trivial $\mu_2$-gerbe structure coming from the center of $G$. When $L = K$ and the genus $g(\Sigma)$ of the curve is at least 2, using the fact that $K^2$ has no higher cohomology we further simplify
$$
    \mathrm{Lag}(M_{\mathrm{CS}}) \simeq H^0(\Sigma, K^2) \times \{E_0\} \times B\mu_2
$$
and the global moment map $\mathrm{Lag}(M_{\mathrm{CS}}) \to \mathrm{Higgs}_G$ (written without the $B\mu_2$-factor) can be identified with the Hitchin section:
$$
    H^0(\Sigma, K^2) \times \{E_0\} \longrightarrow \mathrm{Higgs}_G
$$
$$\psi \longmapsto \left(E_0,  \phi_0 + \psi = \begin{bmatrix}
    0& \psi\\ 1 & 0
\end{bmatrix}\right).$$ 

More generally, when $G$ is a simple group and $\rho$ is a principal $\mathfrak{sl}_2$-triple, the Lagrangian $\mathrm{Lag}(M_{\mathrm{CS}}) \to \mathrm{Higgs}_G$ can be identified with the Hitchin section (multiplied by the trivial $Z_G$-gerbe) by an identical calculation. 
\end{ex}

Expanding upon the calculations of the preceding example, we have the following
\begin{thm} \label{thm CS moduli stack}
    Let $\Theta$ be the uniformizing bundle as in \eqref{eqn uniformizing bundle}, and let $\mathrm{Lag}^L(M_{\mathrm{CS}})$ be Gaiotto's Lagrangian associated to the graded Hamiltonian $C_\rho \times G$-space $M_{\mathrm{CS}}$. Then there is a diagram
    \begin{equation}
        \begin{tikzcd}
	& {\mathrm{Lag}^L(M_{\mathrm{CS}})} \\
	{\mathrm{Higgs}_{C_\rho}^L} & {\mathrm{Higgs}_{C_\rho \times G}^L} & {\mathrm{Higgs}_G^L}
	\arrow["{\mathrm{forget}}"', from=1-2, to=2-1]
	\arrow["{\mu_{\mathrm{CS}}}"', from=1-2, to=2-2]
	\arrow["{\mathrm{Slo}_\Theta}", from=1-2, to=2-3]
	\arrow["{p_1}"', from=2-2, to=2-1]
	\arrow["{p_2}", from=2-2, to=2-3]
\end{tikzcd}
    \end{equation}
with the following properties:
\begin{itemize}
    \item The Slodowy functor $\mathrm{Slo}_\Theta$ is a representable morphism of algebraic stacks. 
    \item When $L = K$, the global moment map $\mu_{\mathrm{CS}}$ is a Lagrangian correspondence between $\mathrm{Higgs}_{C_\rho}$ and $\mathrm{Higgs}_G$.
    \item When $L = K$ and $\rho$ is even, the morphism $\mathrm{Slo}_{\Theta}$ at the level of $\CC$-points recovers the Slodowy functor of Collier--Sanders $\mathrm{Slo}_\Theta: \mathcal{B}_\rho \to \mathrm{Higgs}_G(\CC)$. 
\end{itemize}
\end{thm}

\begin{proof}
    The second bullet point is an immediate consequence of the main theorem of \cite{Ginzburg-Rozenblyum}, with the first bullet point following from the definition of $\mu_{\mathrm{CS}}$ as the base change of a representable morphism (see Lemma \ref{lem Lagrangian is representable}). Continuing to verify the third point, note that by the Slodowy slice construction we can write
    $$
        M_{\mathrm{CS}} \simeq \left(\mathfrak{c} \oplus \bigoplus_{j=1}^N \, V_{n_j}\right) \times G
    $$
    with $C_\rho$ acting on the parenthesized vector space component and on $G$ by left multiplication; by convention the $n_j$'s are nonzero. The stacky quotient by $C_\rho \times G$ yields 
    \begin{equation}
    \label{eqn equivalent_stacks}[M_{\mathrm{CS}}/C_\rho \times G] = \bigg[\mathfrak{c} \oplus \bigoplus_{j=1}^N \, V_{n_j}\bigg/C_\rho\bigg]
    \end{equation}
    with $\Ggr$-grading of weight $2$ on $\mathfrak{c}$, and weight $2+n_j$ on $V_{n_j}$. Twisting by $L^{1/2}$ to the power of these weights, we see that $\mathrm{Lag}^L(M_{\mathrm{CS}})$ classifies tuples $(F, \phi, \{\psi_{n_j}\})$ where $F$ is a $C_\rho$-bundle on $\Sigma$, with $\phi$ a section of 
   $$
        \phi \in \mathrm{Sect}(\Sigma, \mathfrak{c} \times^{C_\rho \times \Ggr} FL^{1/2}) = \mathrm{Sect}(\Sigma, \mathrm{ad}(F) \otimes L)
    $$
    i.e., an $L$-twisted Higgs field on $F$, and $\psi_{n_j}$ are sections
    $$
        \psi_{n_j} \in \mathrm{Sect}(\Sigma, V_{n_j} \times^{C_\rho \times \Ggr} FL^{1/2}) = \mathrm{Sect}(\Sigma, (V_{n_j})_{F} \otimes L^{n_j/2+1}).
    $$
    This is exactly the Collier--Sanders Slodowy category $\mathcal{B}_\rho$ generalized to arbitrary line bundles $L$ and with possibly odd $n_j$'s, and the forgetful map takes $(F,\phi, \{\psi_{n_j}\}) \mapsto (F,\phi) \in \mathrm{Higgs}_{C_\rho}^L$.

    On the other hand, we consider the global moment map in the $G$-direction. Arguing as in Example \ref{ex principal slice is Hitchin section} using the left hand side of \eqref{eqn equivalent_stacks}, we see that $\mathrm{Lag}^L(M_\mathrm{CS})$ equivalently classifies tuples 
    $$
       \left(E,F, \iota, \phi+\phi_0+\sum_j \psi_{n_j}\right) 
    $$
    where $E$ is a $G$-bundle, $(F,\phi)$ is an $L$-twisted $C_\rho$-Higgs bundle, we have an isomorphism $\iota:E \simeq \mathrm{Ind}_{C_\rho \times S_\rho}^G(F \times E_0)$, and the $\psi_{n_j}$'s are sections of $(V_{n_j})_F \otimes L^{n_j/2+1}$ as before. The composition $p_2 \circ \mu_{\mathrm{CS}}$ maps the tuple above to $$\left(E, \phi+\phi_0+\sum_j \psi_{n_j}\right) \in \mathrm{Higgs}_G^L
    $$
    and is naturally isomorphic to $\mathrm{Slo}_\Theta$.
\end{proof} 

\begin{rem}
    Suppose $\rho$ is even. In the case when $S_\rho \simeq \mathrm{SL}_2$, Collier--Sanders' Slodowy functor (Theorem \ref{thm:Theta-Slodowy-equiv-cat}) fails to be faithful only because the morphism of groups $C_\rho \times S_\rho \to G$ is non-injective: it has the center $\mu_2$ as kernel. In our setup, $\mathrm{Lag}(M_{\mathrm{CS}})$ already has the extra $\mu_2$-stabilizer incorporated.
\end{rem}

Interpreting the functor $\mathrm{Slo}_\Theta$ concretely, we may treat the cases of $\rho$ even and $\rho$ odd on an equal footing. We propose the following definitions, generalizing those of Collier--Sanders' Definitions \ref{defn Slodowy category} and \ref{defn (G,P)-Higgs}.
\begin{defn}
    Let $\rho: \mathfrak{sl}_2 \to \g$ be an $\mathfrak{sl}_2$-triple. Let $P = P_\rho$ be the associated parabolic subgroup, $U = U_\rho$ its unipotent radical, and $C = C_\rho$ the centralizer of $\rho$ (as described in Section \ref{subsubsect JM data}. 
    \begin{itemize}
        \item The $\rho$-Slodowy category $\mathcal{B}_\rho$ has as objects $(F,\phi, \{\psi_{n_j}\})$ where $(F,\phi)$ is a $C$-Higgs bundle, and $\psi_{n_j} \in H^0(\Sigma, (V_{n_j})_F \otimes K^{n_j/2+1})$, with the same morphisms as in Definition \ref{defn Slodowy category}. Note that since $n_j$ is allowed to be odd, this definition requires the choice of $K^{1/2}$.
        \item The category $\mathrm{Op}^0(G,P)$ of $(G,P)$-Higgs bundles has as objects $(E,F,\phi)$ where $(E,\phi)$ is a $G$-Higgs bundle, $E = F \times^P G$ is a reduction of structure group to $P$, and the second fundamental form $\mathrm{SFF}_F(\phi)$ of $\phi$ relative to $F$ satisfies
        $$\mathrm{SFF}_F(\phi) \subset H^0(\Sigma, O_K)$$
        where $O$ is the $P$-orbit of $f+\mathfrak{u}_{-1} \subset \mathfrak{g}/\mathfrak{p}$. Note that in the odd case, we have $[f, \mathfrak{u}] \subsetneq \mathfrak{p}$ since $f \in \g_{-2}$ and thus $\mathrm{ad}_f: \g_1 \to \g_{-1} \not\subseteq \mathfrak{p}$. 
    \end{itemize}
\end{defn}

We obtain as an immediate corollary of the preceding Theorem \ref{thm CS moduli stack} and Remark \ref{rem reprove CS equivalence} that Collier--Sanders' Theorem \ref{thm:Theta-Slodowy-equiv-cat} can be generalized \textit{verbatim} to the case of odd $\rho$, using the above definitions.

\subsection{$G^\R$-Higgs bundles}

Another setting in which our constructions conveniently recover interesting moduli spaces related to $\mathrm{Higgs}_G$ is the consideration of $G^\RR$-Higgs bundles, which we review below. 

Let $G^\R$ be a connected semisimple real Lie group, with $\sigma: \frakg^\R \to \frakg^\R$ its Cartan involution on the Lie algebra. Let
\[ \frakg^\R = \frakh^\R \oplus \frakm^\R \]
be the $\pm 1$-eigenspace decomposition with respect to the Cartan involution. The Cartan involution extends to the group level:
\[ G^\R = H^\R \exp(\frakm^\R). \]
with $H^\RR$ being the maximal compact subgroup of $G^\RR$. Moreover, the adjoint action of $G^\R$ restricts to an isotropy action of $H^\R$ on $\frakm^\R$. 

We use a missing superscript $\R$ to mean that the object is complexified: in particular, we consider the complex reductive subgroup $H \subset G$ and $\mathfrak{m} = \mathfrak{m}^\R \otimes \C$ as an $H$-representation.

\begin{definition}
  \label{defn:G-R-Higgs-bundles}
  An \textit{$L$-twisted $G^\R$-Higgs bundle} on $\Sigma$ is a pair $(E,\phi)$, where $E$ is a holomorphic $H$-bundle $\Sigma$ and a $\phi$ is a section of $\frakm_E \otimes L$ which we call the \textit{Higgs field}. 
\end{definition}

A morphism of $G^\R$-Higgs bundles is a complex gauge transformation of $H$-bundles that preserves the Higgs field. From the work of \cite{Garcia-PradaEtAl2012, Schmitt2008}, we know that for $L = K$, the set of polystable $G^\R$-Higgs bundles up to isomorphism forms a complex analytic moduli space. 

By considering the cotangent bundle to a the homogeneous space $H \backslash G$, we obtain an algebraic moduli stack $\mathrm{Higgs}_{G^\RR}^L$ of $L$-twisted $G^\RR$-Higgs bundles, giving a natural notion of algebraic universal families of $G^\RR$-Higgs bundles. 

\begin{thm} \label{thm GR Higgs bundles}
    Let $M = T^*(H \backslash G)$. Then the $\CC$-points of the algebraic stack $\mathrm{Higgs}_{G^\R}^L := \mathrm{Lag}^L(M)$ is the category of $L$-twisted $G^\R$-Higgs bundles in the sense of \cite{Garcia-PradaEtAl2012, Schmitt2008}, and the global moment map $\mu_M: \mathrm{Lag}(M) \to \mathrm{Higgs}_G$ sending a $G^\R$-Higgs bundle to its induced $G$-Higgs bundle is a Lagrangian morphism when $L = K$.
\end{thm}
\begin{proof}
    By definition of $M$, we have
    $$
        [M_{L^{1/2}}/G] = [\mathfrak{m}_L/H]
    $$
    Thus, Gaiotto's Lagrangian $\mathrm{Lag}^L(M)$ classifies $H$-bundles $E$ with an $L$-valued Higgs field $\phi$ on $G \times^HE$ which lies in $\mathfrak{m}_L \subset \g_L$. The global moment map is simply
    $$
        \mu_M: \mathrm{Lag}^L(M) \ni (E, \phi) \longmapsto (E \times^H G, \phi) \in \mathrm{Higgs}_G^L,
    $$
    and the Lagrangianity of $\mu_M$ is ensured by \cite{Ginzburg-Rozenblyum}.
\end{proof}

Note that, while the general definition of $\mathrm{Lag}^L$ requires the choice of a square root $L^{1/2}$ of $L$, in the setting of $L$-twisted $G^\RR$-Higgs bundles such a choice is not necessary. 

\begin{rem}
    One can without problem consider when $H$ is an arbitrary reductive subgroup of $G$, but the stackyness and derived structures that are necessary for the global moment map to be Lagrangian may become unwieldy, especially if one's goal is to deduce statements about the moduli space of stable Higgs bundles. For an extreme case, consider $H = \{1\} \subset G$ being the trivial subgroup. Gaiotto's Lagrangian $\mathrm{Lag}(M = T^*G)$ has an underlying classical moduli stack parametrizing Higgs fields on the trivial $G$-bundle, admitting thus a smooth atlas by $H^0(\Sigma, \g \otimes K) \simeq H^0(\Sigma, K)^{\oplus \mathrm{dim}(G)}$. However, as a derived stack we have $\mathrm{Lag}(M) \simeq R\Gamma(\Sigma, K)^{\oplus \mathrm{dim}(G)}$; this is the cotangent fiber to $\mathrm{Higgs}_G$ at the trivial $G$-bundle, which is never classical. 
\end{rem}

\subsubsection{} \label{subsubsec:reality-check}

We explain briefly the origin of the terminology of $G^\RR$-Higgs bundles, a reference of which can be found in \cite{Hitchin1992}. In short, a Kobayashi--Hitchin-type theorem of García-Prada--Gothen--Mundet \cite{Garcia-PradaEtAl2012} states that there is a homeomorphism between the moduli space of stable $G^\R$-Higgs bundles and the moduli space of infinitesimally irreducible $G^\R$-valued representations of $\pi_1(\Sigma)$. In other words, starting from a stable $G^\RR$-Higgs bundle one may solve the associated Hitchin equation for $G$, and observe that the corresponding flat $G$-connection has holonomy valued in $G^\R$.

To this end, let $(E, \phi)$ be a stable $G^\R$-Higgs bundle (so that $E$ is a principal $H$-bundle), and let $F = \mathrm{Ind}_H^G(E)$ be the $G$-bundle induced from $E$. Having a reduction of structure group to $H$ induces a global holomorphic involution $\sigma: F \to F$ of $G$-bundles with $E = F^\sigma$ as fixed points, and there is an isomorphism
\begin{align}
  \label{eqn:complex-Cartan-induced-isomorphism}
  (E, \phi) \simeq \sigma^*(E, \phi).
\end{align}
By stability of $(E, \phi)$, there is a unique reduction of structure group $E = E^{\R} \times^{H^\R} H$ to the maximal compact subgroup $H^\R \subset H$ that solves the Kobayashi--Hitchin equation \cite[Theorem~2.24]{Garcia-PradaEtAl2012}:
\begin{align}
  \label{eqn:Kobayashi-Hitchin}
  \Omega_{D_A} - [\phi, \rho^*\phi] = 0, 
\end{align}
where $\rho: E \to E$ is the antiholomorphic involution of $H$-bundles whose fixed points define the principal $H^\R$-bundle $E^\R$, $D_A$ is the Chern connection on $E$ with respect to $E^{\R}$ and $\Omega_{D_A}$ is the curvature form of $D_A$. Since the Chern connection is an $H^\R$-connection, we have
\begin{align}
  \label{eqn:connection-form-fixed-by-rho}
  \rho^*A = A. 
\end{align}

Let $\tau := \rho \circ \sigma = \sigma \circ \rho$ be the antiholomorphic involution on $G$ whose fixed points define the real form $G^\RR$, so that $F^\tau$ is a principal $G^\R$-bundle. We may apply the Kobayashi--Hitchin theorem \cite{Hitchin1987, Simpson1988} to $(F,\phi)$ viewed as a stable $G$-Higgs bundle, then \eqref{eqn:complex-Cartan-induced-isomorphism} and the uniqueness of solution to \eqref{eqn:Kobayashi-Hitchin} implies that
\begin{align}
  \label{eqn:connection-form-fixed-by-sigma}
  \sigma^*A = A.
\end{align}

Consider the $G$-connection $\nabla := D_A + \phi - \rho^*\phi$ on $F$. Then \eqref{eqn:Kobayashi-Hitchin} implies that $(F,\nabla)$ is a flat $G$-connection, and we can calculate
\begin{align*}
  \tau^*(\nabla) &= \sigma^*\rho^*(A + \phi - \rho^*\phi) \\
                 &= A + \rho^*\sigma^*\phi - \sigma^*\phi \text{ by } \eqref{eqn:connection-form-fixed-by-rho}, \eqref{eqn:connection-form-fixed-by-sigma} \\
                 &= A - \rho^*\phi + \phi \text{ by } \eqref{eqn:complex-Cartan-induced-isomorphism} \\
                 &= \nabla,
\end{align*}
hence the holonomy of $\nabla$ takes values in $G^\R$.

\section{Cayley morphism} \label{sect Cayley morphism}

We are ready to explain our core observation in this article, that the Cayley correspondence is induced from a morphism of Hamiltonian actions (or, following Remarks \ref{rem GL nonsense 1}, \ref{rem AKSZ}, and \ref{rem GL nonsense 2}, a 2-morphism in the category $\mathfrak{B}_{\mathrm{GL}}$). In Sections \ref{subsect magical triples} and \ref{subsect Cayley paper summary}, for the reader's convenience we recall the notion of magical $\mathfrak{sl}_2$-triples and the original Cayley correspondence as constructed and studied in \cite{Cayley}. Given a magical $\mathfrak{sl}_2$-triple, we define two Hamiltonian actions in Section \ref{subsect two Hamiltonian actions} and a morphism between them which leads, in Section \ref{subsect Cayley morphism} to a morphism of Lagrangians over the Hitchin moduli stack, extending the Cayley correspondence. We demonstrate the versatility and efficiency of our perspective by recovering those basic geometric properties obtained by \cite{Cayley} at the level of moduli stacks, which, combined with a study of the stacky and derived structure of Gaiotto's Lagrangians in Appendix \ref{appendix}, recover the main statements of \textit{op. cit} in Section \ref{subsect good moduli spaces}.

\subsection{Magical $\mathfrak{sl}_2$-triples} \label{subsect magical triples}

Given an $\mathfrak{sl}_2$-triple $\rho:\mathfrak{sl}_2 \to \mathfrak{g}$, we define a vector space involution $\sigma_\rho:\mathfrak{g}\to \mathfrak{g}$ acting as $\pm 1$ on summands of the Jacobson--Morozov decomposition \eqref{eq:g-highest-weight-line-decomposition} of $\mathfrak{g}$ by the following rule:
$$
    \sigma_\rho(x)=\begin{cases}x, & x\in \mathfrak{c} \\ (-1)^{k+1}x, & x\in (\mathrm{ad}_f)^k(V_{n_j})\end{cases}.
$$

\begin{definition}
An $\mathfrak{sl}_2$-triple of $\mathfrak{g}$ is called \textit{magical} if $\sigma_\rho$ is a Lie algebra involution of $\mathfrak{g}$. If $\rho$ is magical, the real form $\mathfrak{g}^\R$ associated to $\sigma_\rho$ is called the \textit{canonical real form}. 
\end{definition}

\begin{definition}
  \label{def:cayley-real-form}
  Given a magical $\mathfrak{sl}_2$-triple, the \textit{Cayley real form} $\mathfrak{g}_\mathrm{Cay}^\R$ is defined to be the real form associated to the Lie algebra involution $\theta_\rho:\mathfrak{g}_0 \to \mathfrak{g}_0$,
\begin{align*}
    \theta_\rho(x)=
    \begin{cases}
    x, \quad \quad &x\in \mathfrak{c}, \\
    -x, \quad & \text{otherwise}.
    \end{cases}
\end{align*} 
\end{definition}
We warn the reader that the Cayley real form is not a reductive Lie algebra in general, as its complexification can contain factors of $\mathrm{Lie}(\mathbf{G}_a)$.

\begin{ex}
\label{principal sl2}By the classification of magical triples, all principal $\mathfrak{sl}_2$-triples are magical.
Consider the principal-$\mathfrak{sl}_2$ in $\mathfrak{g}=\mathfrak{sl}_3\C$, whose Jacobson--Morozov decomposition can be presented diagrammatically as follows, where we label the $(\pm 1)$-eigenspaces of $\sigma_\rho$ (in black) and $\theta_\rho$ (in red):

\begin{center}
\begingroup
\renewcommand{\arraystretch}{2} 
\newcommand{\bigplus}{\raisebox{0.7ex}{\scalebox{1.1}{$\scriptstyle\boldsymbol{+}$}}}
\newcommand{\bigminus}{\raisebox{0.7ex}{\scalebox{1.1}{$\scriptstyle\boldsymbol{-}$}}}
\newcommand{\subplus}{\raisebox{-0.7ex}{\scalebox{1.1}{$\scriptstyle\boldsymbol{+}$}}}
\newcommand{\subminus}{\raisebox{-0.7ex}{\scalebox{1.1}{$\scriptstyle\boldsymbol{-}$}}}
\newcommand{\spc}{\hspace{0.1em}} 

\begin{tabular}{c@{\hspace{12pt}}c@{\hspace{12pt}}c@{\hspace{12pt}}c@{\hspace{12pt}}c@{\hspace{12pt}}c}
 $W_4$ & $\bullet\spc^{\bigminus}$ & $\bullet\spc^{\bigplus}$ & $\bullet\spc^{\bigminus}_{\color{red}{\subminus}}$ & $\bullet\spc^{\bigplus}$ & $\bullet\spc^{\bigminus}$ \\ 
 $W_2$ &  & $\bullet\spc^{\bigminus}$ & $\bullet\spc^{\bigplus}_{\color{red}{\subminus}}$ & $\bullet\spc^{\bigminus}$ & \\  
  & $\mathfrak{g}_{-4}$ & $\mathfrak{g}_{-2}$ & $\mathfrak{g}_{0}$ & $\mathfrak{g}_{2}$ & $\mathfrak{g}_{4}$    
\end{tabular}
\endgroup
\end{center}
The canonical real form $\mathfrak{g}^\R$ must have a three-dimensional maximal compact subalgebra since
$$\dim \mathfrak{g}^{\sigma_\rho}=\dim \mathfrak{h}^\R=3.$$ 
Now $\mathfrak{sl}_3$ has precisely three real forms: $\mathfrak{su}_3$, $\mathfrak{su}(1,2)$, $\mathfrak{sl}_3\R$, and their maximal compact subalgebras all have different dimensions.
Therefore, $\mathfrak{g}^\R$ must be the split real form $\mathfrak{sl}_3\R$, whose maximal compact subalgebra $\mathfrak{so}_3\R$ is three-dimensional. Meanwhile, the $(+1)$-eigenspace of $\theta_\rho$ is trivial, hence the maximal compact subalgebra $\mathfrak{c}^\R$ of $\mathfrak{g}_\mathrm{Cay}^\R$ is trivial and the Cayley real form $\mathfrak{g}_\mathrm{Cay}^\R=\R^2$ is a real form of $\mathfrak{g}_0=\C^2$.
\end{ex}

The classification of magical triples are presented using weighted Dynkin diagrams, which is a complete invariant of $\mathfrak{sl}_2$-triples of simple complex Lie algebras \cite[Theorem 3.5.4]{CollingwoodMcGovern2017}. The weight of each node is given by $\alpha(h)$, where $\alpha$ is the corresponding simple root. One observes from the classification list that the weighted Dynkin diagram of any magical $\mathfrak{sl}_2$-triples consists only of weights 0 and 2. Thus, every magical $\mathfrak{sl}_2$-triple is even, that is, $\mathrm{ad}_h$ has only even eigenvalues.

\begin{ex}
    The weighted Dynkin diagram associated to a principal $\mathfrak{sl}_2$ has each simple root labeled with $\alpha(h)=2$. In particular, the principal-$\mathfrak{sl}_2$ in $\mathfrak{sl}_3\C$ discussed in Example \ref{principal sl2} has weighted Dynkin diagram $A_2: \,\,\dynkin[labels*={2,2}] A2$.
\end{ex}

\begin{rem}
    The evenness of magical triples provides another simplification: while the general definition of $\mathrm{Lag}^L$ requires a choice of square root $L^{1/2}$ of $L$, the Lagrangians associated to even triples do not end up depending on this choice. 
\end{rem}

Let $\frakg = \frakh \oplus \frakm$ be the $(\pm 1)$-eigenspace decomposition for the magical involution $\sigma_\rho$, where $\mathfrak{h}$ is the $(+1)$-eigenspace and $\mathfrak{m}$ is the $(-1)$-eigenspace. Directly from the definition of $\sigma_\rho$, we get decompositions of the $(\pm 1)$-eigenspaces in terms of highest weight lines (cf. Equation \eqref{eq:g-highest-weight-line-decomposition}): 
\begin{align*}
  \frakh = \frakc \oplus \bigoplus_{j=1}^{N} \, \bigoplus_{k=1}^{m_j} \, \ad_f^{2k-1} \cdot V_{2m_j}
\end{align*}
and 
\begin{align*}
  \frakm = \bigoplus_{j=1}^{N} \, \bigoplus_{k=0}^{m_j} \, \ad_f^{2k} \cdot V_{2m_j}. 
\end{align*}
Applying $\mathrm{ad}_f$ to these decompositions, we see that $\ad_f(\frakm) = \bigoplus_{j} \bigoplus_{k=1}^{m_j} \ad_f^{2k-1} \cdot V_{2m_j}$ and $\ad_f^2(\frakm) = \bigoplus_{j} \bigoplus_{k=1}^{m_j} \ad_f^{2k} \cdot V_{2m_j}$. Writing $\mathfrak{v} := \oplus_j V_{2m_j}$, the decompositions above can be expressed as
\begin{align}
  \label{eq:h-ad-f-decomposition}
  \frakh = \frakc \oplus \ad_f(\frakm)
\end{align}
and
\begin{align}
  \label{eq:m-ad-f-decomposition}
  \frakm = \mathfrak{v} \oplus \ad_f^2(\frakm). 
\end{align}
Most importantly, we have an obvious isomorphism
\begin{align}
  \label{eq:ad-f-isomorphism}
  \ad_f: \ad_f(\frakm) \xrightarrow{\sim} \ad_f^2(\frakm). 
\end{align}

\subsection{The general Cayley correspondence of \cite{Cayley}} \label{subsect Cayley paper summary}
We recall here the statement of the general Cayley correspondence on the level of moduli spaces following \textit{op. cit}. We first review the construction of a certain subgroup $\tilde{G}^\R\subset G^\R$ and its Lie algebra $\tilde{\mathfrak{g}}^\R$ depending on a magical triple $\rho: \mathfrak{sl}_2 \to \g$. 

Note that $\g_0 \subset \g$, the 0 weight space of $\rho(h)$, is a reductive Lie subalgebra. We may thus consider its semisimple part
$$
    \tilde{\g} := \text{ semisimple subalgebra of } \g_0, 
$$
and write the decomposition
$$\mathfrak{g}_0= \tilde{\mathfrak{g}} \oplus \C^{r(\rho)}.$$ 
Then we define $\mathfrak{g}_{\mathrm{Cay}}^\R$ to be the following real form of $\mathfrak{g}_0$:
\begin{equation}
  \label{eq:cayley-real-form-abelian-semisimple}
\mathfrak{g}_\mathrm{Cay}^\R =  \tilde{\mathfrak{g}}^\R \oplus \R^{r(\rho)},
\end{equation}
where $\tilde{\mathfrak{g}}^\R$ is the real form of $\tilde{\mathfrak{g}}$ given by the restriction of the involution $\theta_\rho$. We define $\tilde{G}^\R$ to be the subgroup of $G$ with Lie algebra $\tilde{\g}^\R$.

\begin{thm}[General Cayley correspondence, \cite{Cayley}]  \label{thm GCC}
Let $\mathcal{M}_{L}(G^\R)$ denote the moduli space of stable $L$-twisted $G^\R$-Higgs bundles for a line bundle $L$ and a real Lie group $G^\R$. Given a magical $\mathfrak{sl}_2$-triple of $G$, there is an injective, open and closed map:
\begin{align*}
        \Psi_\rho:\, \mathcal{M}_{K^{m_c+1}}(\tilde{G}^\R)\times\prod_{j=1}^{r(\rho)}\mathcal{M}_{K^{l_j+1}}(\R^+) \,\,&\longrightarrow \,\,\mathcal{M}(G^\R).
    \end{align*}
where $m_c$ is the unique repeated weight in the $\rho$-decomposition of $\mathfrak{g}$ and $l_j$ are the rest of the weights $n_j$ (see Lemma 5.7 of \textit{op. cit} for details).
\end{thm}
In Table \ref{tab:representations_C} of the appendix, we tabulate a concrete description of the moduli space $\mathcal{M}_{K^{m_c+1}}(\tilde{G}^\R)$: it is the moduli space of stable pairs $(E,\phi)$ where $E$ is a principal $C$-bundle and $\phi$ is a $\mathfrak{v}_0$-valued $K^{m_c+1}$-twisted Higgs field. In fact, applying the principle of Theorem \ref{thm GR Higgs bundles} we may equivalently define 
\begin{equation} \label{eq gtilde is c plus v0}
    \tilde{\g} = \left(\mathfrak{c} = \tilde{\g}^{\theta_\rho = 1}\right) \oplus \left(\mathfrak{v}_0 = \tilde{\g}^{\theta_\rho=-1}\right).
\end{equation}

\subsection{Two Hamiltonian $G$-spaces} \label{subsect two Hamiltonian actions}

Suppose we are given a magical triple $\rho: \mathfrak{sl}_2 \to \g$ (recall that magical triples are even, so we use the convention $n_j = 2m_j$ as in Section \ref{subsubsect JM data}). We will use $\rho$ to construct two closely related Hamiltonian $G$-spaces:
\begin{itemize}
    \item Consider the real form $G^\RR_\rho$ defined by $\rho$, and let $H^\RR \subset G^\RR_\rho$ be its maximal compact subgroup. Let $H \subset G$ be its complexification, and we define 
    $$
    M_\rho := \mathrm{WInd}_H^G(\mathrm{pt}) \simeq \mathfrak{m} \times^H G$$
    with $\Ggr$ scaling the cotangent fiber (i.e, $\mathfrak{m}$) by weight 2, and moment map
   $$
        \mathfrak{m} \times^H G \ni (X,g) \longmapsto \mu(X,g) := g^{-1}Xg
    $$
    \item Consider the centralizer subgroup $C_\rho \subset G$ of $\rho$ and the intersection $C := C_\rho \cap H$. Let $U \subset G$ be the unipotent subgroup associated to $\rho$ as in Section \ref{subsubsect JM data}. Viewing $f \in \mathfrak{u}^*$, we consider the Whittaker induction 
    $$
        M'_\rho := \mathrm{WInd}_{C,\rho}^G(\mathrm{pt})\simeq \big(f+  \oplus_j \, V_{2m_j}\big) \times^{C} G.
   $$
   Recall that the $\Ggr$-action is defined by weight $2+2m_j$ on $V_{2m_j}$ and left multiplication by $\mathrm{exp}(\rho)$ on $G$. 
\end{itemize}

We saw in Theorem \ref{thm GR Higgs bundles} that $\mathrm{Lag}(M_\rho)$ is the moduli stack of $G^\RR_\rho$-Higgs bundles when $L = K$. As an immediate consequence of Theorem \ref{thm CS moduli stack}, the moduli interpretation of $\mathrm{Lag}^L(M'_\rho)$ is clear as well: we simply set the $C$-Higgs field to 0. 
\begin{prop}
    The $\CC$-points of $\mathrm{Lag}^L(M'_\rho)$ classify tuples $(F, \{\psi_{2m_j}\})$ where 
    \begin{itemize}
        \item $F$ is a $C$-bundle (whose induction to $C_\rho$ we denote by $F' = \mathrm{Ind}_C^{C_\rho}(F)$), and
        \item $\psi_{2m_j}$ are sections of $H^0(\Sigma, (V_{2m_j})_{F'} \otimes L^{m_j+1})$ as in Collier--Sanders' Slodowy category. 
    \end{itemize}
\end{prop}

\begin{rem}
    Note that $M'_\rho$ is obtained from $M_{\mathrm{CS}}$ by symplectic reduction by $C$; this can be viewed as a $C$-equivariant Lagrangian correspondence
   $$ \begin{tikzcd}
	\mathcal{L} := \bigg({\mathrm{pt}} & {0 \times_{\mathfrak{c}^*} M_{\mathrm{CS}}} & {M'_\rho}
	\arrow[from=1-2, to=1-1]
	\arrow[from=1-2, to=1-3]
\end{tikzcd}\bigg)$$
Under the functor $A_{\mathrm{Dol}}$ discussed in Remarks \ref{rem GL nonsense 1} and \ref{rem GL nonsense 2}, we see that $A_{\mathrm{Dol}}(\mathcal{L})$ represents a morphism of Lagrangians, sending $(F,\{\psi_{2m_j}\}) \in \mathrm{Lag}(M'_\rho)$ to $(F, 0, \{\psi_{2m_j}\}) \in \mathrm{Lag}(M_{\mathrm{CS}})$. The image consists of $C_\rho$-Higgs bundles with vanishing Higgs field ($K$-twisted), whose underlying bundles have structure group reduced to $C$. These are precisely the conditions in \cite[Lemma~5.6]{Cayley} which need to be satisfied for the Cayley correspondence to produce $G^\RR_\rho$-Higgs bundles.
\end{rem} 

The key computation realizing the Cayley correspondence as a morphism of Lagrangians in the sense of Section \ref{subsect 2functoriality} is encapsulated in the following 
\begin{prop} \label{prop morphism of actions}
    Consider the morphism 
    \begin{equation} \label{eqn map before quotient}
        \phi: (f +\oplus_{j} \, V_{2m_j}) \times G \longrightarrow \mathfrak{m} \times G
      \end{equation}
    defined by inclusion in the first coordinate and identity on the $G$-coordinate. Then $\phi$ descends to a well-defined morphism
    $$
        \phi: M'_\rho \longrightarrow M_\rho
    $$
    which is a $G \times \Ggr$-equivariant morphism of Hamiltonian actions.
\end{prop}
\begin{proof}
    Let $v_j \in V_{2m_j}$ (where, by our convention $m_j > 0$) and $g \in G$, then we must verify the $G \times \Ggr$-equivariance of the formula \eqref{eqn map before quotient}. The $G$-equivariance is clear, so we check, for $t \in \Ggr$, the equation
    \begin{equation} \label{eqn check equivariance}
        t \cdot (f + v_j,g) = (f +t^{2+2m_j}v_j, g) \overset{\phi}{\longmapsto} (t^2 (f + v_j), \exp(\rho)(t)g).
    \end{equation}
    But note that $\exp(\rho)(t)$ is in $H$, so we can rewrite the right hand side of \eqref{defn Gaiotto Lagrangian} as
    $$
        (t^2(f + v_j), \exp(\rho)(t)g) = (t^2\exp(\rho)(f + v_j),g) = (f + t^{2+2m_j}v_j,g)
    $$
    by definition of $V_{2m_j}$, so \eqref{eqn check equivariance} is verified. The fact that $\phi: M'_\rho \to M_\rho$ commutes with moment maps is immediate, since the moment maps are identified with (co)adjoint action. 
\end{proof}

\subsection{Cayley correspondence as a morphism of Lagrangians} \label{subsect Cayley morphism}

In this section we extend the Cayley correspondence of \cite{Cayley} to a morphism of moduli stacks over $L$-twisted Higgs bundles. We also deduce geometric features of this morphism via simple calculations on Hamiltonian actions. 

\begin{thm} \label{thm main 1}
    Let $\rho$ be a magical triple for $G$, defining a real form $G_\rho^\RR$ of $G$. Consider the morphism $\phi: M'_\rho \to M_\rho$ of Hamiltonian actions defined in Proposition \ref{prop morphism of actions}. The induced morphism of Gaiotto's Lagrangians
    $$
        \mathrm{Lag}(\phi): \mathrm{Lag}^L(M'_\rho) \longrightarrow \mathrm{Lag}^L(M_\rho) \simeq \mathrm{Higgs}_{G^\RR_\rho}^L
    $$
    over $\mathrm{Higgs}_G^L$ in the case when $L = K$ extends the Cayley correspondence (in the sense of Equation (5.5) and Lemma 5.6 of \cite{Cayley}) upon evaluation on $\CC$-points. 
\end{thm}
\begin{proof}
Recall that $\Theta=(E_0,\phi_0)$ is the uniformizing Higgs bundle from \eqref{eqn uniformizing bundle}. The global moment map of $\mathrm{Lag}^L(M_\rho')$ is given by the formula
$$
    \mu_{M_\rho'}: (F,\{\psi_{2m_j}\}) \longmapsto \left(E := \mathrm{Ind}_{C \times S_\rho}^G(F \times E_0), \phi_0+\sum_j \psi_{2m_j}\right),
$$
which is the Slodowy functor $\mathrm{Slo}_\Theta$ from Theorem \ref{thm:Theta-Slodowy-equiv-cat} restricted to $\mathrm{Lag}(M'_\rho)$. This is precisely Equation (5.5) and Lemma 5.6 of \cite{Cayley}, where we note that since $\rho$ is magical, $C \times S_\rho$ is a subgroup of $H$ (possibly modulo the diagonal $\pm 1$ in the case when $S_\rho \simeq \mathrm{SL}_2$), so $E$ admits naturally an $H$-reduction $E' := \mathrm{Ind}_{C \times S_\rho}^H(F \times E_0)$. Furthermore, since $f + \g_e \subset \mathfrak{m}$, the section $\phi_0 + \sum_j \, \psi_{2m_j}$ is correspondingly contained in $\mathfrak{m}_E \otimes L \subset \g_E \otimes L$. Thus, the global moment map of $\mathrm{Lag}^L(M_\rho')$ factors through $\mathrm{Lag}(\phi)$ via
\begin{align*}
    (F,\{\psi_{2m_j}) &\overset{\mathrm{Lag}(\phi)}{\longmapsto} \left(E', \phi_0+\sum_j \psi_{2m_j} \in H^0(\Sigma,\mathfrak{m}_E \otimes L)\right)\\
    &\overset{\mu_{M_\rho}}{\longmapsto} \left(E, \phi_0+\sum_j \psi_{2m_j} \in H^0(\Sigma,\g_E \otimes L)\right)
\end{align*}
as we wanted to show.
\end{proof}
\begin{rem}
    Strictly speaking, in \cite{Cayley} there was a restriction that the genus of $\Sigma$ is at least 2. This restriction would prevent the moduli stacks from acquiring further derived and stacky structures (see Appendix \ref{appendix}), but from our perspective it is natural to treat curves of all genera at this stage.
\end{rem}

Because of Theorem \ref{thm main 1}, we write
$$
    \mathrm{Cay}_{G,\rho}^L := \mathrm{Lag}^L(M'_\rho)
$$
and refer to it as the \textit{($L$-twisted) Cayley space}. We also refer to $\mathrm{Lag}(\phi)$ as the \textit{($L$-twisted) Cayley morphism}, reserving the term \textit{Cayley correspondence} for the construction studied in \cite{Cayley}.

At the level of moduli spaces, the Cayley correspondence of \textit{op. cit} was shown to be an open and closed map, which allows one to identify special components of the moduli space of stable $G^\RR$-Higgs bundles as Cayley components. In the rest of this section, we show that the natural generalizations of these geometric properties also hold for the Cayley morphism. 

\begin{thm} \label{thm main 2}
    The Cayley morphism $\mathrm{Lag}(\phi): \mathrm{Cay}_{G,\rho}^L \to \mathrm{Higgs}_{G_\rho^\RR}^L$ induces an equivalence on tangent complexes
    $$
        d\mathrm{Lag}(\phi): \mathbf{T}_{\mathrm{Cay}_{G,\rho}^L} \overset{\simeq}{\longrightarrow} \mathrm{Lag}(\phi)^* \, \mathbf{T}_{\mathrm{Higgs}^L_{G_\rho^\RR}}.
    $$
\end{thm}
\begin{proof}
    We first show that the morphism $\phi_{/G}: [M_\rho'/G] \to [M_\rho/G]$ induces an isomorphism on tangent complexes
    $$
        d\phi_{/G}: \mathbf{T}_{[M_\rho'/G]} \overset{\simeq}{\longrightarrow} \phi_{/G}^* \, \mathbf{T}_{[M_\rho/G]}.
    $$
    Write $\mathfrak{v} = \oplus_j \, V_{2m_j}$. Note that the quotient stacks $[M_\rho'/G] \simeq [f+
    \mathfrak{v}/C]$ and $[M_\rho/G] \simeq [\mathfrak{m}/H]$ are both vector spaces quotiented by a reductive group, so the morphism of tangent complexes induced by $\phi_{/G}$ at a point $f+v \in f+\mathfrak{v}$ is just the morphism of complexes concentrated in degrees $[-1,0]$
    \begin{equation} \label{eq morphism of tangent complex at a point}
        d\phi_{/G,f+v}: \big[\mathfrak{c} \overset{X \mapsto f+\mathrm{ad}_X(v)}{\longrightarrow} f+\mathfrak{v}\big] \longrightarrow \big[\mathfrak{h} \overset{X \mapsto \mathrm{ad}_X(v)}{\longrightarrow} \mathfrak{m}\big]
    \end{equation}
    where we recall that $\mathrm{Lie}(C) \simeq \mathfrak{c}$ is a naturally a Lie subalgebra of $\mathfrak{h}$. Note that by \eqref{eq:h-ad-f-decomposition} and \eqref{eq:m-ad-f-decomposition}, the morphism \eqref{eq morphism of tangent complex at a point} takes the block form
    \begin{equation} \label{eq morphism of tangent complex at a point 2}
        d\phi_{/G,f+v}: \big[\mathfrak{c} \to f+\mathfrak{v}\big] \longrightarrow \big[\mathfrak{c} \oplus \mathrm{ad}_f(\mathfrak{m}) \overset{(\ast)}{\to} (f+\mathfrak{v}) \oplus \mathrm{ad}^2_f(\mathfrak{m})\big]
    \end{equation}
    where $(\ast)$ is the morphism
    \begin{equation}
        (X,m) \longmapsto (f+ \mathrm{ad}_X(v)) \oplus \mathrm{ad}_f(m).
    \end{equation} 
    Since $\mathrm{ad}_f: \mathrm{ad}_f(\mathfrak{m}) \to \mathrm{ad}^2_f(\mathfrak{m})$ is an isomorphism \eqref{eq:ad-f-isomorphism}, we conclude that \eqref{eq morphism of tangent complex at a point} is a quasi-isomorphism.
    
    Knowing that $d\phi_{/G}$ induces an isomorphism of tangent complexes, the fact that $d\mathrm{Lag}(\phi)$ also does is immediate: we note that the tangent complex of $\mathrm{Cay}_{G,\rho}^L$ (resp. $\mathrm{Higgs}_{G_{\rho}^\RR}^L$) at a point $s$ is given by $\mathbf{T}_s \simeq \mathbf{H}^*(\Sigma, s^*\mathbf{T}_{[M_\rho'/G]} \otimes L^{1/2})$ (resp. $\mathbf{T}_s \simeq \mathbf{H}^*(\Sigma, s^*\mathbf{T}_{[M_\rho/G]} \otimes L^{1/2})$), and the differential $d\mathrm{Lag}(\phi)$ at $s$ is induced by $d\phi_{/G}$. Thus, we see by base change that the relative tangent complex of $\mathrm{Lag}(\phi)$ at $s$ is given by
    $$
        \mathbf{T}_{\mathrm{Lag}(\phi),s} =\mathbf{H}^*(\Sigma, \mathrm{Cone}(d\phi_{/G})) = \mathbf{H}^*(\Sigma, \mathbf{T}_{\phi_{/G}}) \simeq 0
    $$
    hence the claimed equivalence.
\end{proof}

Towards showing that the Cayley correspondence $\mathrm{Lag}(\phi): \mathrm{Cay}_{G,\rho}^L \to \mathrm{Higgs}_{G_\rho^\R}^L$ is also closed in a certain sense, recall that a quasi-compact morphism $f: \mathfrak{X} \to \mathfrak{Y}$ is \textit{universally closed} if, for every DVR $R$ with fraction field $K$ and for every solid commutative diagram 
\begin{equation} \label{eqn valuative criterion}
    \begin{tikzcd}
	{\spec K} & {\mathfrak{X}} \\
	{\spec R} & {\mathfrak{Y}}
	\arrow[from=1-1, to=1-2]
	\arrow[from=1-1, to=2-1]
	\arrow["f", from=1-2, to=2-2]
	\arrow[dashed, from=2-1, to=1-2]
	\arrow[from=2-1, to=2-2]
\end{tikzcd}
\end{equation}
there is a dotted arrow making the diagram commute. We call the existence of this dotted arrow the \textit{existence part} of the valuative criterion. 

We first deduce an intermediate result which is generally useful, regarding mapping stacks with target a closed immersion. 

\begin{lem} \label{lem maps into closed immersion}
    Let $i: \mathfrak{X} \to \mathfrak{Y}$ be a closed immersion of algebraic stacks, and let $\Sigma$ be a smooth projective curve. Then the morphism
    $$
        f = ( i \circ \, \cdot \, ) :\mathrm{Map}(\Sigma, \mathfrak{X}) \longrightarrow \mathrm{Map}(\Sigma, \mathfrak{Y})
   $$
    is universally closed.
\end{lem}
\begin{proof}
   Consider a DVR $R$ whose fraction field we denote by $K = \mathrm{Frac}(R)$. Consider an $R$-point $y \in \mathrm{Map}(\Sigma, \mathfrak{Y})$ and a $K$-point $x \in \mathrm{Map}(\Sigma, \mathfrak{X})$ which leads to the following commutative diagram
    \begin{equation}
        \begin{tikzcd}
	{\Sigma_K} & {\mathfrak{X}_K} & {\mathfrak{X}_R} \\
	{\Sigma_R} && {\mathfrak{Y}_R}
	\arrow["x", from=1-1, to=1-2]
	\arrow[hook, from=1-1, to=2-1]
	\arrow[hook, from=1-2, to=1-3]
	\arrow["{i_R}", from=1-3, to=2-3]
	\arrow["y", from=2-1, to=2-3]
\end{tikzcd}
    \end{equation}
    Consider an affine open cover $\mathcal{U} = \{U \to \Sigma\}$ of the curve, whose elements' base changes to $R$ and $K$ will be denoted by $U_R$ and $U_K$, respectively. Since $i_R$ is a closed immersion, by \cite[\href{https://stacks.math.columbia.edu/tag/01KE}{Tag 01KE}]{stacks-project} we see that upon possibly replacing $U_R$ by an \'etale cover, $y|_{U_R}$ can be lifted to a morphism $\widetilde{y}_{U_R}: U_R \to \mathfrak{X}_R$ whose restriction to $U_K$ coincides with $x|_{U_K}$. Furthermore, since $i_R$ is a monomorphism on functor of points, from the commutative diagram
    \begin{equation}
        \begin{tikzcd}
	{\mathfrak{X}_R(U_R)} & {\mathfrak{Y}_R(U_R)} \\
	{\mathfrak{X}_R((U \times_\Sigma U')_R)} & {\mathfrak{Y}_R((U \times_\Sigma U')_R)} \\
	{\mathfrak{X}_R(U'_R)} & {\mathfrak{Y}_R(U'_R)}
	\arrow[hook, from=1-1, to=1-2]
	\arrow[from=1-1, to=2-1]
	\arrow[from=1-2, to=2-2]
	\arrow[hook, from=2-1, to=2-2]
	\arrow[from=3-1, to=2-1]
	\arrow[hook, from=3-1, to=3-2]
	\arrow[from=3-2, to=2-2]
\end{tikzcd}
    \end{equation}
    formed by any two pieces of the cover $U, U' \in \mathcal{U}$, we see that \begin{equation} \label{eqn monomorphism}
        \widetilde{y}_{U_R}|_{(U \times U')|_R} \simeq \widetilde{y}_{U'_R}|_{(U \times U')|_R}.
    \end{equation} 
    The cocycle condition on triple intersection for $\{\widetilde{y}_{U_R}\}_{U \in \mathcal{U}}$ is implied by the cocycle condition on triple intersections for the $\{y|_{U_R}\}_{U \in \mathcal{U}}$, the latter collection of which glues to the morphism $y$ which we started with, along with an analogous diagram as \eqref{eqn monomorphism} involving triple intersections. By descent along the \'etale cover $\mathcal{U}_R = \{U_R\} \to \Sigma$, we obtain a morphism $\widetilde{y}: \Sigma_R \to \mathfrak{X}_R$ extending $x$, as we wanted to show. 
\end{proof}

 We adapt Balaji--Seshadri and Balaji--Parameswaran's proof of the semistable reduction theorem for principal bundles for the following lemma (cf. \cite[Proposition~2.8]{BalajiSeshadri2002}, \cite[Proposition~3]{BalajiParameswaran2003}).
 
 \begin{lem} \label{lem reduction of structure group}
    Let $G$ be a reductive group, and let $H \subseteq G$ be a (not necessarily connected) reductive subgroup, such that the component group $\pi_0(H)$ is abelian. Let $\Sigma$ be a smooth projective curve. Then the natural map
    \begin{align*}
      \Bun_H \longrightarrow \Bun_G
    \end{align*}
    is universally closed. 
  \end{lem}

  \begin{proof}
    Towards checking the existence part of the evaluative criterion, we let $R$ be an DVR and $K$ its fraction field. Suppose we have a family of $G$-bundles $\calE_R \to \Sigma_R$ on $\Sigma$ parametrized by $\spec R$, with an $H$-reduction over $\Sigma_K$; we write $\mathcal{F}_K$ to denote the resulting principal $H$-bundle over $\Sigma_K$. The reduction can be encoded as follows: write $X = G/H$, then we have a section $s_K: \Sigma_K \to (X_{\calE})_K$. We would like to extend $s_K$ to a section of $(X_{\calE})_R$ over $\Sigma_R$. 

    Fix once and for all an auxiliary open $U \subset \Sigma$ whose complement consists of one point, and whose base change to $\spec R$ and $\spec K$ we denote by $U_R \subset \Sigma_R$ and $U_K \subset \Sigma_K$, respectively. By a theorem of Drinfeld--Simpson \cite[Theorem~3]{DrinfeldSimpson1995}, upon replacing $\spec R$ by a finite \'etale cover we may assume that $\mathcal{F}|_{U_K}$ and $\mathcal{E}|_{U_R}$ are trivial torsors. Thus, the bundle $X_{\calE}|_{U_R}$ is also trivialized as $X \times U_R$ and we may extend $s_K$ to a section over the open set
    $$
        s': V := U_R \cup \Sigma_K \longrightarrow X_\mathcal{E}|_V
    $$
    by gluing together $s$ over $\Sigma_K$ with the trivial extension over $U_R$. In other words, we have an $H$-reduction of $\mathcal{E}|_V$ to the $H$-bundle on $V$ glued together from $\mathcal{F}_K$ on $\Sigma_K$ and the trivial bundle on $U_R$. Since $V \subset \Sigma_R$ contains all primes of height 1, if $H$ were geometrically connected, by \cite[Theorem~6.13]{Colliot-TheleneSansuc1979} we have $H^1(V,H) \simeq H^1(\Sigma_R,H)$, and thus the $H$-reduction extends to $\Sigma_R$. 

    To reduce to the connected case, we consider the connected-\'etale exact sequence of group schemes
    $$
        1 \longrightarrow H_0 \longrightarrow H \longrightarrow \pi_0(H) \longrightarrow 1
    $$
    which induces a morphism of long exact sequence in cohomology over $\Sigma_R$ and $V$:
    $$
        \begin{tikzcd}
	\cdots & {H^1(\Sigma_R,H_0)} & {H^1(\Sigma_R, H)} & {H^1(\Sigma_R, \pi_0(H))} & \cdots \\
	\cdots & {H^1(V, H_0)} & {H^1(V, H)} & {H^1(V, \pi_0(H))} & \cdots
	\arrow[from=1-1, to=1-2]
	\arrow[from=1-2, to=1-3]
	\arrow["\simeq", from=1-2, to=2-2]
	\arrow[from=1-3, to=1-4]
	\arrow[from=1-3, to=2-3]
	\arrow[from=1-4, to=1-5]
	\arrow[from=1-4, to=2-4]
	\arrow[from=2-1, to=2-2]
	\arrow[from=2-2, to=2-3]
	\arrow[from=2-3, to=2-4]
	\arrow[from=2-4, to=2-5]
\end{tikzcd}
    $$
 where the isomorphism of the left vertical arrow is an application of the previously mentioned purity theorem of Colliot-Thélène--Sansuc. What we need, i.e., that the middle vertical arrow is an isomorphism, would thus follow from the right vertical arrow $\alpha: H^1(\Sigma_R, \pi_0(H)) \to H^1(V, \pi_0(H))$ being an isomorphism. This latter can be deduced by a standard argument using Kummer sequences: reducing to the case when $\pi_0(H) \simeq \mu_n$ (which is possible since we assume that $\pi_0(H)$ is abelian), we have
$$
 H^1(\Sigma_R, \mu_n) \simeq H^2(\Sigma_R, \Gm)[n] \simeq H^2(V, \Gm)[n] \simeq H^1(V, \mu_n)
$$
 where in the middle we apply again the purity theorem of \textit{loc. cit}.
  \end{proof}

\begin{thm} \label{thm main 3}
    The Cayley morphism
    $$
        \mathrm{Lag}(\phi): \mathrm{Cay}^L_{G,\rho} \longrightarrow \mathrm{Higgs}^L_{G_\rho^\RR}
    $$
    is universally closed. 
\end{thm} 
\begin{proof}
    Note that for any graded Hamiltonian $G$-action $M$, Gaiotto's Lagrangian $\mathrm{Lag}^L(M)$ is the fiber over $L^{1/2}$ of the morphism $\mathrm{Map}(\Sigma, [M/G \times \Ggr]) \to \mathrm{Map}(\Sigma, B\Ggr) = \mathrm{Pic}$; thus, to show that $\mathrm{Lag}(\phi)$ is closed it suffices to show that the morphism 
    \begin{equation} \label{eqn phi tilde}
        \widetilde{\phi}: L_1 := \mathrm{Map}(\Sigma, [M_\rho'/G \times \Ggr]) \longrightarrow L_2 := \mathrm{Map}(\Sigma, [M_\rho/G \times \Ggr])
    \end{equation}
    is universally closed. Now the morphism of targets of these mapping stacks $[M_\rho'/G \times \Ggr] \simeq [f+V_\mathfrak{m}/C\times \Ggr] \to [M_\rho'/G \times \Ggr] \simeq [\mathfrak{m}/H \times \Ggr]$ which induces $\widetilde{\phi}$ can be written as a composition
    $$
        [f+V_\mathfrak{m}/C \times \Ggr] \overset{i}{\longrightarrow} [\mathfrak{m}/C \times \Ggr] \overset{j}{\longrightarrow} [\mathfrak{m}/H \times \Ggr]
    $$
    which induces a factorization of $\widetilde{\phi}$ into
    $$
        \widetilde{\phi}: L_1 \overset{I}{\longrightarrow} L_3 := \mathrm{Map}(\Sigma, [\mathfrak{m}/C \times \Ggr]) \overset{J}{\longrightarrow} L_2.
    $$
    Note that $i$ is a closed immersion, which can be checked on the atlas $\mathfrak{m} \to [\mathfrak{m}/C \times \Ggr]$; by Lemma \ref{lem maps into closed immersion}, the induced map $I: L_1 \to L_3$ is universally closed. 

    It remains to show that $J$ is a universally closed morphism as well. Given a diagram \eqref{eqn valuative criterion} with $\mathfrak{X} = L_3$ and $\mathfrak{Y} = L_2$, we see that the universal closedness of $J$ is equivalent to the following statement: given a $C \times \Ggr$-bundle $\mathcal{E}_K$ on $\Sigma_K$ whose induction $\mathcal{E}^H_K := \mathcal{E}_K \times^{C} H$ extends to an $H \times \Ggr$-bundle $\mathcal{E}^H_R$ over $\Sigma_R$ (after possibly replacing $R$ by an \'etale cover), we can find a $C \times \Ggr$-reduction of $\mathcal{E}^H$ over $\Sigma_R$ whose restriction over $\Sigma_K$ is exactly $\mathcal{E}_K$; in other words, we need only apply Lemma \ref{lem reduction of structure group}.

    In order to apply Lemma \ref{lem reduction of structure group}, we need to verify that $\pi_0(C)$ is abelian. To this end we view $C = (\tilde{G})^{\theta_\rho}$ as the fixed point set of the complexified Cartan involution acting on the semisimple part of the complexified Cayley real form (see Definition~\ref{def:cayley-real-form} and Equation~\eqref{eq:cayley-real-form-abelian-semisimple}). By Steinberg’s theorem (\cite[Theorem~9.1]{Steinberg1968}, \cite[Proposition~1.27]{DigneMichel1994}), since $\theta_\rho$ preserves a Borel subgroup and maximal torus, $\pi_0(C)$ is either trivial or an abelian $2$-group. 
\end{proof}

We comment that, contrary to the Cayley correspondence of \cite{Cayley}, the Cayley morphism $\mathrm{Lag}(\phi): \mathrm{Cay}^L_{G,\rho} \to \mathrm{Higgs}_{G^\rho_\R}$ has some (expected) pathological properties. For instance, it is easy to see from the proof of Theorem \ref{thm main 3} that $\mathrm{Lag}(\phi)$ is not separated in general, as $C$ is a proper subgroup of $H$.

\subsection{Factorizations of the Cayley morphism}
\label{subsect factorization}

In some sporadic cases, two Cayley morphisms can be factored through one another when $\frakg$ admits two distinct $G$-conjugacy classes of magical triples (cf. \cite[Remark~7.5]{Cayley}). In this Section we explain this phenomenon using the perspective of a composition of morphisms of Hamiltonian spaces. 

We may notice that the list of canonical real forms associated to a magical triple in Theorem~\ref{thm:classification-canonical-real-forms} is not mutually exclusive. There are four special families of canonical real forms that appear twice in the four cases of \textit{loc. cit}: the real forms
\begin{align*}
  \frakg^\R = \mathfrak{sp}_{2n}\R, \; \mathfrak{so}_{n,n+1}, \; \mathfrak{so}_{n,n}, \; \mathfrak{f}_4^4
\end{align*}
are both split real of case (1), and of case (2), (3), (3), (4) respectively in the classification.

In particular, in these cases $\frakg$ admits both a principal triple, which we denote by $\rho_1$, and a non-principal magical triple, which we denote by $\rho_2$. In the following discussion we use a subscript $i$ to label the data associated to the triple $\rho_i$. 

Assume $\frakg^\R$ is one of the three special families as above. We may fix the triples $\rho_1, \rho_2$ such that $\frakg^\R=\frakg^\R_1=\frakg^\R_2$. Then the semisimple part $\tilde{\frakg}^\R_2 \subset \frakg_{\mathrm{Cay},2}^\R$ of the Cayley real form is a split real form. That is, we may view $\tilde{\frakg}^\R_2$ as the canonical real form associated to a principal triple $\rho_3$ in $\tilde{\frakg}_2$, and the Cayley space $\mathrm{Cay}_{{G}_2,\rho_2}$ admits a Cayley map as well. To see this, we analyze the Slodowy slices $\mathcal{S}_{\rho_i} = \rho_i(f) + \frakg_e$. 

Let $\frakg(e_2) := Z_{\frakc_2^\perp}(\frakc_2)$ be the double centralizer of $\rho_2$ complementary to the centralizer $\frakc_2$ of $\rho_2$. In other words, we have
\begin{align*}
  Z_\frakg(Z_\frakg(\rho_2)) = Z_\frakg(\frakc_2) = Z_{\frakc_2}(\frakc_2) \oplus \frakg(e_2). 
\end{align*}
Then $\mathrm{Im}(\rho_2) \subset \frakg(e_2)$ is a principal triple of $\frakg(e_2)$. In fact, $\frakg(e_2)^\R$ is the $\Theta_2$-principal subalgebra in $\frakg^\R$ of Guichard--Wienhard \cite{GuichardWienhard2024}.

The Slodowy slices decompose into
\begin{align*}
  \mathcal{S}_{\rho_i} = S_{\rho_i} \cap \frakg(e_2) \oplus \mathcal{S}_{\rho_i} \cap \frakg(e_2)^\perp. 
\end{align*}
First note that we may restrict $\rho_1$ to a principal triple in $\frakg(e_2)$. Since the principal triple is unique up to conjugation by $G(e_2)$, we may identify the Slodowy slices restricted to $\frakg(e_2)$:
\begin{align*}
  \iota: \mathcal{S}_{\rho_1} \cap \frakg(e_2) \simeq \mathcal{S}_{\rho_2} \cap \frakg(e_2). 
\end{align*}

According to \cite[Lemma~5.7]{Cayley}, there is a unique weight module $W_{2m_{c,2}}$ which intersects $\tilde{\frakg}_2$ nontrivially. Suppose
\begin{align*}
  \tilde{\frakg}_2 = \frakc_2 \oplus \frakm_{\mathrm{Cay},2}
\end{align*}
is the complexified Cartan decomposition with respect to $\theta_{e_2}$. Then we may identify
\begin{align*}
  \mathcal{S}_{\rho_2} \cap \frakg(e_2)^\perp \simeq \frakm_{\mathrm{Cay},2}^{K^{2m_{c,2}}}.
\end{align*}
Again $\rho_1$ restricts to a principal triple in $\tilde{\frakg}_2$, and we may identify
\begin{align*}
  \mathcal{S}_{\rho_1} \cap \frakg(e_2)^\perp \simeq \mathcal{S}_{\rho_3}, 
\end{align*}
where $\rho_3$ is the principal triple that gives rise to $\tilde{\g}^\R_2$ as a split real form.

Factorize the Cayley map $\phi_1$ for the principal triple $\rho_1$ into a $K^{2m_{c,2}}$-twisted Cayley map $\phi_3$ for the principal triple $\rho_3$ in $\tilde{g}_2$, and the Cayley map $\phi_2$ for the non-principal magical triple $\rho_2$ as follows: 
\begin{equation}
  \label{eq:factorize-Cayley-map}
  \begin{tikzcd}
    { \mathcal{S}_{\rho_1} = \mathcal{S}_{\rho_3} \oplus (\mathcal{S}_{\rho_1} \cap \frakg(e_2)) } && { \frakm_{\mathrm{Cay},2}^{K^{2m_{c,2}}} \oplus (\mathcal{S}_{\rho_2} \cap \frakg(e_2)) = \mathcal{S}_{\rho_2} } \\
    & { \frakm }
    \arrow["\phi_3^{K^{2m_{c,2}}} \oplus \iota", from=1-1, to=1-3]
    \arrow["{\phi_1}"', from=1-1, to=2-2]
    \arrow["{\phi_2}", from=1-3, to=2-2]
  \end{tikzcd}. 
\end{equation}
Note that by assumption, the magical involutions $\sigma_{e_1}$ and $\sigma_{e_2}$ agree, hence $\frakm = \frakm_1 = \frakm_2$ and $H = H_1 = H_2$. We also have $C_1 \subset C_2$ and the quotients are compatible.

\subsection{Good moduli spaces}\label{subsect good moduli spaces}

Let $\mathcal{M}(G)$ be the moduli space of stable $G$-Higgs bundles, i.e., the rigidification of the stable open substack in $\mathrm{Higgs}_G$. For applications to the geometry of the manifold $\mathcal{M}(G)$, it is important to remember that Gaiotto's Lagrangians are Lagrangian in a derived sense; in particular, they do not directly implicate the existence of corresponding Lagrangian submanifolds of $\mathcal{M}(G)$. 

In order to connect our Theorem \ref{thm main 1} to the moduli space-level Cayley correspondence of \cite{Cayley}, we need to analyze the stability and derived structures of the Cayley space $\mathrm{Cay}_{G,\rho}$. Since this discussion is slightly orthogonal to our exposition so far, we have deferred these arguments to Appendix \ref{appendix}, but we summarize the conclusion here in the following statement. 
\begin{thm}[\cite{Cayley}, Corollary \ref{cor classicality of Cayley morphism}] \label{thm main 4} 
    Suppose $g(\Sigma) \geq 2$ and $G$ is semisimple. Restricting to the stable moduli stack of $G$-Higgs bundles, the Cayley morphism of Theorem \ref{thm main 1} induces a morphism of moduli spaces over $\mathcal{M}(G)$. The image of this scheme-theoretic Cayley morphism is an isomorphism onto connected components, each of which maps quasi-finitely onto (possibly singular) Lagrangian submanifolds of $\mathcal{M}(G)$.
\end{thm}

We note that Theorem \ref{thm main 4} does not say anything about the components of $\mathrm{Higgs}_{G_\rho^\RR}$ that do not lie in the image of the Cayley morphism. It is entirely possible that these components are derived and thus not of expected (Lagrangian) dimension when considered classically.

\section{\texorpdfstring{$S$}{S}-duality and relative Langlands duality} \label{sect S-duality}

Recall that $S$-duality at the level of objects of $\mathfrak{B}_{\mathrm{GL}}$ is simply Langlands' duality of reductive groups, and it manifests in the Dolbeault setting as a Langlands duality between Hitchin systems of Langlands dual groups \cite{Donagi-Pantev}. The Cayley correspondence, being a morphism of (BAA)-branes, can thus be regarded as an example of a morphism of boundary conditions, subject to $S$-duality analysis in the mathematical formulation of relative Langlands duality \cite{BZSV} in the Dolbeault setting \cite{C}. In future work we hope to pursue this direction more systematically, while in this current section, we content ourselves with some precise expectations in one specific example, where the real form in question is Hitchin's $\mathrm{U}(n,n)$-brane \cite[Section 7]{Hitchin2013}. 

To explain the notion of $S$-duality in the Dolbeault context\footnote{The reader who would prefer to fast forward to the mathematically formulated statements can skip to Section \ref{subsect Sdual of Cayley} without losing much.} and its implications on the Cayley correspondence, we return to the setting described by Remarks \ref{rem GL nonsense 1}, \ref{rem AKSZ} and \ref{rem GL nonsense 2}. We expect a functor $B_{\mathrm{Dol}}$ (B-twist of Dolbeault geometric Langlands) on the category $\mathfrak{B}_{\mathrm{GL}}$ whose assignment to objects agrees with $A_{\mathrm{Dol}}$:
$$
    \mathfrak{B}_{\mathrm{GL}} \ni G \text{ reductive group } \overset{B_{\mathrm{Dol}}}{\longmapsto} \mathrm{QC}(\mathrm{Higgs}_G), 
$$
but whose evaluation on 1-morphisms gives hyperholomorphic sheaves, i.e., quasi-coherent (or in fact more conveniently, ind-coherent) sheaves on $\mathrm{Higgs}_G$ which remain holomorphic upon rotations of complex structures. A preliminary mathematical definition of such hyperholomorphic sheaves on $\mathrm{Higgs}_G$, or \textit{(BBB)-branes}, was introduced in \cite{Franco-Hanson} and in \cite{CF} we modify the definition to study a host of examples admitting gauge-theoretic description along the lines of the relative Langlands program. 

One expects that $S$-duality manifests as a hypothetical involution 
$$
    \check{( \, \cdot \, )}: \mathfrak{B}_{\mathrm{GL}} \longrightarrow \mathfrak{B}_{\mathrm{GL}}
$$
extending, at the object level, Langlands' duality of reductive groups $G \leftrightarrow \check{G}$. Furthermore, $S$-duality should intertwine the evaluation of the functors $A_{\mathrm{Dol}}$ and $B_{\mathrm{Dol}}$.

For instance, for an object $G \in \mathfrak{B}_{\mathrm{GL}}$ we obtain an equivalence of categories
\begin{equation} \label{equation Dolbeault Langlands TQFT}
    A_{\mathrm{Dol}}(G) \overset{?}{\simeq} B_{\mathrm{Dol}}(\check{G})
\end{equation}
where $\check{G}$ is the Langlands dual reductive group of $G$, and equation \eqref{equation Dolbeault Langlands TQFT} manifests mathematically as the Fourier--Mukai conjecture of \cite{Donagi-Pantev}. Digging deeper into the category of boundary conditions $\mathfrak{B}_{\mathrm{GL}}$, one may consider the consequences of $S$-duality at the level of 1-morphisms, i.e., Hamiltonian actions. The procedure of $S$-duality is generally not yet well-understood besides in the case of \textit{hyperspherical actions} in the sense of \cite{BZSV}, and via the formula proposed by \cite{Nakajima} building upon previous work on ring objects in the Satake category \cite{BFN}. In any case, we expect that the equivalence \eqref{equation Dolbeault Langlands TQFT} intertwines, for a Hamiltonian $G$-action $M$ and its $S$-dual $\check{G}$-Hamiltonian action $\check{M}$, a pair\footnote{In fact, since $S$-duality is agnostic to the A/B-twist constructions, we obtain a \textit{dual pair} of mirror (BAA)/(BBB)-branes
\begin{equation}
    A_{\mathrm{Dol}}(\check{G}) \ni A_{\mathrm{Dol}}(\check{M}) \leftrightarrow B_{\mathrm{Dol}}(M) \in B_{\mathrm{Dol}}(G).
\end{equation}} of mirror (BAA)/(BBB)-branes
\begin{equation}
    A_{\mathrm{Dol}}(G) \ni A_{\mathrm{Dol}}(M) \leftrightarrow B_{\mathrm{Dol}}(\check{M}) \in B_{\mathrm{Dol}}(\check{G}).
\end{equation}

Consider now the Hamiltonian $G$-spaces $M_\rho, M_\rho'$ arising from a magical triple $\rho$, as defined in Section \ref{subsect two Hamiltonian actions}, and the ``Cayley" morphism $\phi: M_\rho' \to M_\rho$ of Proposition \ref{prop morphism of actions}. Considering its graph, we have a 2-morphism
\begin{equation}
    M_\rho' \longleftarrow \phi\longrightarrow M_\rho 
\end{equation}
whose $S$-dual should be a 2-morphism 
\begin{equation}
    \check{M}_\rho \longleftarrow \check{\phi} \longrightarrow \check{M}_\rho'.
\end{equation}
Evaluating the $B$-twist functor $B_{\mathrm{Dol}}$ on the hypothetical $S$-dual 2-morphism $\check{\phi}$ to $\phi$, the latter of which induces the Cayley morphism $\mathrm{Lag}(\phi): \mathrm{Cay}_{G,\rho} \to \mathrm{Higgs}_{G_\rho^\RR}$, we expect a morphism of hyperholomorphic sheaves
\begin{equation} \label{eqn S dual conjecture}
    B_{\mathrm{Dol}}(\check{\phi}): B_{\mathrm{Dol}}(\check{M}_\rho)  \longrightarrow B_{\mathrm{Dol}}(\check{M}_\rho')
\end{equation}
over $\mathrm{Higgs}_{\check{G}}$ which is Fourier--Mukai dual to the Cayley morphism.\footnote{Note that, contrary to $A_{\mathrm{Dol}}$, the functor $B_{\mathrm{Dol}}$ is \textit{covariant} on 2-morphisms since we consider distributions instead of functions in \cite{CF}, following \cite{BZSV}'s conventions on $L$-sheaves. On the other hand, by our convention $S$-duality itself is \textit{contravariant} on 2-morphisms, so their composition is again \textit{contravariant}.} Note that $\check{\phi}$ need not arise as the graph of a morphism of Hamiltonian actions, but may take the form of a genuine equivariant Lagrangian correspondence. Our goal in the following is to precisely write down our expectation \eqref{eqn S dual conjecture} in terms of Conjecture \ref{conj 3} while proving the case when $G = \mathrm{PGL}_{2n}$ and $M_\rho$ gives rise to $\mathrm{PU}(n,n)$-Higgs bundles.

\subsection{Summary of Dolbeault geometric Langlands}

We review rapidly the most salient features of the Dolbeault geometric Langlands duality of \cite{Donagi-Pantev} in order to state our results precisely.

\subsubsection{Dual Hitchin fibrations}
Let $G, \check{G}$ be a pair of Langlands dual reductive groups. We consider the Langlands dual Hitchin morphisms
\begin{equation} \label{eqn dual Hitchin fibration}
\begin{tikzcd}
	{\mathrm{Higgs}_G} && {\mathrm{Higgs}_{\check{G}}} \\
	& {\mathcal{A} := \mathrm{Sect}(\Sigma, \mathfrak{c}_{K^{1/2}})}
	\arrow["\chi", from=1-1, to=2-2]
	\arrow["{\check{\chi}}"', from=1-3, to=2-2]
\end{tikzcd}
\end{equation}
where 
$$\mathfrak{c} = \g\sslash G \simeq \mathfrak{h}\sslash W \overset{\text{Killing}}{\simeq} \mathfrak{h}^*\sslash W \check{\g}\sslash \check{G}$$ 
are the Chevalley quotients of $G$ and $\check{G}$, which are identified by the Killing form. The fibers of \eqref{eqn dual Hitchin fibration} over an open dense subset $\mathcal{A}^\diamond \subset \mathcal{A}$, the complement of a certain discriminant locus, can be identified with torsors over dual Beilinson 1-motives over $\mathcal{A}^\diamond$. Upon choosing $K^{1/2}$ and a compatible pinning of $G$ and $\check{G}$, the associated Hitchin section (as in Example \ref{ex principal slice is Hitchin section}, with the principal element $e$ landing in the Borel subgroups specified by pinnings) trivializes these dual torsors over $\mathcal{A}^\diamond$. 

Here, and in the following, the superscript $( \, \cdot \, )^\diamond$ denotes restriction to the nice locus $\mathcal{A}^\diamond \subset \mathcal{A}$ of the common Hitchin base for $G$ and $\check{G}$ for spaces and quasi-coherent sheaves. 

\subsubsection{Fourier--Mukai duality}
Fourier--Mukai duality between Beilinson 1-motives then gives a derived equivalence
\begin{equation} \label{eqn FM equivalence}
    \mathbf{S}_{\mathrm{Dol}}: \mathrm{QC}(\mathrm{Higgs}_G^\diamond/\mathcal{A}^\diamond) \overset{\sim}{\longrightarrow} \mathrm{QC}(\mathrm{Higgs}_{\check{G}}^\diamond/\mathcal{A}^\diamond),
\end{equation}
which gives the most accessible part of the Dolbeault Langlands correspondence of Donagi--Pantev \cite{Donagi-Pantev}. We shall use the normalization in \cite{Schnell} (generalized routinely to the case of disconnected abelian varieties and Deligne--Mumford abelian stacks in the cases of interest), post-composed with an application of Grothendieck duality. In other words, we conjugate the original Fourier--Mukai duality of \cite{Mukai} by Grothendieck duality on the source and the target to obtain a \textit{covariant} duality\footnote{This is purely an aesthetic choice, since we work on the smooth $\diamond$-locus and there is no difference between quasi-coherent and ind-coherent sheaves. However, we opt for this description since the hyperholomorphic sheaves defined in \cite{CF} are most naturally ind-coherent.}, normalized by matching the skyscraper sheaf at the identity on one side with the dualizing sheaf on the other. 

Let $A_i, \check{A}_i$ be two pairs of dual Beilinson 1-motives, for $i = 1,2$. With our conventions, Fourier--Mukai duality $\mathbf{S}_i :\mathrm{QC}(A_i) \overset{\sim}{\to} \mathrm{QC}(\check{A}_i)$ satisfies the following relation with respect to a morphism $f: A_1 \to A_2$ of Beilinson 1-motives:
\begin{equation}
    \mathbf{S}_2 \circ f_* \simeq \check{f}^! \circ \mathbf{S}_1
\end{equation}
where $\check{f} : \check{A}_2 \to \check{A}_1$ is dual to $f$; this follows immediately from post-composing Proposition 4.1 of \cite{Schnell} with Grothendieck duality. In particular, stipulating that $\mathbf{S}_1(\omega_{A_1}) = \delta_0$ is the skyscraper sheaf at the origin of $\check{A}_1$, we can calculate
\begin{equation} \label{eqn Fourier-Mukai functoriality}
    (\mathbf{S}_2 \circ f_*)(\omega_{A_1}) = \check{f}^!(\delta_0) = \omega(\mathrm{ker}(\check{f})).
\end{equation}
where $\mathrm{ker}(\check{f})$ is the kernel of $\check{f}$ in the sense of Beilinson 1-motives. 

\subsubsection{Characterizing $\mathbf{S}_{\mathrm{Dol}}$}
Recall that the Fourier--Mukai duality $\mathbf{S}_{\mathrm{Dol}}$ is uniquely characterized by the following two specifications:
\begin{itemize}
    \item (Normalization). Trivialize $\mathrm{Higgs}_G^\diamond$ over $\mathcal{A}^\diamond$ as a Beilinson 1-motive, with the identity section given by the Hitchin section. Then the Hitchin section is sent to the relative dualizing sheaf of $\mathrm{Higgs}_{\check{G}}^\diamond$ over $\mathcal{A}^\diamond$ (equation (3.1) of \cite{Schnell}):
    \begin{equation} \label{eqn normalization of S-duality}
        \mathbf{S}_{\mathrm{Dol}}: \mathcal{O}(\text{Hitchin section})^\diamond \longmapsto \omega(\mathrm{Higgs}_{\check{G}}^\diamond/\mathcal{A}^\diamond).
    \end{equation}
    \item (Hecke--Wilson compatibility). Since we do not directly use this perspective, we give just a heuristic explanation. Complete details can be found in the theorem statements of Theorems A, B, C and Appendix A of \cite{Donagi-Pantev}. Let $\lambda$ be a miniscule dominant coweight of $G$, and $x \in \Sigma$ an arbitrary closed point. Then we have a moduli stack of Hecke correspondences
\begin{equation}
    \begin{tikzcd}
	& {\mathrm{Hecke}_{\lambda,x}} \\
	{\mathrm{Higgs}_G} && {\mathrm{Higgs}_G}
	\arrow["{h^\leftarrow}"', from=1-2, to=2-1]
	\arrow["{h^\rightarrow}", from=1-2, to=2-3]
\end{tikzcd}
\end{equation}
where $\mathrm{Hecke}_{\lambda,x}$ parametrizes a pair of Higgs bundles differing by a $\lambda$-modification at $x$, and $h^\leftarrow, h^\to$ are projections to the pre/post modification, respectively. We write $\mathbf{T}_{\lambda,x}$ for the Fourier--Mukai transform on $\mathrm{QC}(\mathrm{Higgs}_G^\diamond)$ induced by the structure sheaf on $\mathrm{Hecke}_{\lambda,x}$. 

The same label $\lambda$ can be interpreted as a dominant cocharacter of the Langlands dual group $\check{G}$ which gives rise to a highest weight irreducible representation $W_\lambda$ of $\check{G}$. The pair of $\lambda$ and the base point $x$ gives rise to a vector bundle on $\mathrm{Higgs}_{\check{G}}$ defined by pulling back $W_\lambda$ along the morphism
\begin{equation}
    \mathrm{ev}_x: \mathrm{Higgs}_{\check{G}} \longrightarrow B\check{G}
\end{equation}
and we write $\mathbf{W}_{\lambda,x}$ for the \textit{Wilson operator} on $\mathrm{QC}(\mathrm{Higgs}_{\check{G}}^\diamond)$ induced by tensorization with $\mathrm{ev}_x^* W_\lambda$. 

Then $\mathbf{S}_{\mathrm{Dol}}$ intertwines the Hecke and Wilson operators on the $\diamond$-locus labeled by the same data, i.e., the following diagram commutes:
\begin{equation} \label{eqn Hecke Wilson compatibility}
    \begin{tikzcd}
	{\mathrm{QC}(\mathrm{Higgs}_G^\diamond/\mathcal{A}^\diamond)} & {\mathrm{QC}(\mathrm{Higgs}_{\check{G}}^\diamond/\mathcal{A}^\diamond)} \\
	{\mathrm{QC}(\mathrm{Higgs}_G^\diamond/\mathcal{A}^\diamond)} & {\mathrm{QC}(\mathrm{Higgs}_{\check{G}}^\diamond/\mathcal{A}^\diamond)}
	\arrow["{\mathbf{S}_{\mathrm{Dol}}}", from=1-1, to=1-2]
	\arrow["{\mathbf{T}_{x,\lambda}}"', from=1-1, to=2-1]
	\arrow["{\mathbf{W}_{x,\lambda}}", from=1-2, to=2-2]
	\arrow["{\mathbf{S}_{\mathrm{Dol}}}"', from=2-1, to=2-2]
\end{tikzcd}
\end{equation}
\end{itemize}

\subsubsection{BNR correspondence}
Passing to a Hitchin fiber via the BNR correspondence, we translate the Normalization and Hecke--Wilson compatibility specifications above in the case of $G = \mathrm{PGL}_n$ and $\check{G} = \mathrm{SL}_n$\footnote{A description for general Dynkin types can be found in Appendix A of \cite{Donagi-Pantev}; we content ourselves with the type A case (i.e., spectral curves instead of cameral curves), since this is all we will need.}: let $a \in \mathcal{A}^\diamond$ parametrize a smooth integral spectral curve $\Sigma_a \subset T^*\Sigma$ whose intersection with the zero section is transversal, then we have
\begin{itemize}
    \item (Normalization). The Hitchin fibers above $a$ can be described by (disconnected) abelian prym stacks of the spectral curve $\Sigma_a$ by the BNR correspondence:\footnote{$\mathrm{Prym}_G(\Sigma_a)$ is a \textit{disconnected} abelian variety with component group $\ZZ/n\ZZ$, while $\mathrm{Prym}_{\check{G}}(\Sigma_a)$ is a $\mu_n$-gerbe over the prym variety of $\check{G}$. }
    $$\chi^{-1}(a) \simeq \mathrm{Prym}_G(\Sigma_a)   = [\mathrm{Pic}(\Sigma_a)/\mathrm{Pic}(\Sigma)],\text{ and }$$
    $$ \check{\chi}^{-1}(a) \simeq \mathrm{Prym}_{\check{G}}(\Sigma_a) \simeq \mathrm{Ker}(\mathrm{Nm}: \mathrm{Pic}(\Sigma_a) \to \mathrm{Pic}(\Sigma)).
   $$
    The BNR correspondence sends a line bundle on $\Sigma_a$ to its pushforward along the spectral cover $\Sigma_a \to \Sigma$, with Higgs field defined by the embedding $\Sigma_a \subset T^*\Sigma$. Under the normalization \eqref{eqn normalization of S-duality}, the Hitchin section restricted to $\chi^{-1}(a)$ is the skyscraper sheaf at the identity of $\mathrm{Prym}_G(\Sigma_a)$, and we have normalized its Fourier--Mukai image to be the dualizing sheaf of $\mathrm{Prym}_{\check{G}}(\Sigma_a)$.
    \item (Hecke--Wilson compatibility). Under the BNR correspondence, Hecke--Wilson compatibility becomes the familiar fact that Fourier--Mukai transforms intertwine translations (Hecke operators) with tensorizations (Wilson operators); see, e.g., Proposition 5.1 of \cite{Schnell}. Let $x \in \Sigma_a \cap \Sigma$ be an intersection point, and let $\lambda = \varpi_1$ be the first fundamental coweight of $G$ (corresponding to the standard representation of $\check{G}$). Then $\mathbf{T}_{\lambda,x}$ corresponds to translation by the point $[\mathcal{O}_{\Sigma_a}(x)] \in \mathrm{Prym}_G(\Sigma_a)$. The $\mathbf{S}_{\mathrm{Dol}}$-dual operation is given by 
    $$
        \mathbf{W}_{\lambda,x} = \mathrm{ev}_x^* \, \mathrm{std}_{\check{G}} \, \otimes (\, \cdot \,).
    $$
\end{itemize}

\subsubsection{(BAA)/(BBB) mirror symmetry}
Loosely speaking, the duality \eqref{eqn FM equivalence} should exchange (BAA)-branes (which we can interpret at first approximation as structure sheaves of holomorphic Lagrangians) with (BBB)-branes (which we can regard at first approximation as hyperholomorphic sheaves). The study of mirror partners under $\mathbf{S}_{\mathrm{Dol}}$ has been the subject of intensive research \cite{Gaiotto, Hitchin2013, HameisterMorrissey2024, Franco-PN, Hausel-Thaddeus}, and our current discussion illustrates an application of the relative Langlands program in the context of (BAA)/(BBB) mirror symmetry.

\subsection{Hitchin's $\mathrm{U}(n,n)$-brane and the Jacquet--Shalika brane} \label{subsect Sdual of Cayley}

We consider the pair of Langlands dual groups
$$
    G = \mathrm{PGL}_{2n} \text{ and } \check{G} = \mathrm{SL}_{2n}.
$$
Consider the magical triple $\rho: \mathfrak{sl}_2 \to \g$ defined by 
$$
    \rho(f) = \begin{bmatrix} 0 & 0\\ \mathrm{Id}_n & 0\end{bmatrix},
$$
from which we obtain the Hamiltonian $G$-spaces
$$
    M_\rho' =\mathrm{WInd}_{\mathrm{PGL}_n^\Delta, \rho}^G(\mathrm{pt}) = T^*_f(U \mathrm{PGL}_n^\Delta \backslash G) \text{ and }
$$
$$M_\rho = \mathrm{WInd}_{P(\mathrm{GL}_n \times \mathrm{GL}_n)}^G(\mathrm{pt}) =T^*(\mathrm{P}(\mathrm{GL}_n \times \mathrm{GL}_n) \backslash G), 
$$
where $U \simeq \mathrm{Mat}_{n \times n}$ is the $(n,n)$-block upper triangular unipotent subgroup, and $f$ can be identified with the trace form on $U$. The real form of $G$ participating in this Cayley correspondence is
$$
    \mathrm{Lag}(M'_\rho) \longrightarrow \mathrm{Lag}(M_\rho) = \mathrm{PU}(n,n)\text{-Higgs bundles}.
$$
We write the structural morphisms of the associated Lagrangians to the Hitchin moduli space as
\begin{equation} \label{eqn A side Cayley triangle}
    \begin{tikzcd}
	{\mathrm{Cay}_{G,\rho}} && {\mathrm{Higgs}_{\mathrm{PU}(n,n)}} \\
	& {\mathrm{Higgs}_G}
	\arrow["{\mathrm{Lag}(\phi)}", from=1-1, to=1-3]
	\arrow["{\pi'}"', from=1-1, to=2-2]
	\arrow["\pi", from=1-3, to=2-2]
\end{tikzcd}
\end{equation}
where $\mathrm{Lag}(\phi)$ is the Cayley correspondence morphism.

In the remainder of this section, we give a proposal for the (BBB)-mirror to the holomorphic Lagrangian $\mathrm{Cay}_{G,\rho}$ over $\mathrm{Higgs}_G$. 

\begin{thm} \label{conj 1}
    The Fourer--Mukai dual of $(\pi'_*\mathcal{O})^\diamond$ is the pushforward of the dualizing sheaf of $\mathrm{Higgs}_{\mathrm{Sp}_{2n}}$ under the natural map $\mathfrak{i}: \mathrm{Higgs}_{\mathrm{Sp}_{2n}}^\diamond \to \mathrm{Higgs}_{\check{G}}^\diamond$. 
\end{thm}

\begin{rem}
    We arrived at Theorem \ref{conj 1} via considerations in the Dolbeault form of the relative Langlands program, where an early indication of such a duality took the form of an automorphic integral constructed by Jacquet and Shalika \cite{Jacquet-Shalika}. It is amusing to note that, in order to deduce the nonvanishing properties of the Jacquet--Shalika integral, they first unfolded to the Shalika model, which is the arithmetic analogue of $\mathrm{Cay}_{G,\rho}$. 
\end{rem}

We shall write, for the remainder of this discussion, 
$$
    B_{\mathrm{Dol}}(\check{M}_\rho') := \mathfrak{i}_* \, \omega(\mathrm{Higgs}_{\mathrm{Sp}_{2n}})
$$
for the relative dualizing sheaf of $\mathfrak{i}$, and
$$
    A_{\mathrm{Dol}}(M_\rho') := \pi_*'\mathcal{O}(\mathrm{Cay}_{G,\rho})
$$
for the pushforward of the structure sheaf of $\mathrm{Cay}_{G,\rho} \to \mathrm{Higgs}_G$, so that Theorem \ref{conj 1} can be restated simply as
$$
    \mathbf{S}_{\mathrm{Dol}}: A_{\mathrm{Dol}}(M_\rho')^\diamond \longmapsto B_{\mathrm{Dol}}(\check{M}_\rho')^\diamond.
$$
\begin{rem}
    While in our present context the notation $\check{M}_\rho'$ is purely formal, we can make rigorous sense of it using the main construction of \cite{CF}. Indeed, we propose that $\check{M}_\rho' := T^*(\mathrm{Sp}_{2n} \backslash \check{G})$, and $B_{\mathrm{Dol}}(\check{M}_\rho')$ is the associated hyperholomorphic sheaf over $\mathrm{Higgs}_{\check{G}}$. It is not difficult to imagine, with $\mathrm{Higgs}_{\mathrm{Sp}_{2n}}$ itself being hyperk\"ahler, that $B_{\mathrm{Dol}}(\check{M}_\rho')$ is itself a hyperholomorphic sheaf: for instance, under the $J$-complex structure it can be interpreted as the dualizing sheaf of $\mathrm{Loc}_{\mathrm{Sp}_{2n}}$, the moduli stack of $\mathrm{Sp}_{2n}$-local systems, pushed forward to $\mathrm{Loc}_{\check{G}}$. 
\end{rem}

Note that when $n = 1$ so that $\mathrm{Sp}_2 = \mathrm{SL}_2 = \check{G}$, the proposed mirror $B_{\mathrm{Dol}}(\check{M}_\rho')$ is simply the dualizing sheaf $\omega(\mathrm{Higgs}_{\check{G}})$. On the other hand, the Hamiltonian space $M_\rho'$ is the principal equivariant Slodowy slice, so $\mathrm{Lag}(M_\rho') \to \mathrm{Higgs}_{G}$ is precisely Hitchin's section (see Example \ref{ex principal slice is Hitchin section}), and $A_{\mathrm{Dol}}(M'_\rho)$ is its structure sheaf. Under Fourier--Mukai transform over the $\diamond$-locus, our Theorem \ref{conj 1} reduces to the expected duality \eqref{eqn normalization of S-duality}.

\subsubsection{Proof of Theorem \ref{conj 1}}  First we consider the relative Hitchin morphism 
$$
    \begin{tikzcd}
	{\mathrm{Cay}_{G,\rho}} & {\mathrm{Higgs}_G} \\
	{\mathcal{A}'_\rho} & {\mathcal{A}}
	\arrow["{\pi'}", from=1-1, to=1-2]
	\arrow["{\chi_\rho'}"', from=1-1, to=2-1]
	\arrow["\chi", from=1-2, to=2-2]
	\arrow["p'", from=2-1, to=2-2]
\end{tikzcd}
$$
as in \eqref{eqn relative Hitchin fibration}, where the relative Hitchin base $\mathcal{A}_\rho'$ is defined as
$$
    \mathcal{A}_\rho' = \mathrm{Sect}(\Sigma, (M_\rho'\sslash G)_{K^{1/2}}) \simeq \mathrm{Sect}(\Sigma, (\mathfrak{gl}_n\sslash \mathrm{PGL}_n)_{K^2})
$$
and the morphism $p'$ is induced by the morphism 
\begin{equation} \label{eqn Shalika slice}
    \mathfrak{gl}_n\sslash \mathrm{GL}_n \longrightarrow \g\sslash G
\end{equation}
$$X \longmapsto \begin{bmatrix}0 & X\\ \mathrm{Id}_n & 0\end{bmatrix}$$
by taking twisted sections from $\Sigma$. In particular, we see that $\mathrm{Cay}_{G,\rho}$ can be viewed as the moduli stack of $K^2$-twisted $\mathrm{PGL}_n$-Higgs bundles on $\Sigma$, whose spectral curves are naturally embedded inside the total space of $K^2$; this perspective will be useful later.

On the other hand, we consider the relative Hitchin morphism for the proposed Fourier--Mukai dual:
\begin{equation} \label{eqn proposed FM dual diagram}
    \begin{tikzcd}
	{\mathrm{Higgs}_{\mathrm{Sp}_{2n}}} & {\mathrm{Higgs}_{\check{G}}} \\
	{\mathcal{A}_{\mathrm{Sp}_{2n}}} & {\mathcal{A}}
	\arrow["{\check{\pi}'}", from=1-1, to=1-2]
	\arrow["{\chi_{\mathrm{Sp}_{2n}}}"', from=1-1, to=2-1]
	\arrow["{\check{\chi}}", from=1-2, to=2-2]
	\arrow[from=2-1, to=2-2]
\end{tikzcd}
\end{equation}
and we deduce by direct computation the following
\begin{lem}
    The sheaves $A_{\mathrm{Dol}}(M_\rho')$ and $B_{\mathrm{Dol}}(\check{M}_\rho')$ have the same set-theoretic support along the Hitchin base $\mathcal{A}$. 
\end{lem}
\begin{proof}
    Note that the eigenvalues of a semisimple matrix in $A \in \mathfrak{sp}_{2n}(\CC)$ are of the form 
   $$
        \lambda_1, \ldots, \lambda_n, -\lambda_1, \ldots, -\lambda_n
   $$
    for some $\lambda_1, \ldots, \lambda_n \in \CC^\times$, so its characteristic polynomial take the form
    \begin{equation} \label{eqn symplectic charpoly}
        \mathrm{charpoly}(A) = t^{2n} + a_1t^{2(n-1)} + a_2t^{2(n-2)} + \cdots + a_n
    \end{equation}
    where $a_i = \mathrm{tr}(\wedge^{2i} \, A)$. On the other hand, for a matrix of the form \eqref{eqn Shalika slice}, by direct calculation we have
    $$
        \mathrm{charpoly}\left(\begin{bmatrix} 0 & X\\ \mathrm{Id}_n & 0\end{bmatrix}\right) = t^{2n}+ b_1t^{2(n-1)} + b_2 t^{2(n-2)} + \cdots + b_n
    $$
    where $b_i = \mathrm{tr}(\wedge^i \, X)$. 

    The preceding invariant-theoretic computation implies the desired statement, since \eqref{eqn proposed FM dual diagram} is obtained by taking $K^{1/2}$-twisted sections into the diagram
    \begin{equation}
\begin{tikzcd}
	{} & {[\g/G]} & {} & {[\check{\g}/\check{G}]} \\
	& {[\mathfrak{gl}_n / \mathrm{GL}_n]} & {\g \sslash G\simeq \check{\g} \sslash \check{G}} & {[\mathfrak{sp}_{2n}/\mathrm{Sp}_{2n}]}
	\arrow[from=1-2, to=2-3]
	\arrow[from=1-4, to=2-3]
	\arrow[from=2-2, to=1-2]
	\arrow["{\nu}"', from=2-2, to=2-3]
	\arrow[from=2-4, to=1-4]
	\arrow[from=2-4, to=2-3]
\end{tikzcd}
\end{equation}
where $\nu$ is the morphism \eqref{eqn Shalika slice}. The preceding computation shows that the image of $\nu$ and the horizontal arrow on the right $[\mathfrak{sp}_{2n}/\mathrm{Sp}_{2n}] \to \check{\g}\sslash \check{G}$ coincide, and the support of $A_{\mathrm{Dol}}(M_\rho')$ along the $\mathrm{PGL}_{2n}$-Hitchin base can be computed as $K^{1/2}$-sections into the image of $\nu$.
\end{proof}

If we interpret the characteristic polynomials in the proof of the preceding Lemma as parametrizing spectral curves, then the $\diamond$-locus of the Hitchin base $\mathcal{A}^\diamond$ intersected with the support of $A_{\mathrm{Dol}}(M_\rho')$ and $B_{\mathrm{Dol}}(\check{M}_\rho')$ parametrizes smooth integral symplectic spectral curves. In other words, the eigenvalues of the symplectic Higgs field are all nonzero and distinct. 

Following \cite{Schaposnik2014}, we note that a characteristic polynomial of the form \eqref{eqn symplectic charpoly} acquires an evident involution $t \mapsto -t$. Considering again taking $K^{1/2}$-sections from the curve $\Sigma$, for a point $a \in \mathcal{A}^\diamond$ in the image of $\mathcal{A}_\rho' \to \mathcal{A}$, this involution corresponds to the involution $\sigma$ on the spectral curve $\Sigma_a$ obtained by multiplication by $(-1)$ on the cotangent bundle $T^*\Sigma$. We may consider the corresponding involution 
\begin{equation} \label{eqn involution}
    \sigma_{G,a} \in \mathrm{Aut}\big( \mathrm{Prym}_{G}(\Sigma_a) \simeq \chi^{-1}(a) \big).
\end{equation}
Similarly, we have
\begin{equation}
    \sigma_{\check{G},a} \in \mathrm{Aut}\big( \mathrm{Prym}_{\check{G}}(\Sigma_a) \simeq \check{\chi}^{-1}(a) \big).
\end{equation}
 It is well-known that, under the BNR correspondence, symplectic Higgs bundles are the fixed point locus on $\mathrm{Prym}_{\check{G}}(\Sigma_a)$ of the involution $( \, \cdot \, )^\vee \circ\sigma_{\check{G},a}$ (see, for instance, Section 3 of \cite{Hitchin2006}). In particular, if we write
\begin{equation}
\overline{\Sigma}_a := \Sigma_a/\sigma_{\check{G},a}
\end{equation}
then we have
\begin{lem}
    The symplectic Hitchin fiber $\chi^{-1}_{\mathrm{Sp}}(a) \to \chi_{\check{G}}^{-1}(a)$ can be identified with the kernel 
    $$
        \chi^{-1}_{\mathrm{Sp}}(a) \simeq \mathrm{Ker}\left(\mathrm{Nm}: \mathrm{Pic}(\Sigma_a) \to \mathrm{Pic}(\overline{\Sigma}_a)\right) \to \mathrm{Prym}_{\check{G}}(\Sigma_a).
    $$
    of the morphism of Beilinson 1-motives $\mathrm{Nm}: \mathrm{Pic}(\Sigma_a) \to \mathrm{Pic}(\overline{\Sigma}_a)$ given by the norm map. 
\end{lem}
We write thus
$$
    \mathrm{Prym}_{\mathrm{Sp}_{2n}}(\Sigma_a) := \mathrm{Ker}\left(\mathrm{Nm}: \mathrm{Pic}(\Sigma_a) \to \mathrm{Pic}(\overline{\Sigma}_a)\right)
$$
and, by Langlands duality of $\mathrm{Sp}_{2n}$ and $\mathrm{SO}_{2n+1}$, its dual Beilinson 1-motive is
\begin{equation} \label{eqn Prym of SO2n+1}
    \mathrm{Prym}_{\mathrm{SO}_{2n+1}}(\Sigma_a) := [\mathrm{Pic}(\Sigma_a)/\mathrm{Pic}(\overline{\Sigma}_a)].
\end{equation}

\begin{proof}(of Theorem \ref{conj 1}).
   We apply \eqref{eqn Fourier-Mukai functoriality} to the morphism 
   $$
       \mathfrak{i}_a: \mathrm{Prym}_{\mathrm{Sp}_{2n}}(\Sigma_a) \longrightarrow \mathrm{Prym}_{\check{G}}(\Sigma_a)
   $$
   which yields a dual morphism 
   $$
       \mathfrak{j}_a: \mathrm{Prym}_G(\Sigma_a) \longrightarrow \mathrm{Prym}_{\mathrm{SO}_{2n+1}}(\Sigma_a). 
   $$
   where we view $\mathrm{SO}_{2n+1}$ as the Langlands dual group of $\mathrm{Sp}_{2n}$. We would like to show that $\mathrm{ker}(\mathfrak{j}_a) \simeq \mathrm{Cay}_a$, which would conclude the proof of Theorem \ref{conj 1}.

    Towards our claim, note that $\mathfrak{j}_a$ can be rewritten as the natural quotient
    $$
        \mathfrak{j}_a: [\mathrm{Pic}(\Sigma_a)/\mathrm{Pic}] \longrightarrow [\mathrm{Pic}(\Sigma_a)/\mathrm{Pic}(\overline{\Sigma}_a)]
    $$
    by \eqref{eqn Prym of SO2n+1}, whose kernel is evidently
    $$
        \mathrm{Ker}(\mathfrak{j}_a) \simeq [\mathrm{Pic}(\overline{\Sigma}_a)/\mathrm{Pic}(\Sigma)].
    $$
    This latter quotient stack can be identified, by the generalized BNR correspondence, with the spectral curve of a $K^2$-twisted $\mathrm{PGL}_n$-Higgs bundle over $\Sigma$ with characteristic polynomial specified by $a$, which is exactly the moduli description of $\mathrm{Cay}_a$.
\end{proof}

  \subsubsection{}
  The above proof of Theorem \ref{conj 1} is conceptual and quick, but it is useful to have a concrete grasp on the spectral data parametrized by the stacks involved, so we give another ``direct" argument. To this end, let $p, q, r$ be the natural projection maps
  \begin{equation*}
    \begin{tikzcd}
      {\Sigma_a} && {\bar{\Sigma}_a} \\
      & {\Sigma}
      \arrow["q", from=1-1, to=1-3]
      \arrow["{p}"', from=1-1, to=2-2]
      \arrow["{r}", from=1-3, to=2-2]
    \end{tikzcd}\end{equation*}
   from the spectral curves parametrized by $a \in (\mathcal{A}'_\rho)^\diamond$. Using the analysis of the Hitchin fibers of $\U(n,n)$-Higgs bundles in \cite[Chapter~6]{Schaposnik2013}, \cite{Schaposnik2014}, we will show that \begin{equation}
       \mathrm{Cay}_{\U(n,n),a} \simeq q^*\Pic(\bar{\Sigma}_a),
   \end{equation} 
   from which we will get the desired identification
  \begin{align*}
    \mathrm{Cay}_{\mathrm{PU}(n,n),a} \simeq [q^*\Pic(\bar{\Sigma}_a) / p^*\Pic(\Sigma)]
  \end{align*}
  by passing to the adjoint form of the group. 

  To each $\U(n,n)$-Higgs bundle $(V \oplus W, \Phi)$, we associate a Toledo invariant $\tau = \deg(V) - \deg(W)$. This invariant satisfies a Milnor--Wood inequality: $0 \le |\tau| \le 2n(g-1)$. When $|\tau|=2n(g-1)$ is achieved, we say the Higgs bundle is maximal. It is well-known that $\mathrm{Cay}_{\U(n,n)}$ can be seen as the space of maximal $\U(n,n)$-Higgs bundles as in \cite[Section~3.5]{BradlowEtAl2003}, \cite{BradlowEtAl2007}. 
  
  Suppose $L \in \Pic(\bar{\Sigma}_a)$. By the projection formula we have
  \begin{align*}
    q_*(q^*L \otimes \calO_{\Sigma_a}) \simeq L \otimes (\calO_{\bar{\Sigma}_a} \oplus r^*K^{-1}) \simeq L \oplus L \otimes r^*K^{-1}. 
  \end{align*}
  The associated Toledo invariant is given by
  \begin{align*}
    \tau = \deg(L) - \deg(L \otimes r^*K^{-1}) = n\deg(K) = n(2g-2), 
  \end{align*}
  which is maximal. 

  Conversely if $M \in \Pic(\Sigma_a)$ is the line bundle corresponding to a maximal $\U(n,n)$-Higgs bundle, then $q_*M = U_+ \oplus U_-$ has an associated $r^*K$-valued Higgs field
  \begin{align*}
    \begin{pmatrix}
      & \psi_- \\
      \psi_+ &
    \end{pmatrix}
  \end{align*}
  where $\psi_+: U_+ \to U_- \otimes r^*K$ is an isomorphism by maximality, since it may be viewed as a nonzero section of the degree zero line bundle $U_+^* \otimes U_- \otimes r^*K$. We apply the projection formula again to $q^*U_+$, and then apply the BNR correspondence to see that we must have $M \simeq q^*U_+$.

  In this case we get an $r^*K^2$-twisted Higgs bundle $(U_+, \psi)$, where the Higgs field is the composition
  \begin{align*}
    \psi: U_+ \xrightarrow{\psi_+} U_- \otimes r^*K \xrightarrow{\psi_- \otimes 1_{r^*K}} U_+ \otimes r^*K^2. 
  \end{align*}
  Since $\psi_+$ is an isomorphism, we also have
  \begin{align*}
    (U_+, \psi) \simeq (U_- \otimes r^*K, \psi_+ \circ \psi \circ \psi_+^{-1}). 
  \end{align*}
  Taking the pushforward $(r_*U_+, r_*\psi)$ gives us a $K^2$-twisted rank $n$ Higgs bundle on $\Sigma$.

\subsection{Symplectic Dirac--Higgs bundles} \label{subsect symplectic DH bundle}
On the other hand, the Fourier--Mukai dual to $\mathrm{U}(n,n)$-Higgs bundles inside $\mathrm{GL}_{2n}$-Higgs bundles were already predicted and studied to some extent by Hitchin in Section 7 of \cite{Hitchin2013}, working at the level of moduli spaces. We reproduce the arguments in \textit{loc. cit}, paying some attention to the refinement to moduli stacks which recovers a more complete $S$-duality phenomenon. 

To describe Hitchin's answer, we consider, for the universal $\mathrm{Sp}_{2n}$-Higgs bundle $(\mathcal{E}, \varphi)$, the associated standard rank $2n$ Higgs bundle which we continue to denote by $(\mathcal{E}, \varphi)$ and its \textit{Dirac--Higgs complex}
$$
    \mathrm{DH}_{\mathrm{Sp}_{2n}}^\bullet := R\Gamma^\bullet(\Sigma, \mathcal{E} \overset{\varphi}{\longrightarrow} \mathcal{E} \otimes K)
$$
which we regard as a (derived) vector bundle over $\mathrm{Higgs}_{\mathrm{Sp}_{2n}}$. Equivalently, we can describe $\mathrm{DH}^\bullet_{\mathrm{Sp}_{2n}}$ as built from the tautological diagram
\begin{equation} \label{eqn evaluation pushforward}
    \begin{tikzcd}
	{\Sigma_{\mathrm{Dol}}\times \mathrm{Higgs}_{\mathrm{Sp}_{2n}}} & {B\mathrm{Sp}_{2n}} \\
	{\mathrm{Higgs}_{\mathrm{Sp}_{2n}}}
	\arrow["{\mathrm{ev}}", from=1-1, to=1-2]
	\arrow["{p_2}", from=1-1, to=2-1]
\end{tikzcd}
\end{equation}
where $\Sigma_{\mathrm{Dol}} := B\widehat{T}_\Sigma$ is the classifying stack of the formal group of the tangent bundle over $\Sigma$, introduced by Simpson \cite{Simpson1999}. The primary motivation for introducing $\Sigma_{\mathrm{Dol}}$ was to view Higgs bundles on $\Sigma$ as usual bundles on $\Sigma_{\mathrm{Dol}}$, which allows us to give another useful presentation of the Hitchin moduli stack as $\mathrm{Higgs}_{\mathrm{Sp}_{2n}} = \mathrm{Map}(\Sigma_{\mathrm{Dol}}, \mathrm{BSp}_{2n})$. Via diagram \eqref{eqn evaluation pushforward}, we can describe $\mathrm{DH}^\bullet_{\mathrm{Sp}_{2n}}$ alternatively as
\begin{equation} \label{eqn DH}
    \mathrm{DH}^\bullet_{\mathrm{Sp}_{2n}} = p_{2,*} \, \mathrm{ev}^* \, \mathrm{Std}_{\mathrm{Sp}_{2n}}.
\end{equation} 

Referring to the diagram \eqref{eqn A side Cayley triangle}, we state Hitchin's Theorem in the stack-theoretic setting. 
\begin{thm}[Section 7, \cite{Hitchin2013}; Theorem 1.1 \cite{HameisterEtAl2024}] \label{conj 2}
    The Fourier--Mukai dual of $(\pi_*\mathcal{O})^\diamond$ is identified with 
    $$
        \left(\oplus_{k \geq 0} \, \wedge^k \mathrm{DH}^1_{\mathrm{Sp}_{2n}}[k]\right)^\diamond
    $$
    over $\mathrm{Higgs}_{\mathrm{Sp}_{2n}}^\diamond$, pushed forward to $\mathrm{Higgs}_{\check{G}}^\diamond$. 
\end{thm}

From the perspective of (BAA)/(BBB) mirror symmetry, the mirror to $\mathrm{PU}(n,n)$ proposed by Hitchin ought to be the dualizing sheaf of a certain relative Hitchin moduli space in the sense of \cite{CF}, which is naturally an ind-coherent sheaf on $\mathrm{Higgs}_{\check{G}}$.\footnote{Here, we restrict to the smooth locus $\mathrm{Higgs}_G^\diamond$ and regard ind-coherent sheaves on it as quasi-coherent sheaves.} We write, following the notation introduced after Theorem \ref{conj 1},
$$
    A_{\mathrm{Dol}}(M_\rho) := \pi_* \mathcal{O}(\mathrm{Higgs}_{G_\rho^\RR})
$$
for the pushforward of the structure sheaf along $\mathrm{Higgs}_{G_\rho^\RR} \to \mathrm{Higgs}_G$, and
$$
    B_{\mathrm{Dol}}(\check{M}_\rho) := \oplus_{k \geq 0} \, (\mathrm{Sym}^k\,  \mathrm{DH}_{\mathrm{Sp}_{2n}})[k]
$$
for the ``sheared" symmetric algebra of the Dirac--Higgs bundle. Note that on the locus of $\mathrm{Higgs}_{\mathrm{Sp}_{2n}}$ where $\mathrm{DH}^\bullet_{\mathrm{Sp}_{2n}}$ is concentrated in degree 1 (for instance in the $\diamond$-locus), the symmetric algebra should be understood as an exterior algebra, so $B_{\mathrm{Dol}}(\check{M}_\rho)^\diamond$ recovers Hitchin's description. In other words, we can rewrite Theorem \ref{conj 2} simply as
$$
    \mathbf{S}_{\mathrm{Dol}}: A_{\mathrm{Dol}}(M_\rho)^\diamond \longmapsto B_{\mathrm{Dol}}(\check{M}_\rho)^\diamond.
$$

\begin{rem}
    Again, in our present context the notation $\check{M}_\rho$ is purely formal, but in the sense of \cite{CF} we may take it to be the hyperspherical dual $\check{M}_\rho = T^*(\mathrm{Std}_{\mathrm{Sp}} \times^{\mathrm{Sp}_{2n}} \mathrm{SL}_{2n})$ of $M_\rho$. 
\end{rem}

We highlight three slight modifications from the literal statements of \cite{Hitchin2013} which should be well-known to experts:
\begin{itemize}
    \item (Moduli space v.s. moduli stack). The universal bundle $(\mathcal{E}, \varphi)$ does not exist on the moduli space of stable $\mathrm{Sp}_{2n}$-Higgs bundles, but only on $\mathrm{Higgs}_{\mathrm{Sp}_{2n}}$; for this reason, in \textit{loc. cit} Hitchin works with only \textit{even} exterior powers of $\mathrm{DH}_{\mathrm{Sp}_{2n}}$, as these vector bundles descend (on the stable locus) to the moduli space of stable $\mathrm{Sp}_{2n}$-Higgs bundles. Correspondingly, our $B_{\mathrm{Dol}}(\check{M}_\rho)^\diamond$ has double the rank of the bundle considered in \textit{loc. cit}. Dually, on the A-side Hitchin makes a restriction of considering only the components of even degree in $\mathrm{Higgs}_G$ for the purposes of applying the nonabelian Hodge correspondence. We consider all components, which results in double the amount of points in each Hitchin fiber supporting $A_{\mathrm{Dol}}(M_\rho)^\diamond$. 
    \item (Description away from the $\diamond$-locus). Since $B_{\mathrm{Dol}}(\check{M}_\rho)$ has infinite rank when restricted to the loci where $\mathrm{DH}^0_{\mathrm{Sp}_{2n}} \neq 0$, we expect that the Fourier--Mukai dual of $A_{\mathrm{Dol}}(M_\rho)$ outside the $\diamond$-locus could have infinite rank at these points. While our definition seems natural, to the best of our knowledge there is no technology with which one could verify Fourier--Mukai duality on this locus. 
    \item (Shearing). The exterior algebra of $\mathrm{DH}^1_{\mathrm{Sp}_{2n}}$ (which itself lives in cohomological degree 1) is concentrated in cohomological degrees in the interval $[0,\mathrm{rk} \, \mathrm{DH}^1_{\mathrm{Sp}_{2n}}]$, where the $k$th wedge power lives in degree $k$. Shearing simply moves the $k$th wedge power piece back to degree 0. 
\end{itemize}

We record a strategy to prove Hitchin's Theorem \ref{conj 2} following Section 7 of \cite{Hitchin2013}. Although the argument should be well-known to experts, we include the sketch here for completeness and since we appeal to its argument in formulating our Conjecture \ref{conj 3}.

\begin{proof}(Hitchin's Theorem \ref{conj 2} following \textit{loc. cit}.)
    Let $a \in (\mathcal{A}'_\rho)^\diamond$. Then the points in $\chi^{-1}(a) = \mathrm{Prym}_G(\Sigma_a)$ corresponding to $\mathrm{PU}(n,n)$-Higgs bundles can be obtained from the point $\mathrm{Cay}_{\mathrm{PU}(n,n)} \cap \chi^{-1}(a)$ (of maximal Toledo invariant) by Hecke modifcations along all possible subsets of points in $\Sigma_a \cap \Sigma$, of which there are $4n(g-1)$. In symbols, 
    \begin{equation} \label{eqn U(n,n) is Hecke acting on Cayley}
        A_{\mathrm{Dol}}(M_\rho')|_{\chi^{-1}(a)} = \bigoplus_{J \subset (\Sigma \cap \Sigma_a)} \, \mathbf{T}_J  \, (A_{\mathrm{Dol}}(M_\rho'))
    \end{equation}
    where $\mathbf{T}_J := \mathbf{T}_{x_1} \circ \cdots \circ \mathbf{T}_{x_{|J|}}$ is the composition of (commuting) Hecke operators at the points $\{x_1, \ldots, x_{|J|}\}$ in $J$. Applying $S$-duality, Hecke--Wilson compatibility \eqref{eqn Hecke Wilson compatibility} and Theorem \ref{conj 1}, we have
    \begin{equation}
        \mathbf{S}_{\mathrm{Dol}}(A_{\mathrm{Dol}}(M_\rho')|_{\chi^{-1}(a)}) = \bigoplus_J \, \mathbf{W}_J \otimes B_{\mathrm{Dol}}(\check{M}_\rho')|_{\check{\chi}^{-1}(a)}.
    \end{equation}
    where $\mathbf{W}_J := \mathbf{W}_{x_1} \circ \cdots \circ \mathbf{W}_{x_{|J|}}$ is the composition of (commuting) Wilson operators at the points in $J$. 
    
    Note that $\mathrm{DH}^1_{\mathrm{Sp}_{2n}}$ restricted to the Hitchin fiber $\check{\chi}^{-1}(a)$ in $\mathrm{Higgs}_{\check{G}}^\diamond$ can be identified with the cokernel of the universal Higgs field in cohomological degree 1: this is a direct sum of line bundles, one for each point in $\Sigma \cap \Sigma_a$. Taking exterior algebra and shearing, we obtain a direct sum over subsets $J \subset \Sigma \cap \Sigma_a$ where the summand indexed by $J$ is the tensor product over points in $J$ of the cokernel-line bundles. Thus, we recognize that $\bigoplus_J \, \mathbf{W}_J \simeq B_{\mathrm{Dol}}(\check{M}_\rho)|_{\check{\chi}^{-1}(a)}$ when restricted to $\check{\chi}^{-1}(a)$ which allows us to conclude.
\end{proof}

\begin{rem}
    Similar results were obtained in \cite{HameisterMorrissey2024} and \cite{HameisterEtAl2024} for more general symmetric pairs (among many other examples), via an analysis of a certain intermediate \textit{regular quotient} of the stack quotient $[M_\rho/G]$ and the coarse quotient $M_\rho\sslash G$. The regular quotient contains non-separated loci which, upon considering $K^{1/2}$-twisted maps from the curve, correspond to the $2^{4n(g-1)}$ possible Hecke modifications.
\end{rem}

\subsection{Mirror of the Cayley morphism}

We arrive finally at our proposal for the Fourier--Mukai dual to the Cayley morphism. We shall precisely state our expectations when the magical real form is \textit{tempered} (to be defined below), and prove the case of $G^\RR = \mathrm{PU}(n,n)$; this case follows quickly from the previous calculations of Theorems \ref{conj 1} and \ref{conj 2}, and we expect that those Cayley morphisms attached to tempered real forms should be approachable by similar techniques.

Recall that for a symmetric $G$-variety $X$ (or more generally a spherical $G$-variety $X$), one can attach a \textit{Nadler/Gaitsgory--Nadler/spherical-dual group} \cite{Nadler, Gaitsgory-Nadler, Knop-Schalke, Sakellaridis-Venkatesh} $\check{G}_X$ with a morphism to $\check{G}$. 

\begin{defn}
    Let $\rho$ be a magical $\mathfrak{sl}_2$-triple in $\g$, and let $H \subset G$ be the complexification of the maximal compact subgroup of the real form $G_\rho^\R$. We say that $\rho$ is \textit{tempered} if the spherical dual group $\check{G}_\rho$ of the symmetric variety $X_\rho = H \backslash G$ admits no commuting $\mathfrak{sl}_2$-triples in $\check{\g}$.
\end{defn}
By slight abuse of terminology, we will also say that a real form, or a Cayley morphism is tempered if they arise from a tempered magical triple.

\begin{rem}
    Usually, the spherical dual group of a spherical $G$-variety $X$ is denoted by $\check{G}_X$ in the literature. Here we opt to emphasize the dependence on the magical triple $\rho$ in the notation instead. 
\end{rem}
\begin{rem}
    The terminology of temperedness originates from automorphic literature as a measure of growth of certain automorphic forms. Slightly more precisely, those spherical varieties $X$ which are tempered should distinguish (in the sense of Plancherel measure, see Section 16 of \cite{Sakellaridis-Venkatesh}) those automorphic forms whose Arthur parameter has vanishing Arthur-$\mathrm{SL}_2$. Roughly speaking, this condition corresponds to the automorphic form being ``as close to $L^2$ as possible". 
\end{rem}

We expect that the Fourier--Mukai duals of Cayley morphisms in the tempered case should be induced by the projection to the zero section of a certain Dirac--Higgs bundle, whose support is the Hitchin moduli stack for the Nadler/spherical dual group. In order to spell out these expectations more precisely, we tabulate here the list of tempered magical triples, following the classification of \cite{Cayley} (see Theorem \ref{thm:classification-canonical-real-forms}), including their Nadler/spherical dual groups $\check{G}_\rho$ (up to isogeny).

\begin{center}
\begin{table}[h!] \label{table hermitian tube type}
\centering
\caption{Tempered magical triples}
\label{tab:tempered}
\small 
\renewcommand{\arraystretch}{1.2} 
\begin{tabular}{c|c|c|c|c} 
\Xhline{3\arrayrulewidth} 
Type & Magical type &\textbf{$\g^\R_\rho$} & \textbf{$\mathfrak{h}$} & $\check{\g}_\rho \,(\subset \check{\g})$ \\
\Xhline{3\arrayrulewidth} 
All & (1) & split real form &  & $\check{\g}$  \\

\Xhline{2\arrayrulewidth} 
$\mathrm{A}_{2n-1}$ & (2) & $\mathfrak{su}(n,n)$ & $\mathfrak{s}(\mathfrak{gl}_n \times \mathfrak{gl}_n)$ &  $\mathfrak{sp}_{2n} \, (\subset \mathfrak{sl}_{2n})$  \\

\Xhline{2\arrayrulewidth} 
$\mathrm{B}_n$ & (3)$\ast$ & $\mathfrak{so}(n,n+1)$ & $\mathfrak{so}_n \times \mathfrak{so}_{n+1}$ & $ \mathfrak{sp}_{2n} \, (\subset \mathfrak{sp}_{2n})$   \\

\Xhline{2\arrayrulewidth} 
$\mathrm{C}_n$ & (2)$\ast$ & $\mathfrak{sp}_{2n}\RR$ & $\mathfrak{gl}_n$ &  $\mathfrak{so}_{2n+1} \, (\subset \mathfrak{so}_{2n+1})$\\

\Xhline{2\arrayrulewidth} 
$\mathrm{D}_{n+1}$ & (3) & $\mathfrak{so}(n,n+2)$ & $\mathfrak{so}_n \times \mathfrak{so}_{n+2}$ & $\mathfrak{so}_{2n+1} \, (\subset \mathfrak{so}_{2n+2})$   \\

\Xhline{2\arrayrulewidth} 
$\mathrm{D}_n$ & (3)$\ast$ & $\mathfrak{so}(n,n)$ & $\mathfrak{so}_n \times \mathfrak{so}_n$ & $\mathfrak{so}_{2n} \, (\subset \mathfrak{so}_{2n})$   \\

\Xhline{2\arrayrulewidth} 
$\mathrm{E}_6$ & (4) & $\mathfrak{e}_{6(2)}$ & $\mathfrak{sl}_6 \times \mathfrak{sl}_2$ & $\mathfrak{f}_4 \, (\subset \mathfrak{e}_6)$   \\

\Xhline{2\arrayrulewidth} 
 $\mathrm{F}_4$ & (4)$\ast$ & $\mathfrak{f}_{4(4)}$ & $\mathfrak{sp}_6 \times \mathfrak{sl}_2$ & $\mathfrak{f}_4 \, (\subset \mathfrak{f}_4)$  \\

\Xhline{3\arrayrulewidth} 
\end{tabular}
\end{table}
\end{center}
The notation used in the above table is as follows:
\begin{itemize}
    \item $\g^\R_\rho$: the (Lie algebra of the) real form of $G$ associated to $\rho$. 
    \item $\mathfrak{h}$: the complexification of the maximal compact subalgebra of $\g^\R_\rho$.
    \item $\check{\g}_\rho$: the (Nadler/spherical)-dual group of the symmetric spherical $G$-variety $X_\rho = H \backslash G$ (see Section 4.2 of \cite{BZSV} for an explicit construction).
    \item The magical types marked with an asterisk $\ast$ are those appearing in Section \ref{subsect factorization}: these real forms admit a magical triple of type (1) as well. The expectation for these types have to be modified slightly.  
\end{itemize}

In terms of BZSV triples and relative Langlands duality, Table \ref{tab:tempered} can be viewed as a list of \textit{tempered hyperspherical dual Hamiltonian actions} (see Section 5 of \textit{op. cit}):
$$G \acts M_\rho = \mathrm{WInd}_H^G(\mathrm{pt}) \overset{\text{hyp.sph. dual}}{\longleftrightarrow} \check{G} \acts \check{M}_\rho = \mathrm{WInd}_{\check{G}_\rho, \mathrm{triv}}^{\check{G}}(S_\rho).$$ 
where $S_\rho$ is a certain symplectic representation of $\check{G}_\rho$ (see Sections 4.3 and 4.4 of \textit{loc. cit} for its definition). Note that not all tempered magical triples lead to hyperspherical varieties: the only condition which could fail is that of \textit{connected generic stabilizers} (this occurs, for instance, for the principal magical triple in type $\mathrm{A}$). While we expect that these exceptions can be resolved eventually (by allowing Deligne--Mumford stacks to be its hyperspherical dual), we will exclude them for the present discussion.

\begin{rem}
    Under the magical assumption, we see that a real form is tempered if and only if it is quasi-split. Furthermore, it is \textit{strongly tempered} (meaning that the Nadler/spherical dual group is the full Langlands dual group) if and only if it is split. 
\end{rem}

Suppose $S_\rho$ is a representation of \textit{cotangent type}, i.e., that $S_\rho \simeq V_\rho \oplus V_\rho^*$ for some representation $V_\rho$ of $\check{G}_\rho$. We write, for those $M_\rho$ appearing as the recipient of a tempered Cayley morphism, 
$$B_{\mathrm{Dol}}(\check{M}_\rho) := (\mathrm{Sym}^k \, \mathrm{DH}_{V_\rho})[k] \in \mathrm{QC}(\mathrm{Higgs}_{\check{G}}),$$
the sheared symmetric algebra of the Dirac Higgs bundle with fiber $V_\rho$ viewed as a quasi-coherent sheaf over $\mathrm{Higgs}_{\check{G}_\rho}$ pushed-forward to $\mathrm{Higgs}_{\check{G}}$. Here, the Dirac--Higgs bundle is defined as in Section \ref{subsect symplectic DH bundle}:
$$\mathrm{DH}^\bullet_{V_\rho} := R\Gamma^\bullet(\Sigma, V_{\rho,\mathcal{E}} \overset{\varphi}{\longrightarrow} V_{\rho, \mathcal{E}} \otimes K),$$
where $(\mathcal{E},\varphi)$ is the universal $\check{G}_\rho$-Higgs bundle. 

\begin{conj} \label{conj 3}
    Let $\rho: \mathfrak{sl}_2 \to \g$ be a tempered magical $\mathfrak{sl}_2$-triple. Assume that 
    \begin{itemize}
        \item $M_\rho$ is hyperspherical (equiv. $X_\rho$ has connected generic stabilizers), and that
        \item the representation $S_\rho \simeq V_\rho \oplus V_\rho^*$ of $\check{G}_\rho$ is of cotangent type.
    \end{itemize} 
    Let $G$ be the adjoint group with Lie algebra $\g$, and let $\check{G}$ be its Langlands dual group. Regard the Cayley morphism for $G^\RR_\rho$ as a morphism of sheaves
    $$
        A_{\mathrm{Dol}}(\phi):A_{\mathrm{Dol}}(M_\rho) \longrightarrow A_{\mathrm{Dol}}(M_\rho') \in \mathrm{QC}(\mathrm{Higgs}_G).
    $$
    Then we have the following pieces of matching data under $\mathbf{S}_{\mathrm{Dol}}$, over the $\diamond$-locus:
    \begin{enumerate}
        \item The $\mathbf{S}_{\mathrm{Dol}}$-dual to $A_{\mathrm{Dol}}(M_\rho')$ is $B_{\mathrm{Dol}}(\check{M}_\rho') = \omega(\mathrm{Higgs}_{\check{G}_\rho})$ (in the $\ast$-cases, a finite rank vector bundle $W_\rho$ tensored with $\omega(\mathrm{Higgs}_{\check{G}_\rho})$).
        \item The $\mathbf{S}_{\mathrm{Dol}}$-dual to $A_{\mathrm{Dol}}(M_\rho)$ is $B_{\mathrm{Dol}}(\check{M}_\rho) = \oplus_{k \geq 0} (\mathrm{Sym}^k \, \mathrm{DH}_{V_\rho})[k]$, the sheared symmetric algebra of the Dirac--Higgs bundle with fiber $V_\rho$ (tensored with $W_\rho$ in the $\ast$-cases).
        \item The $\mathbf{S}_{\mathrm{Dol}}$-dual to $A_{\mathrm{Dol}}(\phi)$ is the natural morphism
   $$
        B_{\mathrm{Dol}}(\check{M}_\rho) \longrightarrow B_{\mathrm{Dol}}(\check{M}_\rho')
    $$
    induced by the projection to the zero section of the Dirac--Higgs bundle
    \begin{equation} \label{eqn projection of bundle}
       \mathrm{DH}_{V_\rho} \longrightarrow \mathrm{Higgs}_{\check{G}_\rho} 
    \end{equation}
    (and tensoring with $W_\rho$ in the $\ast$-cases).
    \end{enumerate}

\end{conj} 
\begin{proof}(for $G^\RR = \mathrm{PU}(n,n)$.) Point (1) is exactly Theorem \ref{conj 1} and point (2) is the proof of Hitchin's Theorem \ref{conj 2}, so it remains to deduce point (3).

    From equation \eqref{eqn U(n,n) is Hecke acting on Cayley}, we see that $A_{\mathrm{Dol}}(M_\rho')$ corresponds to projection to the summand with $J = \emptyset$ is the empty set. Under $\mathbf{S}_{\mathrm{Dol}}$, this corresponds to projection to the summand $\mathbf{W}_{\emptyset} \otimes B_{\mathrm{Dol}}(M_\rho') \simeq \omega(\mathrm{Higgs}_{\mathrm{Sp}_{2n}})^\diamond$ 
\end{proof}

We end with some important remarks and observations about the above conjecture.

\begin{rem}
    In the $\ast$-cases, the appearance of the extra vector bundle $W_\rho$ is not mysterious. Note from Table \ref{tab:tempered} that in the cases of Type $\mathrm{B}_n(3), \mathrm{C}_n(2), \mathrm{D}_n(3),$ and $\mathrm{F}_4(4)$, the Nadler dual group is the full Langlands dual group (i.e., they are \textit{strongly tempered}). Without modifying by the finite rank vector bundle $W_\rho$, Conjecture \ref{conj 3} cannot be true, since $\omega(\mathrm{Higgs}_{\check{G}_\rho}) = \omega(\mathrm{Higgs}_{\check{G}})$ must be $\mathbf{S}_{\mathrm{Dol}}$-dual to the Hitchin section. 

    On the other hand, from Section \ref{subsect factorization} we learned that the Hitchin components (i.e., the principal Cayley components) form a subset of Cayley components. Intersecting with a Hitchin fiber in the $\diamond$-locus, we expect to be able to Hecke-modify to obtain all of the Cayley points from the Hitchin ones. The Fourier--Mukai dual to this operation leads to the finite rank (hyperholomorphic) vector bundle $W_\rho$. 
\end{rem}

\begin{rem}
Part (2) of Conjecture \ref{conj 3} is closely related to Conjecture 1.9 of \cite{HameisterEtAl2024}, which they have also proven in the $G^\R = \mathrm{PU}(n,n)$-case (Theorem 1.1 ``Friedberg--Jacquet case" of \textit{op. cit}). Part (1) of our Conjecture can be viewed as a natural extension of relative Langlands duality in the Dolbeault form to the twisted polarized setting (which includes, but is more general than, those Cayley spaces that appear here). 
\end{rem}

\begin{rem}
    It is tempting to formulate Conjecture \ref{conj 3} in the non-tempered case. This is not impossible, as the relevant (BBB)-branes $B_{\mathrm{Dol}}(\check{M}_\rho)$ and $B_{\mathrm{Dol}}(\check{M}_\rho)$ can be constructed (see Section 5 of \cite{CF}). However, it seems that in the non-tempered case both the Cayley morphism and its proposed mirrors will have support in the complement of the $\diamond$-locus. Thus, to the best of our knowledge, one cannot perform a direct Fourier--Mukai duality check for these hypothetical mirror pairs.
\end{rem}

\subsection{$S$-duality in codimension 2}
Let us contextualize our Conjecture \ref{conj 3} within the relative Langlands program and recent works surrounding the topic, focusing on the case when $G^\RR_\rho = \mathrm{PU}(n,n)$. On the one hand, $A_{\mathrm{Dol}}(\phi)$ is induced by the morphism $\phi$ of Hamiltonian spaces (a special case of Proposition \ref{prop morphism of actions}) 
\begin{equation}\label{eqn morphism 1}
    \phi: [M'_\rho/G] \longrightarrow [M_\rho/G].
\end{equation}
On the other hand, following the constructions of \cite{CF} we have
$$
    \check{M}_\rho = T^*(\mathbf{A}^{2n} \times^{\mathrm{Sp}_{2n}} \check{G}) \,  \text{ and } \, \check{M}_\rho' = T^*(\mathrm{Sp}_{2n} \backslash \check{G})
$$
and the morphism \eqref{eqn projection of bundle} is induced by (choosing a $\check{G}$-stable polarization and) the natural morphism 
\begin{equation} \label{eqn morphism 2}
    \check{\phi}:  [\mathbf{A}^{2n} \times^{\mathrm{Sp}_{2n}} \check{G}/\check{G}] \longrightarrow [\mathrm{Sp}_{2n} \backslash \check{G}/\check{G}]
\end{equation}
of projection to the zero section of a $\check{G}$-equivariant vector bundle.  

From the perspective of functorial field theory developed in Remarks \ref{rem GL nonsense 1}, \ref{rem AKSZ}, and \ref{rem GL nonsense 2}, we are proposing that in the hypotehtical category $\mathfrak{B}_{\mathrm{GL}}$, the following two 2-morphisms are $S$-dual to each other:
\begin{equation}
\begin{tikzcd}
	& {\mathcal{L}_\phi} &&& {\mathcal{L}_{\check{\phi}}} \\
	{M_\rho'} && {M_\rho} & {\check{M}_\rho'} && {\check{M}_\rho}
	\arrow[from=1-2, to=2-1]
	\arrow[from=1-2, to=2-3]
	\arrow[from=1-5, to=2-4]
	\arrow[from=1-5, to=2-6]
\end{tikzcd}
\end{equation}
where $\mathcal{L}_\phi$ is the equivariant Lagrangian correspondence obtained by taking the graph of $\phi$, while $\mathcal{L}_{\check{\phi}}$ is the equivariant Lagrangian correspondence induced by the polarized morphism $\check{\phi}$:
\begin{equation} \mathcal{L}_{\check{\phi}} = \left\{(x,\check{\phi}(x); d\check{\phi}(\xi), \xi): \begin{matrix}x \in \mathbf{A}^{2n} \times^{\mathrm{Sp}_{2n}} \check{G}, \text{ and }\\ \xi \in T^*_{\check{\phi}(x)}(\mathrm{Sp}_{2n} \backslash\check{G})\end{matrix}\right\}.
\end{equation}

Let us explain, to the best of our knowledge, the status of the statement of our proposal. First of all, the pairs of Hamiltonian actions
$$
    (G \acts M_\rho) \longleftrightarrow (\check{G} \acts \check{M}_\rho)
$$
$$
    (G \acts M_\rho') \longleftrightarrow (\check{G} \acts \check{M}_\rho')
$$
are known to be hyperspherical dual in the sense of \cite{BZSV}, hence our Theorem \ref{conj 1} and the well-known Theorem \ref{conj 2} of Hitchin complete half of\footnote{Indeed, we only consider the A-twist for those Hamiltonian actions on the left, and B-twist for those Hamiltonian actions on the right. One should be able to switch the A-twist and the B-twist and obtain two more pairs of (BAA)/(BBB) mirrors.} the relative Langlands duality attached to these hyperspherical dual pairs in the Dolbeault setting. On the other hand, hyperspherical duality is not a functor, so at the moment of writing one cannot mathematically formulate the statement that ``$\mathcal{L}_\phi$ is the hyperspherical dual $\mathcal{L}_{\check{\phi}}$". 

Nonetheless, we have learned from Hiraku Nakajima that the above two pairs of Hamiltonian actions can be shown to be $S$-dual in the mathematical sense as defined in \cite{Nakajima} (which is expected to coincide with hyperspherical duality whenever the latter applies), if one assumes natural compatibility of the $S$-dual construction in \textit{op. cit} with symplectic reductions. This latter construction is functorial for certain nice morphisms, and it would be interesting to investigate whether \eqref{eqn morphism 2} is $S$-dual to \eqref{eqn morphism 1} in this sense. 

Finally, we mention two more mathematical checks of our proposal in the ``de Rham" setting. While these statements will not directly imply nor be implied by the ``Dolbeault" form of our proposal, it does lend credence to the fact that, if one is able to construct a piece of $\mathfrak{B}_{\mathrm{GL}}$, and that $S$-duality does extend to certain 2-morphisms, then our proposal should be valid. 

\begin{itemize}
    \item (de Rham form, over a global curve, $n=1$). The left hand side column of \eqref{eqn intro S-dual of Cayley diagram 2} under the ``de Rham A-twist" becomes a morphism of BZSV's period sheaves attached to the Hamiltonian $G$-actions $M_1$ and $M_2$. When $n = 1$ the period sheaf attached to $M_1$ is precisely the Whittaker sheaf, which forms a direct summand of the upper left corner term of Theorem 7.4.2 of \cite{Feng-Wang}. The period sheaf attached to $M_2$ is precisely $\mathcal{P}_X$ in the same diagram of \textit{loc. cit} (upon passing to the adjoint form of the group). We find our corresponding summand of the BZSV $L$-sheaf attached to $\check{M}_1$ as the $\bullet = 0$ summand of $\mathrm{Loc}^{\mathrm{spec}}(\mathrm{Fact}(\mathrm{Sym}^\bullet \mathrm{Std}))^\fltns$ (which is simply the dualizing sheaf of $\mathrm{Loc}_{\check{G}}$), and  $\mathcal{L}_{\check{X}}$ in \textit{loc. cit} is precisely the BZSV $L$-sheaf attached to $\check{M}_2$. 
    \item (Nakajima's $S$-duality). Appealing to the expected involutivity of $S$-duality, we may apply the definition of \cite{Nakajima} to $\check{M}_1$ and $\check{M}_2$ (both of which are polarized). Using the fact that the definition of \textit{op. cit} is functorial with respect certain morphisms, one can compute a tentative $S$-dual to $\mathcal{L}_{\check{\phi}}$, which coincides with the Cayley morphism $\mathcal{L}_\phi$. 
\end{itemize}
A more detailed explanation of these observations in the de Rham setting will appear in subsequent work.

\appendix
\section{Lagrangian morphisms and Lagrangian submanifolds}
\label{appendix}

Let $G$ be a semisimple algebraic group whose center we denote by $Z_G$, and only for the appendix, we assume that the genus $g(\Sigma)$ of our curve is at least 2. 

Let $\mathfrak{M}(G) = \mathrm{Higgs}_G$ be the moduli stack of $G$-Higgs bundles on our fixed algebraic curve $\Sigma$. The assumption that $G$ is semisimple (rather than more generally reductive) should not be essential, but it does simplify matters slightly for us in the following way. Write $\mathfrak{M}(G)^{\mathrm{st}} \subset \mathfrak{M}(G)$ for the open substack of stable $G$-Higgs bundles, whose automorphism group reduces to the center. Then semisimplicity of $G$ ensures that $\mathfrak{M}(G)^{\mathrm{st}}$ is a (classical) Deligne--Mumford stack, in fact a $Z_G$-gerbe over its coarse moduli space $\mathcal{M}(G)$, the moduli space of stable $G$-Higgs bundles.

\subsection{Classicality of the Cayley space}

Let $\mu: \mathfrak{L}(M) = \mathrm{Lag}(M) \to \mathfrak{M}(G)$ be Gaiotto's Lagrangian attached to some Hamiltonian $G$-space $M$. We are interested in the passage from $\mathfrak{M}(G)$ to the moduli scheme of stable $G$-Higgs bundles $\mathcal{M}(G)$, and most importantly, what happens to $\mathfrak{L}(M)$ in this passage. 
\begin{lem} \label{lem Lagrangian is representable}
    The morphism $\mu: \mathfrak{L}(M) \to \mathfrak{M}(G)$ is representable.
\end{lem}
\begin{proof}
    Since $\mathfrak{L}(M)$ is constructed out of base change from the morphism 
    \begin{equation}\label{eqn nontwisted mapping stacks}
       \widetilde{\mu}: \mathrm{Map}(\Sigma, [M/G \times \Ggr]) \longrightarrow \mathrm{Map}(\Sigma, [\g/G \times \Ggr])
    \end{equation}
    we may reduce to showing that $\widetilde{\mu}$ is representable. Since the latter is a mapping stack with target the representable morphism $[M/G \times \Ggr] \to [\g/G \times \Ggr]$, its representability is also immediate. 
\end{proof}

\begin{cor}
    $\mathfrak{L}(M)|_{\mathfrak{M}(G)^{\mathrm{st}}}$ is a (possibly derived) Deligne--Mumford stack.
\end{cor}
\begin{proof}
    Since $\mu$ is representable, we must have that the induced morphism on automorphism groups for any object in $\mathfrak{L}(M)|_{\mathfrak{M}(G)^{\mathrm{st}}}$ is injective. Indeed, if there existed some $x \in \mathfrak{L}(M)|_{\mathfrak{M}(G)^{\mathrm{st}}}$ for which $\mu_x: \mathrm{Aut}_{\mathfrak{L}(M)|_{\mathfrak{M}(G)^{\mathrm{st}}}}(x) \to \mathrm{Aut}_{\mathfrak{M}(G)^{\mathrm{st}}}(x) = Z_G$ has kernel, then we may consider the subgroup $\mathrm{Im}(\mu_x) \subset Z_G$ and pullback along the morphism $B \, \mathrm{Im}(\mu_x) \to \mathfrak{M}(G)^{\mathrm{st}}$ at $\mu(x)$ to observe that the fiber is not representable. In particular, since $Z_G$ is a finite group, we see that $\mathfrak{L}(M)|_{\mathfrak{M}(G)^{\mathrm{st}}}$ is a (possibly derived) Deligne--Mumford stack. 
    
\end{proof}

From now on, we restrict to those Hamiltonian actions $M$ of the second kind in Section \ref{subsect two Hamiltonian actions}, i.e., $\mathfrak{L}(M)$ is the Cayley space associated to some magical $\mathfrak{sl}_2$-triple.

\begin{lem}
    Let $\mathfrak{L}(M)$ be one of Gaiotto's Lagrangians with $M$ of the second kind in Section \ref{subsect two Hamiltonian actions}. Then the restriction $\mathfrak{L}(M)|_{\mathfrak{M}(G)^{\mathrm{st}}}$ is a (classical) Deligne--Mumford stack. 
\end{lem}
\begin{proof}
    
    To detect classicality of $\mathfrak{L}(M)|_{\mathfrak{M}(G)^{\mathrm{st}}}$ we compute its tangent complex and show that it is supported in nonpositive degrees. Let $y = (E,s) \in \mathfrak{L}(M)$ be a point, where $E$ is the underlying $G$-bundle and $s$ is a $K^{1/2}$-twisted section of the fiber bundle $M_E$. The tangent complex to $\mathfrak{L}(M)$ at $y$ is computed by the complex
    $$
        \mathbf{T}_y\mathfrak{L}(M) = R\Gamma(\Sigma, [\g_E \overset{\mathrm{act}_s}{\longrightarrow} T_sM_E] \times^{\Ggr}K^{1/2})[1]
    $$
    where we used the description of $\mathfrak{L}(M)$ as a mapping stack to pull back the tangent complex of $[M/G]$, which is concentrated in degrees $[-1,0]$. In principle, the hypercohomology spectral sequence tells us that $H^1(\Sigma, T_sM_E)$ may contribute to degree 1 of $\mathbf{T}_y\mathfrak{L}(M)$, but we will show that this does not occur for the Cayley spaces when we restrict to $\mathfrak{M}(G)^{\mathrm{st}}$. 

    Recall that in the notation introduced in Section \ref{subsect two Hamiltonian actions}, we have
    $$
        M = \left(f+ \oplus_j \, V_{2m_j}\right) \times^C G.
    $$
    Write $\mathfrak{v} := \oplus_j \, V_{2m_j}$, and we may view $s$ as a section of the associated $\mathfrak{v}$-bundle to the principal $C$-bundle $E$. We verify from Table \ref{tab:representations_C} that $\mathfrak{v}$, as a $C$-representation, is a direct sum of a single nontrivial irreducible representation along with copies of the trivial representation:
    \begin{equation} \label{eqn v as a C rep}
    \mathfrak{v} \simeq \mathfrak{v}_0 \oplus \mathbf{1}^{\oplus r}. 
    \end{equation} 
    
    Returning to the computation of the tangent complex $\mathbf{T}_y\mathfrak{L}(M)$ using the decomposition \eqref{eqn v as a C rep}, we see that we can split the tangent complex at $y$ into two summands
    $$
        \mathbf{T}_y \mathfrak{L}(M) \simeq R\Gamma(\Sigma, [\mathfrak{c}_E \overset{\mathrm{act}_s}{\to} \mathfrak{v}_{0,E}] \times^{\Ggr} K^{1/2}) [1] \oplus R\Gamma(\Sigma, (\mathbf{1}^{\oplus r}) \times^{\Ggr} K^{1/2}).
    $$
    Since the curve $\Sigma$ has genus at least 2 and the $K^{1/2}$-twists on the second summand are of weight $\geq 4$, the second summand is in fact a nonderived vector space in degree 0. For the first summand, we observe first that $s$ gives rise to a $K^{m_c+1}$-twisted Higgs field $\phi$ on $E$ valued in $\mathfrak{v}_0$. By \eqref{eq gtilde is c plus v0}, we may regard $(E,s)$ as a $\tilde{G}^\RR$-Higgs bundle, and the identification is compatible with tangent complexes:
    \begin{equation} \label{eq tangent complex and Dolbeault cohomology}
        R\Gamma(\Sigma, [\mathfrak{c}_E \overset{\mathrm{act}_s}{\to} \mathfrak{v}_{0,E}] \times^{\Ggr} K^{1/2}) [1] \simeq R\Gamma_{\mathrm{Dol}}(\Sigma, (E,s)).
    \end{equation}
    Stability of the point $y$ implies the stability of $(E,s)$ as a $\tilde{G}^\RR$-Higgs bundle. By a standard Serre duality argument (see Proposition 6.8 of \cite{Cayley}, for instance) along with the fact that $m_c > 0$, we may conclude that \eqref{eq tangent complex and Dolbeault cohomology} vanishes in degree 1. 
\end{proof}

We are now ready to pass to moduli spaces. Write
$$
    \mathcal{M}(G) := \mathfrak{M}(G)^{\mathrm{st}, \mathrm{rig}}
$$
for the moduli space of stable $G$-Higgs bundles obtained by passing to the stable locus and rigidifying; it is well-known that $\mathcal{M}(G)$ is a holomorphic symplectic manifold. The symplectic structures on $\mathfrak{M}(G)$ and $\mathcal{M}(G)$ are compatible in the sense that the rigidification morphism $\mathfrak{M}(G)^{\mathrm{st}} \to \mathcal{M}(G)$ is a $Z_G$-gerbe, so it induces an isomorphism on tangent complexes equipped with their respective symplectic pairings.

Base changing the morphism $\mu:\mathfrak{L}(M)|_{\mathfrak{M}(G)^{\mathrm{st}}} \to \mathfrak{M}(G)^{\mathrm{st}}$ along a splitting $\mathcal{M}(G) \to \mathfrak{M}(G)^{\mathrm{st}}$ of the $Z_G$-gerbe (which is possible upon picking an auxiliary base point on the curve), we obtain a morphism
$$
    \overline{\mu}:L(M) := \mathfrak{L}(M)|_{\mathfrak{M}(G)^{\mathrm{st}}} \times_{\mathfrak{M}(G)^{\mathrm{st}}} \mathcal{M}(G) \longrightarrow \mathcal{M}(G).
$$
Note that $L(M)$ is a priori a derived scheme, since $\mu$ is representable. By the previous Lemma, however, we can conclude the following.
\begin{cor} \label{cor classicality of Cayley space}
    $L(M)$ is a classical scheme and $\overline{\mu}$ is a quasi-finite morphism onto a (possibly singular) Lagrangian submanifold of $\mathcal{M}(G)$.
\end{cor}

While we cannot guarantee that the tangent space of $L(M)$ has constant dimension, we may conclude that $L(M)$ contains an open dense subset whose tangent spaces map onto a Lagrangian subspace of the tangent space of $\mathcal{M}(G)$. 

\subsection{Classicality of the Cayley morphism}

As a consequence of Theorem \ref{thm main 2} and the classicality of the Cayley space, we see that the moduli stack of $G^\RR_\rho$-Higgs bundles, for a magical $\rho$, is also essentially classical on the image of the Cayley morphism, as was stated in Theorem \ref{thm main 4}. To formulate this consequence properly, recall the Hamiltonian spaces $M_\rho$ and $M_\rho'$ for $\rho$ a magical $\mathfrak{sl}_2$ triple, in the notation of Section \ref{subsect two Hamiltonian actions}. Write $\mathcal{M}(G_\rho^\RR)$ for the moduli space of stable $G_\rho^\RR$-Higgs bundles.

\begin{cor}\label{cor classicality of Cayley morphism}
    Let $\mathcal{M}(G_\rho^\RR)^\circ \subset \mathcal{M}(G_\rho^\RR)$ be the image of $L(M_\rho')$ under the Cayley morphism. Then $\mathcal{M}(G^\RR_\rho)^\circ$ is a union of connected components, each of which maps quasi-finitely onto (possibly singular) Lagrangian submanifolds of $\mathcal{M}(G)$.
\end{cor}
\begin{proof}
    The classicality of $\mathfrak{L}(M'_\rho)|_{\mathfrak{M}(G)^{\mathrm{st}}}$, the representability of $\mathfrak{L}(M_\rho) \to \mathfrak{M}(G)$, and Theorem \ref{thm main 2}, together imply that the schematic Cayley morphism $L(M_\rho') = \mathrm{Cay}_\rho \to \mathcal{M}_{G_\rho^\RR}$ is a morphism of schemes inducing isomorphisms on tangent spaces. Together with Theorem \ref{thm main 3} we see that $L(M_\rho') \to \mathcal{M}(G_\rho^\RR)$ is an isomorphism onto connected components which we called $\mathcal{M}(G_\rho^\RR)^\circ$. By Corollary \ref{cor classicality of Cayley space}, we deduce that $\mathcal{M}(G_\rho^\RR)^\circ$ maps quasi-finitely onto Lagrangian submanifolds of $\mathcal{M}(G)$. 
\end{proof}

It is interesting to note that if we consider $\mathfrak{L}(M_\rho)$ itself, i.e., the moduli stack of $G_\rho^\RR$-Higgs bundles for a magical triple $\rho$, it is not \textit{a priori} clear that $\mathfrak{L}(M_\rho)$ does not have non-classical components. One can observe immediately, as a consequence of the Lagrangianity of $\mathfrak{L}(M_\rho) \to \mathfrak{M}(G)$ that the following criterion holds. 
\begin{lem} \label{lem classical implies Lagrangian submanifold}
    Let $M$ be the Hamiltonian $G$-action associated to a real form $G^\RR$ as in Section \ref{subsect two Hamiltonian actions}. of The Lagrangian $\mathfrak{L}(M)|_{\mathfrak{M}(G)^{\mathrm{st}}} \to \mathfrak{M}(G)^{\mathrm{st}}$ is classical if and only if $\mathcal{M}(G^\RR)$ is quasi-finite over (possibly singular) Lagrangian submanifolds of $\mathcal{M}(G)$.
\end{lem}
\begin{proof}
    We have seen that $\mathfrak{L}(M)|_{\mathfrak{M}(G)^{\mathrm{st}}}$ is a derived Deligne--Mumford stack with tangent complex concentrated in degrees $[0,1]$. The classicality of $\mathfrak{L}(M)|_{\mathfrak{M}(G)^{\mathrm{st}}}$ is equivalent to the condition that the tangent complex is concentrated in degree 0, which is in turn equivalent to the tangent spaces of $\mathcal{M}(G_\rho^\RR)$ being mapped to a Lagrangian subspace of the tangent spaces of $\mathcal{M}(G)$.
\end{proof}

\subsection{Table of $C$-representations}

By applying $\ad_f$ on each summand of the highest weight spaces $V_{2m_j}$ the appropriate number of times, we arrive at a $C$-equivariant isomorphism 

\begin{equation}
    \mathfrak{v} := \oplus_j V_{2m_j} \overset{\sim}{\longrightarrow} \mathfrak{g}_0 / \mathfrak{c}.
\end{equation} 

Recall the four families of Lie algebras $\mathfrak{g}$ introduced in Theorem~\ref{thm:classification-canonical-real-forms} which admit magical triples. We tabulate the $C$-representations that appear in $\mathfrak{v}$ in the following list, denoting by $\mathbf{1}$ the trivial 1-dimensional $C$-representation. Note that the $\mathbf{1}$-free part of $\mathfrak{v}$ is always an irreducible representation which we called $\mathfrak{v}_0$ in the proof of Corollary \ref{cor classicality of Cayley morphism}, in Equation \eqref{eqn v as a C rep}. 

\begin{center}
\begin{table}[h!]
\centering
\caption{Representations of the magical triple centralizer $C$}
\label{tab:representations_C}
\small 
\renewcommand{\arraystretch}{1.2} 
\begin{tabular}{l|l|c|c|c} 
\Xhline{3\arrayrulewidth} 
\textbf{Case} & \textbf{Type} & \textbf{$\mathfrak{g}_0$} & \textbf{$\mathfrak{c}$} & \textbf{$\mathfrak{v} \simeq \mathfrak{v}_0 \oplus \mathbf{1}^{\oplus r}$} \\
\Xhline{3\arrayrulewidth} 
(1) & rank $r$ & $\CC^r$ & $0$ & $\mathbf{1}^r$ \\
\Xhline{2\arrayrulewidth} 
(2) & $\text{A}_{2n-1}$ & $\mathfrak{sl}_n\CC \oplus \mathfrak{sl}_n\CC \oplus \CC$ & $\mathfrak{sl}_n\CC$ & $\mathrm{Ad}_C \oplus \mathbf{1}$ \\
\cline{2-5} 
    & $\text{B}_n$ & $\mathfrak{so}_{2n-1}\CC \oplus \CC$ & $\mathfrak{so}_{2n-2}\CC$ & $\mathrm{Std}_C \oplus \mathbf{1}$ \\
\cline{2-5}
    & $\text{C}_n$ & $\mathfrak{sl}_n\CC \oplus \CC$ & $\mathfrak{so}_n\CC$ & $V_{(n^2+n-2)/2} \oplus \mathbf{1}$ \\
\cline{2-5}
    & $\text{D}_n$ & $\mathfrak{so}_{2n-2}\CC \oplus \CC$ & $\mathfrak{so}_{2n-3}\CC$ & $\mathrm{Std}_C \oplus \mathbf{1}$ \\
\cline{2-5}
    & $\text{D}_{2n}$ & $\mathfrak{sl}_{2n}\CC \oplus \CC$ & $\mathfrak{sp}_{2n}\CC$ & $V_{2n^2-n-1} \oplus \mathbf{1}$ \\
\cline{2-5}
    & $\text{E}_7$ & $\mathfrak{e}_6 \oplus \CC$ & $\mathfrak{f}_4$ & $V_{26} \oplus \mathbf{1}$ \\
\Xhline{2\arrayrulewidth} 
(3) & $\text{B}, \text{C}$ & $\mathfrak{so}_{N-2p+2}\CC \oplus \CC^{p-1}$ & $\mathfrak{so}_{N-2p+1}\CC$ & $\mathrm{Std}_C \oplus \mathbf{1}^{p-1}$ \\
\Xhline{2\arrayrulewidth} 
(4) & $\text{E}_6$ & $\mathfrak{sl}_3\CC \oplus \mathfrak{sl}_3\CC \oplus \CC^2$ & $\mathfrak{sl}_3\CC$ & $\mathrm{Ad}_C \oplus \mathbf{1}^2$ \\
\cline{2-5}
    & $\text{E}_7$ & $\mathfrak{sl}_6\CC \oplus \CC^2$ & $\mathfrak{sp}_6\CC$ & $V_{14} \oplus \mathbf{1}^2$ \\
\cline{2-5}
    & $\text{E}_8$ & $\mathfrak{e}_6 \oplus \CC^2$ & $\mathfrak{f}_4$ & $V_{26} \oplus \mathbf{1}^2$ \\
\cline{2-5}
    & $\text{F}_4$ & $\mathfrak{sl}_3\CC \oplus \CC^2$ & $\mathfrak{so}_3\CC$ & $\mathrm{Std}_C \oplus \mathbf{1}^2$ \\
\Xhline{3\arrayrulewidth} 
\end{tabular}
\end{table}
\end{center}
The cases correspond to the cases in Theorem \ref{thm:classification-canonical-real-forms}: (1) split, (2) Hermitian of tube type, (3) special orthogonal, and (4) exceptional. We have used $V_k$ for an integer $k$ to denote the unique irreducible representation of dimension $k$ for each type, and the multiplicity $r$ of the trivial representation is exactly $r(\rho)$ in the statement of Theorem \ref{thm GCC}.

\bibliographystyle{my-amsalpha}
\bibliography{v9}
\end{document}